\documentclass[british,a4paper,10pt,twoside]{article}
\usepackage[T1]{fontenc}
\usepackage[latin9]{inputenc}
\usepackage{url}
\usepackage{amsmath}
\usepackage{amsthm}
\usepackage{amssymb}
\usepackage[numbers,square,sort&compress,comma,numbers]{natbib}
\usepackage{xargs}[2008/03/08]

\makeatletter
\numberwithin{equation}{section}
\numberwithin{figure}{section}
\theoremstyle{plain}
\newtheorem{thm}{\protect\theoremname}[section]
\theoremstyle{remark}
\newtheorem{rem}[thm]{\protect\remarkname}
\theoremstyle{plain}
\newtheorem{cor}[thm]{\protect\corollaryname}
\theoremstyle{plain}
\newtheorem{lem}[thm]{\protect\lemmaname}
\theoremstyle{plain}
\newtheorem{prop}[thm]{\protect\propositionname}

\@ifundefined{date}{}{\date{}}
\newcommand{\cdl}{c\`{a}dl\`{a}g }
\newcommand{\Ito}{It\^o }
\newcommand{\Itos}{It\^o's }

\allowdisplaybreaks

\usepackage[margin=1.5in, marginparwidth=1.2in]{geometry}

\usepackage{amsfonts}
\usepackage{latexsym}

\usepackage[colorlinks,linkcolor=blue,citecolor=blue,urlcolor=blue,pdfauthor={B. Gess, V. Konarovskyi}, pdftitle={Quantified central limit theorem for the simple symmetric exclusion process}, pdftex]{hyperref}
\hypersetup{
  colorlinks=false,
  pdfborder={0 0 0}
  }
\usepackage{url}
\usepackage{xcolor}

\usepackage{tocloft}
\usepackage{fancyhdr}

\usepackage[toc,page]{appendix}

\fancypagestyle{plain}{

  \fancyhf{}
  \fancyfoot[L]{\footnotesize \today \, }%\texttt{version=\jobname}}
  \fancyfoot[C]{\thepage}
}

\renewenvironment{abstract}
{
\begin{center}
\begin{minipage}{.9\textwidth}\small\textbf{Abstract}\noindent
}
{
\end{minipage}
\end{center}
}

\newenvironment{keywords}
{
\begin{center}
\begin{minipage}{.9\textwidth}\small\textbf{Keywords}:\noindent
}
{
\end{minipage}
\end{center}
}

\def\subjclass#1{{\renewcommand{\thefootnote}{}%
\footnote{\emph{Mathematics Subject Classification (2020):} #1}}}
\def\keywords#1{{\renewcommand{\thefootnote}{}%
\footnote{\emph{Keywords:} #1}}}

\makeatother

\usepackage{babel}
\providecommand{\corollaryname}{Corollary}
\providecommand{\lemmaname}{Lemma}
\providecommand{\propositionname}{Proposition}
\providecommand{\remarkname}{Remark}
\providecommand{\theoremname}{Theorem}

\begin{document}
\global\long\def\norm#1{\|#1\|}%
 
\global\long\def\bnorm#1{\left\Vert #1\right\Vert }%
 
\global\long\def\of#1{(#1)}%
 
\global\long\def\eps{\varepsilon}%
 
\global\long\def\R{\mathbb{R}}%
 
\global\long\def\T{\mathbb{T}}%
 
\global\long\def\Z{\mathbb{Z}}%
 
\global\long\def\N{\mathbb{N}}%
 
\global\long\def\p{\mathbb{P}}%
 
\global\long\def\E{\mathbb{E}}%
 
\global\long\def\I{\mathbb{I}}%
 
\global\long\def\H{\mathrm{H}}%
 
\global\long\def\rE{\mathrm{E}}%
 
\global\long\def\Cf{\mathrm{C}}%
 
\global\long\def\cL{\mathcal{L}}%
 
\global\long\def\F{\mathcal{F}}%
 
\global\long\def\cN{\mathcal{N}}%
 
\global\long\def\cS{\mathcal{S}}%
 
\global\long\def\cH{\mathcal{H}}%
 
\global\long\def\rM{\mathcal{\mathrm{M}}}%
 
\global\long\def\rR{\mathcal{\mathrm{R}}}%
 
\global\long\def\inn#1{\langle#1\rangle}%
 
\global\long\def\binn#1{\left\langle #1\right\rangle }%
 
\global\long\def\lrarr{\leftrightarrow}%
 
\global\long\def\var{\mathcal{\mathrm{Var\,}}}%
 
\global\long\def\cK{\mathcal{K}}%
\global\long\def\cM{\mathcal{M}}%
\global\long\def\pr{\mathrm{pr}}%
 
\global\long\def\C{\mathbb{C}}%
\global\long\def\Leb{\mathrm{Leb}}%
\global\long\def\i{\mathfrak{i}}%
\global\long\def\ex{\mathrm{ex}}%
\global\long\def\divv{\mathrm{div}}%
\global\long\def\cG{\mathcal{G}}%
\global\long\def\Tr{\mathrm{Tr}}%
\newcommandx\MLOgH[2][usedefault, addprefix=\global, 1=, 2=]{\cL_{#1}(H_{#2}(\T^{d}))}%
\newcommandx\MLOH[2][usedefault, addprefix=\global, 1=, 2=]{\cL_{#1}^{sym}(H_{#2}(\T^{d});\R)}%
\global\long\def\cov{\mathrm{Cov}}%
\newcommandx\Hrho[1][usedefault, addprefix=\global, 1=]{H_{#1}(\T^{d};[0,1])}%
\newcommandx\Hrhoint[1][usedefault, addprefix=\global, 1=]{H_{#1}(\T^{d};(0,1))}%
\global\long\def\supp{\mathrm{supp}}%
\global\long\def\Td{\T^{d}}%
 
\global\long\def\Tnd{\T_{n}^{d}}%
\global\long\def\Znd{\mathbb{Z}_{n}^{d}}%
 
\global\long\def\EEP{\{0,1\}^{\Tnd}}%
\global\long\def\EHE{[0,1]^{\Tnd}}%
\global\long\def\L{L_{2}(\T^{d})}%
\global\long\def\Ln{L_{2}(\Tnd)}%
 
\global\long\def\id{\mathrm{id}}%
\newcommandx\Cd[1][usedefault, addprefix=\global, 1=]{\mathrm{C}^{#1}(\T^{d})}%
\renewcommandx\H[1][usedefault, addprefix=\global, 1=J]{H_{#1}}%
 \newcommandx\MLOg[2][usedefault, addprefix=\global, 1=, 2=]{\cL_{#1}^{g}(#2)}%
\newcommandx\MLO[2][usedefault, addprefix=\global, 1=, 2=]{\cL_{#1}(#2)}%
\newcommandx\MLOHS[2][usedefault, addprefix=\global, 1=, 2=]{\cL_{#1}^{HS}(#2)}%
\newcommandx\CF[2][usedefault, addprefix=\global, 1=, 2=]{\Cf^{#1}(#2)}%
\newcommandx\CFl[2][usedefault, addprefix=\global, 1=, 2=]{\Cf_{l}^{#1}(#2)}%
\global\long\def\Her{\mathrm{Sym}}%
\newcommandx\normC[2][usedefault, addprefix=\global, 1=, 2=]{\|#2\|_{\mathrm{C}^{#1}}}%
\newcommandx\bnormC[2][usedefault, addprefix=\global, 1=, 2=]{\left\Vert #2\right\Vert _{\mathrm{C}^{#1}}}%
\newcommandx\normL[1][usedefault, addprefix=\global, 1=]{\|#1\|_{L_{2}}}%
\newcommandx\bnormL[1][usedefault, addprefix=\global, 1=]{\left\Vert #1\right\Vert _{L_{2}}}%
\newcommandx\normH[2][usedefault, addprefix=\global, 1=, 2=]{\|#2\|_{H_{#1}}}%
\newcommandx\bnormH[2][usedefault, addprefix=\global, 1=, 2=]{\left\Vert #2\right\Vert _{H_{#1}}}%
\newcommandx\normMLO[2][usedefault, addprefix=\global, 1=, 2=]{\|#2\|_{\cL_{#1}}}%
\newcommandx\bnormMLO[2][usedefault, addprefix=\global, 1=, 2=]{\left\Vert #2\right\Vert _{\cL_{#1}}}%
\global\long\def\tr{\mathrm{trace}}%
\global\long\def\Image{\mathrm{Im}}%
\global\long\def\law{\mathrm{Law}}%

\title{A quantitative central limit theorem for the simple symmetric exclusion
process}
\author{Benjamin Gess\thanks{Fakultät für Mathematik, Universität Bielefeld, 33615 Bielefeld, Germany.
E-mail: \protect\url{benjamin.gess@math.uni-bielefeld.de}} \thanks{Max Planck Institute for Mathematics in the Sciences, 04103 Leipzig,
Germany.} , Vitalii Konarovskyi\thanks{Faculty of mathematics, informatics and natural sciences, University
of Hamburg, 20146 Hamburg, Germany. E-mail: \protect\url{vitalii.konarovskyi@uni-hamburg.de}} \thanks{Institute of Mathematics of NAS of Ukraine, 01024 Kyiv, Ukraine.}}
\maketitle
\begin{abstract}
A quantitative central limit theorem for the simple symmetric exclusion
process (SSEP) on a $d$-dimensional discrete torus is proven. The
argument is based on a comparison of the generators of the density
fluctuation field of the SSEP and the generalized Ornstein-Uhlenbeck
process, as well as on an infinite-dimensional Berry-Essen bound for
the initial particle fluctuations. The obtained rate of convergence
is optimal. 
\end{abstract}


\tableofcontents{}

\section{Introduction}

We consider the simple symmetric exclusion process (SSEP) on the $d$-dimensional
discrete torus $\T_{n}^{d}:=\left\{ \frac{2\pi}{2n+1}k:\ k\in\left\{ -n,\ldots,n\right\} ^{d}\right\} $.
This is a continuous time Markov process that describes the evolution
of particles located at points of $\Tnd$, where each side can contain
at most one particle. A particle at site $x\in\Tnd$ attempts to jump
to one of the nearest neighboring sides after an exponential waiting
time. If the target side is occupied, then the jump does not take
place.

As usual, the state space for SSEP is $\EEP$, where $\eta(x)=1$
provided the side $x$ is occupied by a particle, and $\eta(x)=0$
otherwise. The generator of the SSEP is defined by
\begin{align}
\cG_{n}^{EP}F(\eta): & =\frac{(2n+1)^{2}}{2}\sum_{j=1}^{d}\sum_{x\in\T_{n}}\left[F(\eta^{x\lrarr x+e_{j}})-F(\eta)\right],\quad\eta\in\EEP,\label{eq:generator_of_SSEP}
\end{align}
for each function $F:\EEP\to\R$, where 
\[
\eta^{x\lrarr y}(z)=\begin{cases}
\eta(z), & z\not=x,y,\\
\eta(y), & z=x,\\
\eta(x), & z=y,
\end{cases}\quad z\in\Tnd,
\]
and $e_{j}=e_{j}^{n}$ denote the canonical vectors of $\Tnd$. For
a function $\rho:\Tnd\to[0,1]$, we let $\nu_{\rho}^{n}$ be the product
measure on $\EEP$ with marginals given by $\nu_{\rho}^{n}\left\{ \eta(x)=1\right\} =\rho(x)$,
$x\in\Tnd$. Let $\eta^{n}=(\eta_{t}^{n})_{t\ge0}$ be the SSEP with
the initial distribution $\nu_{\rho_{0}^{n}}^{n}$, where $\rho_{0}^{n}:\Tnd\to[0,1]$
and the sequence $\rho_{0}^{n}$, $n\ge1$, converges to a profile
$\rho_{0}:\T\to[0,1]$ as $n\to\infty$. 

It is well-known \citep[Theorem 2.1]{Kipnis_Landim:1999} that the
hydrodynamic limit of $\eta^{n}$, $n\ge1$, is given by the solution
to the heat equation
\begin{equation}
d\rho_{t}^{\infty}=2\pi^{2}\Delta\rho_{t}^{\infty}dt\label{eq:heat_PDE}
\end{equation}
on $\Td$ starting from $\rho_{0}$.

By \citep{Ferrari_Presutti:1988,GaKiSp,Ravishankar:1992}, also a
central limit theorem (CLT) is known. Precisely, it is known that
the density fluctuation field 
\[
\zeta_{t}^{n}(x):=(2n+1)^{d/2}\left(\eta_{t}^{n}(x)-\rho_{t}^{n}(x)\right),\quad x\in\Tnd,
\]
with $\rho_{t}^{n}(x):=\E\eta_{t}^{n}(x)$ converges to the solution
of the linear SPDE
\begin{align}
d\zeta_{t}^{\infty} & =2\pi^{2}\Delta\zeta_{t}^{\infty}dt+2\pi\nabla\cdot\left(\sqrt{\rho_{t}^{\infty}(1-\rho_{t}^{\infty})}dW_{t}\right)\label{eq:SPDE_for_OU_process}
\end{align}
in the Sobolev space $\H[-I]$ for $I>\frac{d}{2}+1$ started from
$\zeta_{0}$, where $(dW_{t})_{t\ge0}$ is a $d$-dimensional space-time
white noise, and $\zeta_{0}$ is a centered Gaussian distribution
in $\H[-I]$ with variance $\E\left[\inn{\zeta_{0},\varphi}^{2}\right]=\inn{\rho_{0}(1-\rho_{0})\varphi,\varphi}$
for smooth functions $\varphi$ on $\Td$.

By \citep[Theorem A.1]{Jara:2006}, the discretization error of the
heat equation $\|\E\eta^{n}-\rho^{\infty}\|_{\infty}$ behaves like
$(2n+1)^{-2}$. Therefore, informally, the CLT corresponds to the
expansion 
\begin{equation}
\eta_{t}^{n}(x)=\rho_{t}^{\infty}(x)+(2n+1)^{-d/2}\zeta_{t}^{\infty}(x)+(2n+1)^{-(d/2\wedge1)}o(1).\label{eq:expansion}
\end{equation}
Since the proof given in \citep{Ravishankar:1992} proceeds via a
compactness argument, the martingale central limit theorem, and the
Holley and Stroock theory \citep{Holley_Stroock:1979,Holley_Stroock:1978},
it does not allow the derivation of a quantitative convergence estimate
in the central limit theorem, nor in (\ref{eq:expansion}). This open
problem is solved in the present work, with an optimal rate of convergence.
It appears that this is the first result proving a quantitative central
limit theorem in the context of a non-equilibrium particle system\footnote{In contrast, in the setting of weakly interacting particle systems,
related high order expansions have been obtained in \citep{Chassagneux_Szpruch:2022}.}.

The proof developed in this work is instead based on the formula
\begin{equation}
\E F(\hat{\rho}_{t}^{n},\hat{\zeta}_{t}^{n})-\E F(\rho_{t}^{\infty,n},\zeta_{t}^{\infty,n})=\int_{0}^{t}\E\left[\left(\cG^{FF}-\cG^{OU}\right)P_{t-s}^{OU}F(\hat{\rho}_{s}^{n},\hat{\zeta}_{s}^{n})\right]ds,\label{eq:the_main_comparison_of_generators-2}
\end{equation}
see e.g. \citep[Lemma 1.2.5]{Ethier:1986}, which allows to deduce
estimates on the difference of the semigroups $(P_{t}^{FF})_{t\ge0}$
and $(P_{t}^{OU})_{t\ge0}$ associated with the Markov processes $(\rho^{n},\zeta^{n})$
and $(\rho^{\infty},\zeta^{\infty})$, from the difference of their
generators $\cG^{FF}$ and $\cG^{OU}$. Here, $\hat{f}=\ex_{n}f$
denotes smooth interpolation and $(\rho^{\infty,n},\zeta^{\infty,n})$
is a solution to (\ref{eq:heat_PDE}) and (\ref{eq:SPDE_for_OU_process})
started from the initial particle configuration $(\hat{\rho}_{0}^{n},\hat{\zeta}_{0}^{n})$. 

The estimation of the right hand side of (\ref{eq:the_main_comparison_of_generators-2}),
however, leads to several challenges: Firstly, the difference between
generators can be estimated only on sufficiently regular functions
$U=P_{t-s}^{OU}F$. Moreover, the obtained errors depend on higher-order
derivatives of $U$, the norms of $\hat{\rho}_{s}^{n}$, $\hat{\zeta}_{s}^{n}$
in corresponding Sobolev spaces and the expression $B(\zeta_{s}^{n}):=[\ex_{n}(\zeta_{s}^{n}\tau\zeta_{s}^{n})]^{2}$
for the shift operator $\tau$ on $\Tnd$. The differentiability of
$P_{t}^{OU}F$ is a non-trivial problem because the diffusion coefficient
$f(\rho^{\infty})=\sqrt{\rho^{\infty}(1-\rho^{\infty})}$ in (\ref{eq:SPDE_for_OU_process})
is not differentiable. Therefore, the standard approach to the preservation
of regularity of infinite-dimensional Kolmogorov equations, by proving
the regularity on the level of the corresponding SPDE, cannot be applied.
This issue is resolved in this work by a more careful infinite-dimensional
analysis based on the fact that the process $\zeta^{\infty}$ is Gaussian.
A second important ingredient to this part of the proof is a careful
choice of the extension operator $\ex_{n}$ in order to guarantee
the differentiability of $U$ at points $(\ex_{n}\rho,\ex_{n}\zeta)$
appearing in (\ref{eq:the_main_comparison_of_generators-2}), and
in order to quantitatively control discretization errors (lattice
effects) and interpolation errors, see e.g.~Proposition \ref{prop:expansion_of_cG_FEP}. 

Secondly, the control of the expectation of error terms requires additional
path properties of the SSEP compared to the proof of the non-quantified
CLT in \citep{Ravishankar:1992}. For instance, the bound of $\E[B(\zeta_{s}^{n})]$
fundamentally relies on the estimation of the four-point correlation
function $\E[\prod_{i=1}^{4}(\eta_{s}^{n}(x_{i})-\rho_{s}^{n}(x_{i}))]$,
while only the two point correlation function is used in \citep{Ravishankar:1992}. 

Thirdly, quantitative, optimal estimates for the initial fluctuations
\[
P_{t}^{OU}F(\hat{\rho}_{0}^{n},\hat{\zeta}_{0}^{n})-P_{t}^{OU}F(\rho_{0},\zeta_{0})=\E F(\rho_{t}^{\infty,n},\zeta_{t}^{\infty,n})-\E F(\rho_{t}^{\infty},\zeta_{t}^{\infty})
\]
are required. Compared to Stein's method in the finite-dimensional
context, see e.g., \citep{Meckes:2009,Reinert_Roellin:2009}, the
present situation is more challenging, since the dimension of $\zeta_{0}^{\infty,n}$
diverges with $n\to\infty$, for observables $F$ that are not assumed
to be of the specific form of partial sums. This difficulty is resolved
in the present work by carefully controling the constants appearing
in the application of Stein's method, thereby proving their independence
of the dimension. 

We refer the reader to Section \ref{subsec:Preliminaries} in the
appendix for the basic notation. As above, let $(\eta_{t}^{n})_{t\ge0}$
be the SSEP with the initial distribution $\nu_{\rho_{0}^{n}}^{n}$,
$(\rho_{t}^{n})_{t\ge0}$ its expectation field and $\left(\zeta_{t}^{n}\right)_{t\ge0}$
its density fluctuation field for each $n\ge1$. Let also $\left(\rho_{t}^{\infty}\right)_{t\ge0}$
be a solution to the heat equation (\ref{eq:heat_PDE}) started from
$\rho_{0}$, and $\left(\zeta_{t}^{\infty}\right)_{t\ge0}$ a solution
to (\ref{eq:SPDE_for_OU_process}) with the initial condition $\zeta_{0}$.
The following theorem is the main result of the paper. 

\pagebreak{}
\begin{thm}
\label{thm:main_result} Let $J>\frac{d}{2}\vee2$, $I>d+3$, $\tilde{J}>J+d+5$
and $F\in\Cf_{l,HS}^{1,3}(\H[J],\H[-I])$. Furthermore, assume that
$\rho_{0}\in\H[\tilde{J}]$ takes values in $[0,1]$ and $\rho_{0}^{n}$
is the restriction of $\rho_{0}$ to $\Tnd$ for each $n\ge1$. Then,
for each $T>0$ there exists a constant $C$ independent of $F$ and
$n$ such that 
\[
\sup_{t\in[0,T]}\left|\E F(\hat{\rho}_{t}^{n},\hat{\zeta}_{t}^{n})-\E F(\rho_{t}^{\infty},\zeta_{t}^{\infty})\right|\le\frac{C}{n^{\frac{d}{2}\wedge1}}\norm F_{\Cf_{l,HS}^{1,3}}.
\]
\end{thm}

\begin{rem}
The rate $\frac{1}{n^{\frac{d}{2}\wedge1}}$ cannot be improved in
the statement of Theorem \ref{thm:main_result}, since it also includes
the discretization error that equals $\frac{1}{n}$. 
\end{rem}

The following corollary directly follows from Theorem \ref{thm:main_result}. 
\begin{cor}
Under the assumptions of Theorem \ref{thm:main_result}, for each
$T>0$ and $m\ge1$ there exists a constant $C$ such that
\[
\sup_{t\in[0,T]}\left|\E f\left(\inn{\vec{\varphi},\zeta_{t}^{\infty}}\right)-\E f(\inn{\vec{\varphi},\zeta_{t}^{n}}_{n})\right|\le\frac{C}{n^{\frac{d}{2}\wedge1}}\norm f_{\Cf_{l}^{3}}\norm{\vec{\varphi}}_{\Cf^{\lceil I\rceil}}
\]
for all $n\ge1$, $f\in\Cf^{3}(\R^{m})$ and $\vec{\varphi}\in\left(\Cf^{\lceil I\rceil}(\Td)\right)^{m}$. 
\end{cor}

In \citep{Dirr_Fehrman:2020}, the SPDEs 
\[
d\eta_{t}^{n,\delta}=\partial_{xx}\eta_{t}^{n,\delta}dt+\frac{1}{\sqrt{n}}\partial_{x}\left(\sqrt{\eta_{t}^{n,\delta}(1-\eta_{t}^{n,\delta})}dW_{t}^{\delta}\right)
\]
have been analyzed as effective models for the one-dimensional SSEP,
where $(dW_{t}^{\delta})_{t\ge0}$ is a mollified $1$-dimensional
space-time white noise. In appropriate scaling regimes, it was concluded
that 
\begin{equation}
\E F(\hat{\eta}^{n})-\E F(\eta^{n,\delta})=o\left(n^{-\frac{1}{2}}\right)\label{eq:comparison_between_sol_to_nonlinear_equation_and_particle_system-1}
\end{equation}
which improves over the deterministic error 
\[
\E F(\hat{\eta}^{n})-\E F(\rho^{\infty})=O\left(n^{-\frac{1}{2}}\right).
\]
While one would expect (\ref{eq:comparison_between_sol_to_nonlinear_equation_and_particle_system-1})
to be of order $O(n^{-1})$, this was left open in \citep{Dirr_Fehrman:2020}
since a quantified CLT for a SSEP was missing, thus giving further
motivation for the questions addressed in the present work. 

\textbf{The work is organized as follows.} The basic notation and
some facts are collected and postponed to the appendix in Section
\ref{subsec:Preliminaries}. Section \ref{sec:Particle-system} is
devoted to an expansion of generators associated with the particle
system and an investigation of some path properties of the system.
In particular, the expansion of generators of the SSEP and its density
fluctuation field is obtained in Sections \ref{subsec:Expansion-of-generator}
and \ref{subsec:Fluctuations-of-SSEP}, respectively. Estimates of
the expectation of Sobolev norms of $\hat{\rho}_{s}^{n}$, $\hat{\zeta}_{s}^{n}$
and the control of $\E[B(\zeta_{s}^{n})]$ are obtained in Section
\ref{subsec:Some-properties-of}. The aim of Section \ref{sec:Generalized-Ornstein-Uhlenbeck}
is to show the regularity of the semigroup associated with the Ornstein-Uhlenbeck
process in both variables $\rho_{0}$ and $\zeta_{0}$. The differentiability
of $U$ in $\zeta_{0}$ straightforward follows from the linearity
of the SPDE (\ref{eq:SPDE_for_OU_process}). Therefore, the main focus
of this section is concentrated on the regularity of $U$ with respect
to $\rho_{0}$. The differentiability of the covariance operator of
$\zeta_{t}^{\infty}$ in $\rho_{0}$ is obtained in Section \ref{subsec:Differentiability-of-semigroup}.
Then, using a kind of the integration-by-parts formula for Gaussian
distributions, we get the differentiability of $U$. The Berry-Essen
bound on the rate of convergence of particle fluctuations $\hat{\zeta}_{0}^{n}$
to the Gaussian random distribution $\zeta_{0}^{\infty}$ in a corresponding
Sobolev space is obtained in Section \ref{sec:Berry-Esseen-bound}.
For this, we adapt the finite-dimensional approach, e.g, from \citep{Meckes:2009,Reinert_Roellin:2009},
to Sobolev spaces. The rest of the appendix is devoted to some properties
of $\pr_{n}$ and $\ex_{n}$ operators, multilinear operators on Sobolev
spaces and Frechet differentiable functions defined on Sobolev spaces
(see Sections \ref{subsec:pr_and_ex_operators}, \ref{subsec:Multilinear-operators-on-H}
and \ref{subsec:Differentiable-functions-on-H}, respectively).

\textbf{Comments on the literature. }For a comprehensive treatment
of equilibrium fluctuations, we refer to the monographs \citep{Kipnis_Landim:1999,KoLaOl2012}
and the detailed review of the literature contained therein.

In the case of gradient models and their perturbations, out-of-equilibrium
fluctuation results have been established in \citep{ChYa1992,DMFeLe1986,JaMe2018,PrSp1983},
including the central limit theorem for the weakly asymmetric simple
exclusion process in \citep{DMPrSc1989,DiGa1991}, and for the one-dimensional
symmetric zero-range process with constant jump rate in \citep{Ferrari_Presutti:1987}.
The central limit theorem for the symmetric simple exclusion process
was first established in \citep{GaKiSp,Ravishankar:1992}. Several
of these works build upon extensions of the equilibrium theories developed
by Holley and Stroock \citep{Holley_Stroock:1979,Holley_Stroock:1978},
as well as the Boltzmann--Gibbs principle \citep{BrRo1984}. A quantitative
form of the Boltzmann--Gibbs principle for independent random walkers,
and particle systems with duality has been obtained in \citep{Ayala_Carinci:2018}.
Additionally, non-equilibrium fluctuations for the boundary-driven
symmetric SSEP are discussed in \citep{LaMiOl2008}, the SSEP with
a slow bond in \citep{ErFrGoNeTa2020}, and for a tagged particle
in SSEP in \citep{Jara:2006}. In the recent contribution \citep{DiTeTi2024}
the joint fluctuations of current and occupation time of the one-dimensional
non-equilibrium simple symmetric exclusion process have been found.
We are not aware of any previous results providing quantitative central
limit theorems for out-of-equilibrium fluctuations in these contexts. 

For recent advances in the analysis of quantitative fluctuations for
non-gradient systems in equilibrium, see \citep{GuMoNi2024}. This
work also reviews a series of studies establishing the non-quantitative
equilibrium central limit theorem for several non-gradient systems.

Recent developments in the quantification of convergence in the law
of large numbers for both gradient and non-gradient systems are documented
in \citep{GiGuMo2022,MeMo2022} and the references cited therein.
Quantitative estimates of propagation of chaos for mean field systems
with singular kernels are provided in \citep{DrJaWa2023,JaWa2018}.
The study of fluctuations in this context has a longstanding history,
including works such as \citep{FeMe1997,It1983,Sz1984}, with recent
contributions in the setting of singular kernels found in \citep{WaZhZh2023}.
A deep analysis of central limit fluctuations around the Boltzmann
equation can be found in \citep{BoGaSaSe2023,Sp1981}.

Furthermore, fluctuation corrections of PDEs, leading to stochastic
PDEs, and their connection to higher-order fluctuation expansions
of particle systems and large deviations, have attracted significant
attention in recent years \citep{CoFi2023,CoFiInRa2023,Dirr_Fehrman:2020,DjKrPe2024,Gess:2019,Fehrman:arxiv:2021,Fe.Ge2023,Gess_SMFE:2022,Gess_SMF:2023}.

Since its development in \citep{St1972}, Stein's method for the derivation
of quantitative estimates on the distance to Gaussians has been an
active and fruitful field, an overview of which would go far beyond
the scope of this article. We restrict to mentioning a few points
of references, where further references to the theory may be found.
The main concepts of Stein's method is discussed in the survey article
\citep{Ross_2011}. Careful estimates for multivariate normal approximation
with Stein's method are obtained in \citep{Meckes:2009,Reinert_Roellin:2009}.
An early contribution extending Stein\textquoteright s method to the
context of approximations of processes, that is, to infinite dimension
is \citep{Ba1990}. See also \citep{CoDe2013} and the references
therein for subsequent generalizations. For applications of Stein\textquoteright s
method in the context of statistical mechanics, we refer to \citep{DeMu2023,EiLo2010}
and the references therein, where Berry-Esseen bounds for Curie-Weiss
and mean-field Ising models have been derived. Stein\textquoteright s
method in infinite dimension has been developed, for example, in \citep{Bourguin2020,Sh2011}
deriving Berry--Esséen type estimate for abstract Wiener measures
and in \citep{Ro2013} for high-dimensional settings. A significant
extension of Stein\textquoteright s method has been achieved by combination
with Malliavin calculus in a line of developments \citep{NoPe2009b,NoPe2009}
and the monograph \citep{NoPe2012}, which, in particular, allows
application going beyond observables taking the specific form of partial
sums. An extension of admissible functionals has been discussed in
\citep{BaRoZh2024}.

\section{Particle system\protect\label{sec:Particle-system}}

The goal of this section is to study some properties of the SSEP needed
for the proof of the main result. In particular, we expand the generator
of the density fluctuation field and show that the leading terms in
this expansion coincide with the generator of an Ornstein-Uhlenbeck
process. 

\subsection{Expansion of the generator of the SSEP\protect\label{subsec:Expansion-of-generator}}

We start from the expansion of the generator of the SSEP. Let $\eta_{t}^{n}=\left(\eta_{t}^{n}(x),\ x\in\T_{n}^{d}\right)$,
$t\ge0$, be the SSEP defined on the configuration space $\EEP\subset L_{2}(\T_{n}^{d})$.
Recall that it is a time continuous Markov Process whose generator
$\cG_{n}^{EP}$ is defined by (\ref{eq:generator_of_SSEP}). We extend
$(\eta_{t}^{n})_{t\ge0}$ to a $\Cf^{\infty}(\T^{d})$-valued process
by considering 
\[
\hat{\eta}_{t}^{n}=\ex_{n}\eta_{t}^{n},\quad t\ge0.
\]
According to (\ref{eq:ex_is_interpolation}), the restriction of $\hat{\eta}_{t}^{n}$
to the set $\T_{n}^{d}\subset\T^{d}$ coincides with $\eta_{t}^{n}$
for each $t\ge0$ and $n\in\N$. Next note that for each $J\in\R$
and $F\in\Cf(\H[J]),$ the process 
\[
M_{t}^{F}:=F(\hat{\eta}_{t}^{n})-F(\hat{\eta}_{0}^{n})-\int_{0}^{t}\cG_{n}^{EP}\left(F\circ\ex_{n}\right)(\eta_{s}^{n})ds,\quad t\ge0,
\]
is a right-continuous martingale with respect to the filtration $(\F_{t}^{\eta^{n}})_{t\ge0}$
generated by $\eta^{n}$. 

In the next statement, we derive an expansion of the generator $\cG_{n}^{EP}$
that will be used for an expansion of the generator of the density
fluctuation field later.
\begin{lem}
\label{lem:expansion_of_cG_EP}Let $I>\frac{d}{2}+1$. Then for each
$F\in\Cf_{l}^{3}(\H[-I])$ and $n\ge1$ 
\begin{align*}
\cG_{n}^{EP}\left(F\circ\ex_{n}\right)(\eta) & =2\pi^{2}\inn{\Delta_{n}\pr_{n}DF(\hat{\eta}),\eta}_{n}\\
 & +\frac{4\pi^{4}}{(2n+1)^{d+2}}\sum_{j=1}^{d}\binn{\Tr\left(\partial_{n,j}^{\otimes2}\pr_{n}^{\otimes2}D^{2}F(\hat{\eta})\right),\left[\partial_{n,j}\eta\right]^{2}}_{n}+R_{n}^{EP}(\eta),
\end{align*}
for all $\eta\in\EEP,$ where 
\[
\left|R_{n}^{EP}(\eta)\right|\le\frac{C_{I}}{(2n+1)^{2d+1}}\norm{D^{3}F}_{\Cf},\quad\eta\in\EEP.
\]
\end{lem}

\begin{proof}
To prove the lemma, we use the Taylor formula (\ref{eq:taylors_formula}).
For $\eta\in\EEP$ we get
\begin{align*}
\cG_{n}^{EP}\left(F\circ\ex_{n}\right)(\eta) & =\frac{(2n+1)^{2}}{2}\sum_{j=1}^{d}\sum_{x\in\T_{n}^{d}}\left(F(\hat{\eta}^{x\lrarr x+e_{j}})-F(\hat{\eta})\right)\\
 & =\frac{(2n+1)^{2}}{2}\sum_{j=1}^{d}\sum_{x\in\T_{n}^{d}}DF(\hat{\eta})\left[\hat{\eta}^{x\leftrightarrow x+e_{j}}-\hat{\eta}\right]\\
 & +\frac{(2n+1)^{2}}{4}\sum_{j=1}^{d}\sum_{x\in\T_{n}^{d}}D^{2}F(\hat{\eta})\left[\left(\hat{\eta}^{x\leftrightarrow x+e_{j}}-\hat{\eta}\right)^{\times2}\right]\\
 & +\frac{(2n+1)^{2}}{2}\sum_{j=1}^{d}\sum_{x\in\T_{n}^{d}}R_{j}(x,\hat{\eta})=:I_{1}+I_{2}+R^{EP},
\end{align*}
where $\hat{\eta}^{x\leftrightarrow x+e_{j}}=\ex_{n}\eta^{x\leftrightarrow x+e_{j}}$
and 
\[
\left|R_{j}(x,\hat{\eta})\right|\le\frac{\norm{D^{3}F}_{\Cf}}{3!}\norm{\hat{\eta}^{x\leftrightarrow x+e_{j}}-\hat{\eta}}_{H_{-I}}^{3}.
\]

Using the equality (\ref{eq:connection_between_ex_and_varsigma}),
we first estimate the expression
\begin{align*}
\norm{\hat{\eta}^{x\leftrightarrow x+e_{j}}-\hat{\eta}}_{H_{-I}}^{2} & =\sum_{k\in\Z^{d}}\frac{1}{\left(1+|k|^{2}\right)^{I}}\left|\inn{\hat{\eta}^{x\leftrightarrow x+e_{j}}-\hat{\eta},\varsigma_{k}}\right|^{2}\\
 & =\sum_{k\in\Z_{n}^{d}}\frac{1}{\left(1+|k|^{2}\right)^{I}}\left|\inn{\eta^{x\leftrightarrow x+e_{j}}-\eta,\varsigma_{k}}_{n}\right|^{2}
\end{align*}
for $x\in\T_{n}$ and $j\in[d]$. Note that for each $\varphi:\T_{n}^{d}\to\C$
\begin{align}
\inn{\varphi,\eta^{x\lrarr x+e_{j}}}_{n}-\inn{\varphi,\eta}_{n} & =\frac{1}{(2n+1)^{d}}\sum_{z\in\Tnd}\varphi(z)\eta^{x\lrarr x+e_{j}}(z)\nonumber \\
 & -\frac{1}{(2n+1)^{d}}\sum_{z\in\T_{n}^{d}}\varphi(z)\eta(z)\nonumber \\
 & =\frac{1}{(2n+1)^{d}}\left[\eta(x+e_{j})-\eta(x)\right]\varphi(x)\label{eq:increments_for_xi_lrarrow}\\
 & +\frac{1}{(2n+1)^{d}}\left[\eta(x)-\eta(x+e_{j})\right]\varphi(x+e_{j})\nonumber \\
 & =\frac{1}{(2n+1)^{d}}\left[\eta(x)-\eta(x+e_{j})\right]\left[\varphi(x+e_{j})-\varphi(x)\right]\nonumber \\
 & =-\frac{4\pi^{2}}{(2n+1)^{d+2}}\partial_{n,j}\eta(x)\partial_{n,j}\varphi(x).\nonumber 
\end{align}
Thus, using (\ref{eq:eignevalues_for_discrete_basis}) and Lemma \ref{lem:estimate_of_discrete_eigenvalues},
we estimate
\begin{align*}
\left|\inn{\eta^{x\leftrightarrow x+e_{j}}-\eta,\varsigma_{k}}_{n}\right| & =\frac{4\pi^{2}}{(2n+1)^{d+2}}\left|\partial_{n,j}\eta(x)\right|\left|\partial_{n,j}\varsigma_{k}(x)\right|\\
 & =\frac{4\pi^{2}}{(2n+1)^{d+2}}\left|\partial_{n,j}\eta(x)\right|\left|\mu_{k,j}^{n}\right|\le\frac{8\pi^{2}|k_{j}|}{(2n+1)^{d+1}}
\end{align*}
and, consequently,
\[
\norm{\hat{\eta}^{x\leftrightarrow x+e_{j}}-\hat{\eta}}_{H_{-I}}^{2}\le\frac{64\pi^{4}}{(2n+1)^{2d+2}}\sum_{k\in\Z_{n}^{d}}\frac{|k_{j}|^{2}}{\left(1+|k|^{2}\right)^{I}}\le\frac{64\pi^{4}C_{I}}{(2n+1)^{2d+2}}
\]
due to $I>\frac{d}{2}+1.$ This implies the inequality
\[
\left|R^{EP}(\eta)\right|\le\frac{(2n+1)^{2}}{2d}\norm{D^{3}F}_{\Cf}\sum_{j=1}^{d}\sum_{x\in\T_{n}}\frac{C_{I}}{(2n+1)^{3d+3}}\le\frac{C_{I}}{(2n+1)^{2d+1}}\norm{D^{3}F}_{\Cf}.
\]

In order to rewrite $I_{1}$ we use the fact that the derivative $DF(\hat{\eta})$
belongs to the dual space of $\H[-I]$. Hence, $DF(\hat{\eta})\in\H[I]$
and 
\begin{align*}
DF(\hat{\eta})\left[\hat{\eta}^{x\leftrightarrow x+e_{j}}-\hat{\eta}\right] & =\inn{DF(\hat{\eta}),\hat{\eta}^{x\leftrightarrow x+e_{j}}-\hat{\eta}}=\binn{\pr_{n}DF(\hat{\eta}),\eta^{x\leftrightarrow x+e_{j}}-\eta}_{n}\\
 & =-\frac{4\pi^{2}}{(2n+1)^{d+2}}\partial_{n,j}\pr_{n}DF(\hat{\eta})(x)\partial_{n,j}\eta(x),
\end{align*}
according to (\ref{eq:connection_betwwen_ex_and_pr}) and (\ref{eq:increments_for_xi_lrarrow}).
This implies
\begin{align*}
I_{1} & =-\frac{(2n+1)^{2}}{2}\frac{4\pi^{2}}{(2n+1)^{d+2}}\sum_{j=1}^{d}\sum_{x\in\Tnd}\partial_{n,j}\pr_{n}DF(\hat{\eta})(x)\partial_{n,j}\eta(x)\\
 & =-2\pi^{2}\sum_{j=1}^{d}\inn{\partial_{n,j}\pr_{n}DF(\hat{\eta}),\partial_{n,j}\eta}_{n}=2\pi^{2}\inn{\Delta_{n}\pr_{n}DF(\hat{\eta}),\eta}_{n},
\end{align*}
where we used the discrete integration by parts formula (\ref{eq:integration_by_parts_in_discrete_space}).

Since $D^{2}F(\hat{\eta})\in\cL_{2}(\H[-I])$, the equality (\ref{eq:connection_between_ex_and_pr_extension})
yields 
\[
D^{2}F(\hat{\eta})\left[\left(\hat{\eta}^{x\leftrightarrow x+e_{j}}-\hat{\eta}\right)^{\times2}\right]=\binn{\pr_{n}^{\otimes2}D^{2}F(\hat{\eta}),\left(\eta^{x\leftrightarrow x+e_{j}}-\eta\right)^{\otimes2}}_{n}.
\]
Similarly to the computation in (\ref{eq:increments_for_xi_lrarrow}),
we get that the right hand side in the expression above equals
\[
\frac{16\pi^{4}}{(2n+1)^{2d+4}}\partial_{n,j}^{\otimes2}\pr_{n}^{\otimes2}D^{2}F(\hat{\eta})(x,x)\left(\partial_{n,j}\eta(x)\right)^{2}.
\]
Hence, 
\begin{align*}
I_{2} & =\frac{16\pi^{4}(2n+1)^{2}}{4(2n+1)^{2d+4}}\sum_{j=1}^{d}\sum_{x\in\T_{n}^{d}}\partial_{n,j}^{\otimes2}\pr_{n}^{\otimes2}D^{2}F(\hat{\eta})(x,x)\left(\partial_{n,j}\eta(x)\right)^{2}\\
 & =\frac{4\pi^{4}}{(2n+1)^{d+2}}\sum_{j=1}^{d}\binn{\Tr\left(\partial_{n,j}^{\otimes2}\pr_{n}^{\otimes2}D^{2}F(\hat{\eta})\right),\left(\partial_{n,j}\eta\right)^{2}}_{n}.
\end{align*}
This completes the proof of the lemma.
\end{proof}

\subsection{Density fluctuation field for the SSEP and its generator\protect\label{subsec:Fluctuations-of-SSEP}}

The aim of this section is to consider the density fluctuation field
\[
\zeta_{t}^{n}(x)=(2n+1)^{d/2}(\eta_{t}^{n}(x)-\rho_{t}^{n}(x)),\ x\in\T_{n}^{d},\quad t\ge0,
\]
for the SSEP and obtain an expansion of its generator. It is easy
to see that the process
\[
\rho_{t}^{n}(x)=\E\eta_{t}^{n}(x),\quad x\in\T_{n}^{d},\ \ t\ge0,
\]
is a unique solution to the discrete heat equation
\begin{equation}
\rho_{t}^{n}(x)=\rho_{0}^{n}(x)+2\pi^{2}\int_{0}^{t}\Delta_{n}\rho_{s}^{n}(x)ds,\quad x\in\T_{n}^{d},\ \ t\ge0,\label{eq:descrete_heat_equation}
\end{equation}
with $\rho_{0}^{n}(x)=\E\eta_{0}^{n}(x)\in[0,1]$, $x\in\T_{n}^{d}$.
Moreover, $\rho_{t}^{n}\in\EHE\subset\Ln$ for all $t\ge0$. Using
the chain rule (see, e.g. Theorem \citep[Theorem 2.2.1]{Cartan:1971})
and the discrete integration-by-parts formula, we get for each $F\in\Cf^{1}(\Ln)$
\begin{align*}
F(\rho_{t}^{n}) & =F(\rho_{0}^{n})+2\pi^{2}\int_{0}^{t}\binn{DF(\rho_{s}^{n}),\Delta_{n}\rho_{s}^{n}}_{n}ds\\
 & =F(\rho_{0}^{n})+2\pi^{2}\int_{0}^{t}\binn{\Delta_{n}DF(\rho_{s}^{n}),\rho_{s}^{n}}_{n}ds,\quad t\ge0.
\end{align*}
In particular, this implies that $(\rho_{t}^{n},\eta_{t}^{n})$, $t\ge0$,
is a Markov process with generator 
\[
2\pi^{2}\binn{\Delta_{n}D_{1}F(\cdot,\eta)(\rho),\rho}_{n}+\left(\cG^{EP}F(\rho,\cdot)\right)(\eta).
\]
 Thus, the process $(\rho_{t}^{n},\zeta_{t}^{n}),$ $t\ge0$, is also
a Markov process and for each $F\in\Cf^{1}(\Ln^{2})$ 
\[
F(\rho_{t}^{n},\zeta_{t}^{n})-F(\rho_{0}^{n},\zeta_{0}^{n})-\int_{0}^{t}\cG^{FF}F(\rho_{s}^{n},\zeta_{s}^{n})ds,\quad t\ge0,
\]
is a martingale with respect to the filtration $(\F_{t}^{\zeta^{n}})_{t\ge0}$
generated by the process $\zeta^{n}$ that coincides with $(\F_{t}^{\eta^{n}})_{t\ge0}$.
Here 
\[
\cG^{FF}F(\rho,\zeta)=2\pi^{2}\binn{\Delta_{n}D_{1}G(\rho,\eta),\rho}_{n}+\cG^{EP}G(\rho,\cdot)(\eta),
\]
where $G(\rho,\eta):=F(\rho,\zeta)$ and $\eta=\rho+(2n+1)^{-d/2}\zeta$.

Similarly to the previous section, we extend $\rho_{t}^{n}$ and $\zeta_{t}^{n}$
to the domain $\Td$ by setting
\begin{equation}
\hat{\rho}_{t}^{n}:=\ex_{n}\rho_{t}^{n}\quad\text{and}\quad\hat{\zeta}_{t}^{n}:=\ex_{n}\zeta_{t}^{n}\label{eq:rho_n_and_eta_n}
\end{equation}
for all $t\ge0$ and consider $(\hat{\rho}_{t}^{n},\hat{\zeta}_{t}^{n})$,
$t\ge0$, as a process with values in $\H[J]\times\H[-I]$ for each
$I,J\in\R$. Since $(\hat{\rho}^{n},\hat{\zeta}^{n})$ is obtained
from the Markov process $(\rho^{n},\zeta^{n})$ using injective mappings,
it is a Markov process too. Furthermore, for each $F\in\Cf^{1}(\H[J]\times\H[-I])$
\begin{equation}
F(\hat{\rho}_{t}^{n},\hat{\zeta}_{t}^{n})-F(\hat{\rho}_{0}^{n},\hat{\zeta}_{0}^{n})-\int_{0}^{t}\cG^{FF}\hat{F}(\rho_{s}^{n},\zeta_{s}^{n})ds,\quad t\ge0,\label{eq:Itos_formula_for_eta_n}
\end{equation}
is a martingale with respect to $(\F_{t}^{\zeta^{n}})_{t\ge0}$, where
$\hat{F}(\rho,\zeta)=F(\ex_{n}\rho,\ex_{n}\zeta)$, $\rho,\zeta\in\Ln$,
and $\hat{F}\in\Cf^{1}(\Ln^{2})$, according to Lemma \ref{lem:differentiability_of_F_ex_n}.
Note that $\cG^{FF}\hat{F}(\rho_{s}^{n},\zeta_{s}^{n})$ can be rewritten
as $\cG^{FF}\hat{F}(\pr_{n}\hat{\rho}_{s}^{n},\pr_{n}\hat{\zeta}_{s}^{n})$
due to Lemma \ref{lem:continuity_of_pr_and_ex}. Consequently, setting
\[
\hat{\cG}^{FF}F(\rho,\zeta):=\cG^{FF}\hat{F}(\pr_{n}\rho,\pr_{n}\zeta)=\cG^{FF}\left(F\circ\ex_{n}^{\times2}\right)(\pr_{n}\rho,\pr_{n}\zeta),\quad\rho\in\H[J],\ \ \zeta\in\H[-I],
\]
where $\ex_{n}^{\times2}(f,g)=(\ex_{n}f,\ex_{n}g),$ $f,g\in\Ln$,
we conclude that for each $F\in\Cf^{1}(\H[J]\times H_{I})$
\begin{equation}
F(\hat{\rho}_{t}^{n},\hat{\zeta}_{t}^{n})-F(\hat{\rho}_{0}^{n},\hat{\zeta}_{0}^{n})-\int_{0}^{t}\hat{\cG}^{FF}F(\hat{\rho}_{s}^{n},\hat{\zeta}_{s}^{n})ds,\quad t\ge0,\label{eq:Itos_formula_for_eta_n-1}
\end{equation}
is a martingale with respect to $(\F_{t}^{\zeta^{n}})_{t\ge0}$. In
particular, the expectation of the martingale in (\ref{eq:Itos_formula_for_eta_n-1})
equals zero. 

We will need the following property for the case of time dependent
functions $F$.
\begin{lem}
\label{lem:Ito_formula_for_eta_t_dependence} Let $J,I\in\R$ and
$F\in\Cf^{1,1,1}([0,\infty),\H[J],H_{-I})$. Then, for all $t\ge0$
and $n\in\N$,
\[
\E F_{t}(\hat{\rho}_{t}^{n},\hat{\zeta}_{t}^{n})=\E F_{0}(\hat{\rho}_{0}^{n},\hat{\zeta}_{0}^{n})+\int_{0}^{t}\E\left[\partial F_{s}(\hat{\rho}_{s}^{n},\hat{\zeta}_{s}^{n})+\hat{\cG}^{FF}F_{s}(\hat{\rho}_{s}^{n},\hat{\zeta}_{s}^{n})\right]ds,\quad t\ge0.
\]
\end{lem}

\begin{proof}
Considering a partition $0=t_{0}<t_{1}<\ldots<t_{m}=t$, we get
\begin{align}
\E F_{t}(\hat{\rho}_{t}^{n},\hat{\zeta}_{t}^{n}) & -\E F_{0}(\hat{\rho}_{0}^{n},\hat{\zeta}_{0}^{n})=\sum_{k=1}^{m}\E\left[F_{t_{i}}(\hat{\rho}_{t_{i}}^{n},\hat{\zeta}_{t_{i}}^{n})-F_{t_{i-1}}(\hat{\rho}_{t_{i-1}}^{n},\hat{\zeta}_{t_{i-1}}^{n})\right]\nonumber \\
 & =\sum_{k=1}^{m}\E\left[F_{t_{i-1}}(\hat{\rho}_{t_{i}}^{n},\hat{\zeta}_{t_{i}}^{n})-F_{t_{i-1}}(\hat{\rho}_{t_{i-1}}^{n},\hat{\zeta}_{t_{i-1}}^{n})\right]\nonumber \\
 & +\sum_{k=1}^{m}\E\left[F_{t_{i}}(\hat{\rho}_{t_{i}}^{n},\hat{\zeta}_{t_{i}}^{n})-F_{t_{i-1}}(\hat{\rho}_{t_{i}}^{n},\hat{\zeta}_{t_{i}}^{n})\right]\label{eq:proof_of_itos_formula_discrete_case}\\
 & =\sum_{k=1}^{m}\E\left[\int_{t_{i-1}}^{t_{i}}\hat{\cG}^{FF}F_{t_{i-1}}(\hat{\rho}_{s}^{n},\hat{\zeta}_{s}^{n})ds\right]\nonumber \\
 & +\sum_{k=1}^{m}\E\left[\int_{t_{i-1}}^{t_{i}}\partial_{s}F_{s}(\hat{\rho}_{t_{i}}^{n},\hat{\zeta}_{t_{i}}^{n})ds\right],\nonumber 
\end{align}
where we used (\ref{eq:Itos_formula_for_eta_n-1}). Trivially, $E_{n}:=\ex_{n}(\EHE)\times\ex_{n}(n^{d/2}\EHE)$
is a compact subset of $\H[J]\times\H[-I]$ because it is closed,
bounded and finite-dimensional. Moreover, $(\hat{\rho}_{t}^{n},\hat{\zeta}_{t}^{n}),$
$t\ge0$, takes values in $E_{n}$. We also note that the functions
$(s,\rho,\zeta)\mapsto\hat{\cG}^{FF}F_{s}(\rho,\zeta)$ and $(s,\rho,\zeta)\mapsto\partial F_{s}(\rho,\zeta)$
are continuous and, thus, bounded on $[0,t]\times E_{n}$. Using the
right-continuity of $(\hat{\rho}_{t}^{n},\hat{\zeta}_{t}^{n}),$ $t\ge0$,
and the dominated convergence theorem, we conclude that the right
hand side of (\ref{eq:proof_of_itos_formula_discrete_case}) converges
to
\[
\int_{0}^{t}\E\left[\hat{\cG}^{FF}F_{s}(\hat{\rho}_{s}^{n},\hat{\zeta}_{s}^{n})+\partial_{s}F(\hat{\rho}_{s}^{n},\hat{\zeta}_{s}^{n})\right]ds,
\]
as the mesh of the partition goes to zero. This completes the proof
of the lemma.
\end{proof}
In the next statement, we get an expansion of $\hat{\cG}_{n}^{FF}F(\hat{\rho},\hat{\zeta})$
needed for its comparison with the generator of an Ornstein-Uhlenbeck
process. Let $\tau_{j}^{n}$ denote the shift operator on $\Tnd$
defined by $\tau_{j}^{n}f(x)=f(x+e_{j}^{n})$. 
\begin{prop}
\label{prop:expansion_of_cG_FEP}Let $J>2$, $I>d+2$, $\tilde{I}\ge0$,
$\lceil\tilde{I}\rceil+1+\frac{d}{2}<I$ and $F\in\Cf_{l,HS}^{1,3}(\H[J],\H[-I])$.
Then for each $n\ge1$ there exists a function $R_{n}:\EHE\times\R^{\Tnd}\to\R$
such that
\begin{align*}
\hat{\cG}_{n}^{FF}F(\hat{\rho},\hat{\zeta}) & =2\pi^{2}\binn{\Delta D_{1}F(\hat{\rho},\hat{\zeta}),\hat{\rho}}+2\pi^{2}\binn{\Delta D_{2}F(\hat{\rho},\hat{\zeta}),\hat{\zeta}}\\
 & +4\pi^{2}\sum_{j=1}^{d}\binn{\Tr\left(\partial_{j}^{\otimes2}D_{2}^{2}F(\hat{\rho},\hat{\zeta})\right),\hat{\rho}(1-\hat{\rho})}\\
 & +\frac{2\pi^{2}}{(2n+1)^{d}}\sum_{j=1}^{d}\binn{\Tr\left(\partial_{j}^{\otimes2}D_{2}^{2}F(\hat{\rho},\hat{\zeta})\right),\ex_{n}\left[\zeta\tau_{j}^{n}\zeta\right]}\\
 & +R_{n}^{FF}(\rho,\zeta),
\end{align*}
and 
\[
\left|R_{n}^{FF}(\rho,\zeta)\right|\le\frac{C_{J,I,\tilde{I}}}{n^{\frac{d}{2}\wedge1}}\norm F_{C_{l,HS}^{1,3}}\left(1+\norm{\hat{\rho}}_{\Cf^{\lceil d/2\rceil+4}}^{2}+\norm{\hat{\rho}}_{\Cf^{\lceil\tilde{I}\rceil}}\right)\left(1+\norm{\hat{\zeta}}_{H_{-I+2}}+\norm{\hat{\zeta}}_{\H[-\tilde{I}]}\right)
\]
for all $\rho\in\EHE$ and $\zeta=(2n+1)^{d/2}(\eta-\rho)$, $\eta\in\EEP$. 
\end{prop}

\begin{proof}
Take $\rho\in\EHE$, $\eta\in\EEP$ and $\zeta:=(2n+1)^{d/2}(\eta-\rho)$.
For $\hat{G}(\rho,\eta):=\hat{F}(\rho,(2n+1)^{d/2}(\eta-\rho))$,
where $\hat{F}(\rho,\zeta)=F(\hat{\rho},\hat{\zeta})$, we first rewrite
\begin{align*}
\binn{\Delta_{n}D_{1}\hat{G}(\rho,\eta),\rho}_{n} & =\binn{\Delta_{n}D_{1}\hat{F}(\rho,(2n+1)^{d/2}(\eta-\rho)),\rho}_{n}\\
 & -(2n+1)^{d/2}\binn{\Delta_{n}D_{2}\hat{F}(\rho,(2n+1)^{d/2}(\eta-\rho)),\rho}_{n}\\
 & =\binn{\Delta_{n}\pr_{n}D_{1}F(\hat{\rho},\hat{\zeta}),\rho}_{n}-(2n+1)^{d/2}\binn{\Delta_{n}\pr_{n}D_{2}F(\hat{\rho},\hat{\zeta}),\rho}_{n}
\end{align*}
due to Lemma \ref{lem:differentiability_of_F_ex_n}. Thus, using Lemma
\ref{lem:expansion_of_cG_EP}, we obtain 
\begin{align}
\cG_{n}^{FF}\hat{F}(\rho,\zeta) & =2\pi^{2}\binn{\Delta_{n}\pr_{n}D_{1}F(\hat{\rho},\hat{\zeta}),\rho}_{n}-2\pi^{2}(2n+1)^{d/2}\binn{\Delta_{n}\pr_{n}D_{2}F(\hat{\rho},\hat{\zeta}),\rho}_{n}\nonumber \\
 & +2\pi^{2}(2n+1)^{d/2}\inn{\Delta_{n}\pr_{n}D_{2}F(\hat{\rho},\hat{\zeta}),\eta}_{n}\nonumber \\
 & +\frac{4\pi^{4}}{(2n+1)^{2}}\sum_{j=1}^{d}\binn{\Tr\left(\partial_{n,j}^{\otimes2}\pr_{n}^{\otimes2}D_{2}^{2}F(\hat{\rho},\hat{\zeta})\right),\left(\partial_{n,j}\eta\right)^{2}}_{n}\label{eq:expansion_for_cG_FP}\\
 & +R_{n}^{EP}(\rho,\eta),\nonumber 
\end{align}
where the error term $R_{n}^{EP}$ satisfies
\[
\left|R_{n}^{EP}(\rho,\eta)\right|\le\frac{C_{I}}{(2n+1)^{d/2+1}}\norm{D_{2}^{3}F}_{\Cf}.
\]

For the first term of the equality (\ref{eq:expansion_for_cG_FP})
we have
\begin{align*}
\Bigg|\binn{\Delta_{n}\pr_{n}D_{1}F(\hat{\rho},\hat{\zeta}),\rho}_{n} & -\binn{\Delta D_{1}F(\hat{\rho},\hat{\zeta}),\hat{\rho}}\Bigg|\\
 & =\left|\binn{\ex_{n}\Delta_{n}\pr_{n}D_{1}F(\hat{\rho},\hat{\zeta}),\hat{\rho}}-\binn{\pr_{n}\Delta D_{1}F(\hat{\rho},\hat{\zeta}),\hat{\rho}}\right|\\
 & \le\bnorm{\ex_{n}\Delta_{n}\pr_{n}D_{1}F(\hat{\rho},\hat{\zeta})-\pr_{n}\Delta D_{1}F(\hat{\rho},\hat{\zeta})}_{H_{J-2}}\norm{\hat{\rho}}_{H_{-J+2}}\\
 & \le\frac{C}{n}\norm{D_{1}F(\hat{\rho},\hat{\zeta})}_{H_{J}}\norm{\hat{\rho}}\le\frac{C}{n}\norm{D_{1}F}_{\Cf},
\end{align*}
according to (\ref{eq:connection_between_continuous_and_discrete_inner_products}),
Lemma \ref{lem:discrete_and_continuous_laplace_operator} and the
fact that $J\ge2$. Similarly, we get

\begin{align*}
\left|\inn{\Delta_{n}\pr_{n}D_{2}F(\hat{\rho},\hat{\zeta}),\zeta}_{n}-\binn{\Delta D_{2}F(\hat{\rho},\hat{\zeta}),\hat{\zeta}}\right| & \le\frac{C}{n}\norm{D_{2}F(\hat{\rho},\hat{\zeta})}_{H_{I}}\norm{\hat{\zeta}}_{H_{-I+2}}\\
 & \le\frac{C}{n}\norm{D_{2}F}_{\Cf}\norm{\hat{\zeta}}_{H_{-I+2}}.
\end{align*}

To rewrite the fourth term in (\ref{eq:expansion_for_cG_FP}), which
will be denoted by $I_{4}$, we first set
\[
U_{j,n}(\rho,\eta):=\Tr\left(\partial_{n,j}^{\otimes2}\pr_{n}^{\otimes2}D_{2}^{2}F(\hat{\rho},\hat{\zeta})\right)
\]
 and note that 
\begin{align*}
\left(\partial_{n,j}\eta(x)\right)^{2} & =\frac{\left(2n+1\right)^{2}}{4\pi^{2}}\left(\eta(x+e_{j}^{n})-\eta(x)\right)^{2}\\
 & =\frac{\left(2n+1\right)^{2}}{4\pi^{2}}\left(\eta(x+e_{j}^{n})+\eta(x)-2\eta(x+e_{j}^{n})\eta(x)\right)
\end{align*}
for all $x\in\Tnd$. In terms of the shift operator $\tau_{j}^{n}$,
we get
\begin{align*}
I_{4} & =\pi^{2}\sum_{j=1}^{d}\binn{U_{j,n}(\rho,\eta),\tau_{j}^{n}\eta+\eta-2\eta\tau_{j}^{n}\eta}_{n}\\
 & =2\pi^{2}\sum_{j=1}^{d}\binn{U_{j,n}(\rho,\eta),\rho(1-\rho)}_{n}+\pi^{2}\sum_{j=1}^{d}\binn{U_{j,n}(\rho,\eta),\tau_{j}^{n}\eta-\rho}_{n}\\
 & +\pi^{2}\sum_{j=1}^{d}\binn{U_{j,n}(\rho,\eta),\eta-\rho}_{n}-2\pi^{2}\sum_{j=1}^{d}\binn{U_{j,n}(\rho,\eta),\eta\tau_{j}^{n}\eta-\rho^{2}}_{n}=:\sum_{i=1}^{4}I_{4,i}.
\end{align*}
We next estimate each term of the right hand side of the equality
above. 

$(I_{4,1})$: According to Proposition \ref{prop:expansion_of_diffusion_term_for_ssep},
there exists a function $R_{j,n}^{4,1}:\EHE\times\EEP\to\Ln$ such
that
\begin{equation}
U_{j,n}(\rho,\eta)=\pr_{n}\Tr\left(\partial_{j}^{\otimes2}D_{2}^{2}F(\hat{\rho},\hat{\zeta})\right)+R_{j,n}^{4,1}(\rho,\eta)\label{eq:function_U_j_n}
\end{equation}
due to $I>d+2$, where 

\begin{align}
\max_{x\in\T_{n}^{d}}\left|R_{j,n}^{4,1}(\rho,\eta)(x)\right| & \le\frac{C_{I}}{n}\sup_{\rho\in\EHE,\zeta\in\R^{\Tnd}}\norm{D_{2}^{2}F(\hat{\rho},\hat{\zeta})}_{\MLOHS[2][{\H[-I]}]}\le\frac{C_{I}}{n}\norm F_{\Cf_{l,HS}^{1,3}}.\label{eq:estimate_of_R_4_1}
\end{align}
We also note that
\begin{equation}
\bnorm{\ex_{n}\rho^{2}-\left(\ex_{n}\rho\right)^{2}}\le\frac{C}{n}\norm{\ex_{n}\rho}_{\Cf^{\lceil d/2\rceil+4}}^{2},\label{eq:estimate_of_ex_rho^2-ex^2_rho}
\end{equation}
by Lemma \ref{lem:exchange_of_ex_and_square}. Thus, setting $H(\rho,\eta):=\Tr\left(\partial_{j}^{\otimes2}D_{2}^{2}F(\hat{\rho},\hat{\zeta})\right)$
and using (\ref{eq:connection_between_continuous_and_discrete_inner_products}),
we can rewrite 
\begin{align*}
\binn{U_{j,n}(\rho,\eta),\rho(1-\rho)}_{n} & =\binn{\pr_{n}H(\rho,\eta),\rho}_{n}-\binn{\pr_{n}H(\rho,\eta),\rho^{2}}_{n}+\binn{R_{j,n}^{4,1}(\rho,\eta),\rho(1-\rho)}_{n}\\
 & =\binn{H(\rho,\eta),\hat{\rho}}-\binn{H(\rho,\eta),\ex_{n}\rho^{2}}+\binn{R_{j,n}^{4,1}(\rho,\eta),\rho(1-\rho)}_{n}\\
 & =\binn{H(\rho,\eta),\hat{\rho}(1-\hat{\rho})}+R_{j,n}^{4}(\rho,\eta),
\end{align*}
where $R_{j,n}^{4}(\rho,\eta):=\binn{R_{j,n}^{4,1}(\rho,\eta),\rho(1-\rho)}_{n}+\binn{H(\rho,\eta),\left(\ex_{n}\rho\right)^{2}-\ex_{n}\rho^{2}}$.
Using (\ref{eq:estimate_of_R_4_1}) and (\ref{eq:estimate_of_ex_rho^2-ex^2_rho}),
we get
\[
\left|R_{j,n}^{4}(\rho,\eta)\right|\le\frac{C_{I}}{n}\norm F_{\Cf_{l,HS}^{1,3}}+\frac{C}{n}\norm{H(\rho,\eta)}\norm{\ex_{n}\rho}_{\Cf^{\lceil d/2\rceil+4}}^{2}.
\]
Since $I>\frac{d}{2}+1$, we can use Lemmas \ref{lem:basic_properties_of_pr}
(iii), \ref{lem:tr_operator} and \ref{lem:norm_of_derivative_of_multilinear_operator}
to get 
\begin{align*}
\norm{H(\rho,\eta)} & \le C_{I}\norm{\partial_{j}^{\otimes2}D_{2}^{2}F(\hat{\rho},\hat{\zeta})}_{\MLOHS[2][{\H[-I+1]}]}\\
 & \le C_{I}\norm{D_{2}^{2}F(\hat{\rho},\hat{\zeta})}_{\MLOHS[2][{\H[-I]}]}\le C_{I}\norm F_{\Cf_{l,HS}^{1,3}}.
\end{align*}
Thus, 
\[
\left|R_{j,n}^{4}(\rho,\eta)\right|\le\frac{C_{I}}{n}\left[\norm F_{\Cf_{l,HS}^{1,3}}\left(1+\norm{\hat{\rho}}_{\Cf^{\lceil d/2\rceil+4}}^{2}\right)\right].
\]

$(I_{4,3})$: Using (\ref{eq:connection_between_continuous_and_discrete_inner_products})
and Proposition \ref{prop:expansion_of_diffusion_term_for_ssep},
we get 
\begin{align*}
\left|\binn{U_{j,n}(\rho,\eta),\eta-\rho}_{n}\right| & =\frac{1}{(2n+1)^{d/2}}\left|\binn{U_{j,n}(\rho,\eta),\zeta}_{n}\right|\\
 & =\frac{1}{(2n+1)^{d/2}}\left|\binn{\ex_{n}U_{j,n}(\rho,\eta),\hat{\zeta}}\right|\\
 & \le\frac{1}{n^{d/2}}\norm{\ex_{n}U_{j,n}(\rho,\eta)}_{H_{\tilde{I}}}\norm{\hat{\zeta}}_{H_{-\tilde{I}}}\\
 & \le\frac{C_{I,\tilde{I}}}{n^{d/2}}\norm{D_{2}^{2}F(\hat{\rho},\hat{\zeta})}_{\MLOHS[2][{\H[-I]}]}\norm{\hat{\zeta}}_{H_{-\tilde{I}}}\\
 & \le\frac{C_{I,\tilde{I}}}{n^{d/2}}\norm F_{\Cf_{l,HS}^{1,3}}\norm{\hat{\zeta}}_{H_{-\tilde{I}}}
\end{align*}
for each $\tilde{I}\ge0$ such that $\tilde{I}+1+\frac{d}{2}<I$.

$(I_{4,2}):$ We first rewrite
\begin{align*}
\binn{U_{j,n}(\rho,\eta),\tau_{j}^{n}\eta-\rho}_{n} & =\binn{U_{j,n}(\rho,\eta),\eta-\rho}_{n}-\binn{U_{j,n}(\rho,\eta),\tau_{j}^{n}\eta-\eta}_{n}\\
 & =\binn{U_{j,n}(\rho,\eta),\eta-\rho}_{n}-\frac{2\pi}{2n+1}\binn{U_{j,n}(\rho,\eta),\partial_{n,j}\eta}_{n}.
\end{align*}
The term $\binn{U_{j,n}(\rho,\eta),\eta-\rho}_{n}$ was estimated
above. We now estimate
\begin{align*}
\left|\binn{U_{j,n}(\rho,\eta),\partial_{n,j}\eta}_{n}\right| & =\left|\binn{\partial_{n,j}U_{j,n}(\rho,\eta),\eta}_{n}\right|=\left|\binn{\ex_{n}\partial_{n,j}U_{j,n}(\rho,\eta),\hat{\eta}}\right|\\
 & \le\norm{\ex_{n}\partial_{n,j}U_{j,n}(\rho,\eta)}\norm{\hat{\eta}}.
\end{align*}
Due to the fact that $\eta$ takes values from $\{0,1\}$ and Corollary
\ref{cor:connection_between_descrete_and_cont_inner_product}, we
get $\norm{\hat{\eta}}=\norm{\ex_{n}\eta}=\norm{\eta}_{n}\le1$. By
Lemma \ref{lem:estimate_of_norm_ex_partial} and Proposition \ref{prop:expansion_of_diffusion_term_for_ssep},
we obtain
\[
\norm{\ex_{n}\partial_{n,j}U_{j,n}(\rho,\eta)}\le\norm{\ex_{n}U_{j,n}(\rho,\eta)}_{H_{1}}\le C_{I}\norm F_{\Cf_{l,HS}^{1,3}},
\]
where we have used the fact that $I>2+\frac{d}{2}$. Thus, 
\[
\binn{U_{j,n}(\rho,\eta),\tau_{j}^{n}\eta-\rho}_{n}\le\frac{C_{I,\tilde{I}}}{n^{d/2}}\norm F_{\Cf_{l,HS}^{1,3}}\norm{\hat{\zeta}}_{H_{-\tilde{I}}}+\frac{C_{J}}{n}\norm F_{\Cf_{l,HS}^{1,3}}.
\]

$(I_{4,4}):$ Using the equality $\eta=\rho+n^{-d/2}\zeta$, we first
rewrite
\begin{align*}
\binn{U_{j,n}(\rho,\eta),\eta\tau_{j}^{n}\eta-\rho^{2}}_{n} & =\binn{U_{j,n}(\rho,\eta),\rho\tau_{j}^{n}\rho-\rho^{2}}_{n}\\
 & +\frac{1}{(2n+1)^{d/2}}\binn{U_{j,n}(\rho,\eta),\rho\tau_{j}^{n}\zeta}_{n}\\
 & +\frac{1}{(2n+1)^{d/2}}\binn{U_{j,n}(\rho,\eta),\zeta\tau_{j}^{n}\rho}_{n}\\
 & +\frac{1}{(2n+1)^{d}}\binn{U_{j,n}(\rho,\eta),\zeta\tau_{j}^{n}\zeta}_{n}.
\end{align*}
Let $I_{4,4,i}$, $i\in[4]$, denote the terms in the right hand side
of the equality above. We first estimate the term $I_{4,4,1}$ as
follows
\begin{align*}
|I_{4,4,1}| & =\frac{2\pi}{2n+1}\left|\binn{U_{j,n}(\rho,\eta)\rho,\partial_{n,j}\rho}_{n}\right|\le\frac{2\pi}{2n+1}\norm{U_{j,n}(\rho,\eta)\rho}_{n}\norm{\partial_{n,j}\rho}_{n}\\
 & \le\frac{2\pi}{2n+1}\norm{U_{j,n}(\rho,\eta)}_{n}\norm{\partial_{n,j}\rho}_{n}=\frac{2\pi}{2n+1}\norm{\ex_{n}U_{j,n}(\rho,\eta)}\norm{\ex_{n}\partial_{n,j}\rho}\\
 & \le\frac{C_{I}}{n}\norm{D_{2}^{2}F(\hat{\rho},\hat{\zeta})}_{\MLOHS[2][{\H[-I]}]}\norm{\hat{\rho}}_{H_{1}}\\
 & \le\frac{C_{I}}{n}\norm F_{\Cf_{l,HS}^{1,3}}\norm{\hat{\rho}}_{H_{1}},
\end{align*}
where we used (\ref{eq:connection_between_continuous_and_discrete_inner_products})
and Proposition \ref{prop:expansion_of_diffusion_term_for_ssep}.
According (\ref{eq:connection_between_continuous_and_discrete_inner_products})
and Lemmas \ref{lem:norm_of_ex_of_product} and \ref{lem:norm_of_ex_shift},
the estimate
\begin{align*}
|I_{4,4,2}| & =\frac{1}{(2n+1)^{d/2}}\left|\binn{U_{j,n}(\rho,\eta)\rho,\tau_{j}^{n}\zeta}_{n}\right|=\frac{1}{(2n+1)^{d/2}}\left|\binn{\ex_{n}\left(U_{j,n}(\rho,\eta)\rho\right),\ex_{n}\tau_{j}^{n}\zeta}\right|\\
 & \le\frac{1}{n^{d/2}}\bnorm{\ex_{n}\left(U_{j,n}(\rho,\eta)\rho\right)}_{H_{\tilde{I}}}\norm{\ex_{n}\tau_{j}^{n}\zeta}_{H_{-\tilde{I}}}\\
 & \le\frac{C_{I,\tilde{J}}}{n^{d/2}}\bnorm{\ex_{n}U_{j,n}(\rho,\eta)}_{H_{\lceil\tilde{I}\rceil}}\norm{\hat{\rho}}_{\Cf^{\lceil\tilde{I}\rceil}}\norm{\hat{\zeta}}_{H_{-\tilde{I}}}\\
 & \le\frac{C_{I,\tilde{I},\tilde{J}}}{n^{d/2}}\norm F_{\Cf_{l,HS}^{1,3}}\norm{\hat{\rho}}_{\Cf^{\lceil\tilde{I}\rceil}}\norm{\hat{\zeta}}_{H_{-\tilde{I}}}
\end{align*}
holds due to $\lceil\tilde{I}\rceil+1+\frac{d}{2}<I$. Here we estimated
$\bnorm{\ex_{n}U_{j,n}(\rho,\eta)}_{H_{\tilde{I}}}$ as in $(I_{4,3})$.
The term $I_{4,4,3}$, can be estimated similarly to $I_{4,4,2}$
by the same expression. Due to the equality (\ref{eq:function_U_j_n}),
we get
\begin{align*}
I_{4,4,4} & =\frac{1}{(2n+1)^{d}}\binn{U_{j,n}(\rho,\eta),\zeta\tau_{j}^{n}\zeta}_{n}\\
 & =\frac{1}{(2n+1)^{d}}\binn{\pr_{n}\Tr\left(\partial_{j}^{\otimes2}D_{2}^{2}F(\hat{\rho},\hat{\zeta})\right),\zeta\tau_{j}^{n}\zeta}_{n}+\binn{R_{j,n}^{4,1}(\rho,\eta),(\eta-\rho)\tau_{j}^{n}(\eta-\rho)}_{n}.
\end{align*}
Now, by (\ref{eq:estimate_of_R_4_1}) and the boundedness of $\eta$
and $\rho$, we obtain
\[
\left|\binn{R_{j,n}^{4,1}(\rho,\eta),(\eta-\rho)\tau_{j}^{n}(\eta-\rho)}_{n}\right|\le\frac{C_{I}}{n}\norm F_{\Cf_{l,HS}^{1,3}}.
\]
This completes the proof of the proposition.
\end{proof}

\subsection{Some properties of the density fluctuation field\protect\label{subsec:Some-properties-of}}

The goal of this section is to estimate the Sobolev norm of the density
fluctuation field and the expectation of the term $\binn{f,\ex_{n}\left[\zeta_{t}^{n}\tau_{j}^{n}\zeta_{t}^{n}\right]}$
appearing in the expansion of the generator $\cG^{FF}$. We first
prove an auxiliary statement.
\begin{lem}
\label{lem:estimate_of_inner_product_of_eta_with_varphi}Let $\rho_{0}^{n}\in\Ln$
take values in $[0,1]$, $\varphi\in\Cf(\T^{d})$ and $(\eta_{t}^{n})_{t\ge0}$
be the SSEP started from $\eta_{0}^{n}=(\eta_{0}^{n}(x))_{x\in\Tnd}$
for each $n\ge1$, where $\eta_{0}^{n}(x)$, $x\in\Tnd$, are independent
random variables with Bernoulli distribution with parameters $\rho_{0}^{n}(x),$
$x\in\Tnd$, respectively. Let also $\rho_{t}^{n}(x)=\E\eta_{t}^{n}(x)$,
$x\in\Tnd$, $t\ge0$, and $\zeta_{t}^{n}=\left(2n+1\right)^{d/2}(\eta_{t}^{n}-\rho_{t}^{n})$,
$t\ge0$. Then, for every $t\ge0$,
\[
\E\left[\inn{\ex_{n}\zeta_{t}^{n},\varphi}^{2}\right]\le\left(1+2\pi^{2}t\bnorm{\nabla_{n}\rho_{0}^{n}}_{n,\Cf}^{2}\right)\norm{\pr_{n}\varphi}_{n,\Cf}^{2}.
\]
\end{lem}

\begin{proof}
We set $\varphi_{n}:=\pr_{n}\varphi$ and rewrite for $n\ge1$
\begin{align*}
\E\left[\inn{\ex_{n}\zeta_{t}^{n},\varphi}^{2}\right] & =\E\left[\inn{\zeta_{t}^{n},\pr_{n}\varphi}_{n}^{2}\right]=\frac{1}{(2n+1)^{2d}}\left[\sum_{x\in\Tnd}\zeta_{t}^{n}(x)\varphi_{n}(x)\right]^{2}\\
 & =\frac{1}{(2n+1)^{2d}}\sum_{x,y\in\Tnd}\E\left[\zeta_{t}^{n}(x)\zeta_{t}^{n}(y)\right]\varphi_{n}(x)\varphi_{n}(y)\\
 & =\frac{1}{(2n+1)^{d}}\sum_{x\in\Tnd}\E\left[\left(\eta_{t}^{n}(x)-\rho_{t}^{n}(x)\right)^{2}\right]\varphi_{n}^{2}(x)\\
 & +\frac{1}{(2n+1)^{d}}\sum_{x\not=y\in\Tnd}\E\left[(\eta_{t}^{n}(x)-\rho_{t}^{n}(x))(\eta_{t}^{n}(y)-\rho_{t}^{n}(y))\right]\varphi_{n}(x)\varphi_{n}(y).
\end{align*}

The first term of the right hand side of the equality above can be
estimated by 
\[
\frac{1}{(2n+1)^{d}}\sum_{x\in\T_{n}^{d}}\varphi_{n}^{2}(x)=\norm{\pr_{n}\varphi}_{n}^{2},
\]
due to the fact that $\eta_{t}^{n}(x)-\rho_{t}^{n}(x)\in[0,1]$ for
all $x\in\Tnd$ and $t\ge0$. The second term can be estimated by
\[
2\pi^{2}\sup_{s\in[0,t]}\max_{u\in\Tnd}\left|\nabla_{n}\rho_{s}^{n}(u)\right|^{2}\norm{\pr_{n}\varphi}_{n,\Cf}^{2}t
\]
similarly to the proof of the main theorem in \citep[p. 32]{Ravishankar:1992}
(see also Section \ref{subsec:Some-additional-facts} for the detailed
estimate). Combining both estimates together, we get
\begin{align*}
\E\left[\inn{\ex_{n}\zeta_{t}^{n},\varphi}^{2}\right] & \le\norm{\varphi_{n}}_{n}^{2}+2\pi^{2}\sup_{s\in[0,t]}\max_{u\in\Tnd}\left|\nabla_{n}\rho_{s}^{n}(u)\right|^{2}\norm{\varphi_{n}}_{n,\Cf}^{2}t\\
 & \le\left(1+2\pi^{2}t\sup_{s\in[0,t]}\bnorm{\nabla_{n}\rho_{s}^{n}}_{n,\Cf}^{2}\right)\norm{\varphi_{n}}_{n,\Cf}^{2}\\
 & \le\left(1+2\pi^{2}t\bnorm{\nabla_{n}\rho_{0}^{n}}_{n,\Cf}^{2}\right)\norm{\varphi_{n}}_{n,\Cf}^{2},
\end{align*}
according to the fact that $\rho_{t}^{n}$, $t\ge0$, is a solution
to (\ref{eq:descrete_heat_equation}) and the maximum principle. This
completes the proof of the lemma.
\end{proof}
\begin{lem}
\label{lem:estimate_of_sobolev_norm_of_eta}Let $I>\frac{d}{2}$.
Under the assumptions of Lemma \ref{lem:estimate_of_inner_product_of_eta_with_varphi},
for every $t\ge0$ one has
\[
\E\left[\norm{\ex_{n}\zeta_{t}^{n}}_{H_{-I}}^{2}\right]\le C_{I}\left(1+2\pi^{2}t\bnorm{\nabla_{n}\rho_{0}^{n}}_{n,\Cf}^{2}\right).
\]
\end{lem}

\begin{proof}
By the definition of $\norm{\cdot}_{H_{-I}}$ and Lemma \ref{lem:estimate_of_inner_product_of_eta_with_varphi},
we get
\begin{align*}
\E\left[\norm{\ex_{n}\zeta_{t}^{n}}_{H_{-I}}^{2}\right] & =\sum_{k\in\Z^{d}}\left(1+|k|^{2}\right)^{-I}\E\left|\inn{\ex_{n}\zeta_{t}^{n},\tilde{\varsigma}_{k}}\right|^{2}\\
 & \le\sum_{k\in\Z^{d}}\left(1+|k|^{2}\right)^{-I}\left(1+2\pi^{2}t\bnorm{\nabla_{n}\rho_{0}^{n}}_{n,\Cf}^{2}\right)\norm{\pr_{n}\tilde{\varsigma}_{k}}_{n,\Cf}^{2}\\
 & =\sum_{k\in\Z_{n}^{d}}\left(1+|k|^{2}\right)^{-I}\left(1+2\pi^{2}t\bnorm{\nabla_{n}\rho_{0}^{n}}_{n,\Cf}^{2}\right)\norm{\tilde{\varsigma}_{k}}_{n,\Cf}^{2}\\
 & \le C_{I}\left(1+2\pi^{2}t\bnorm{\nabla_{n}\rho_{0}^{n}}_{n,\Cf}^{2}\right),
\end{align*}
where we also used the boundedness of $\tilde{\varsigma}_{k}$ for
the estimate of $\norm{\tilde{\varsigma}_{k}}_{n,\Cf}$. The proof
of the lemma is complete.
\end{proof}
We recall that $\tau_{j}^{n}$ denotes the shift operator on $\Tnd$
defined by $\tau_{j}^{n}f(x)=\tau_{j}^{n}(x+e_{j}^{n})$.
\begin{lem}
\label{lem:estimate_of_eta_tau_eta} Let $J>\frac{d}{2}$. Under the
assumptions of Lemma \ref{lem:estimate_of_inner_product_of_eta_with_varphi},
for every $T>0$ there exists a constant $C$ depending on $J$, $T$
and $\sup_{n\ge1}\norm{\nabla_{n}\rho_{0}^{n}}_{n,\Cf}$ such that
for every random variable $f$ in $\H[J]$ with a finite second moment
and defined on the same probability space as $\zeta^{n}$ we have
\[
\left|\frac{1}{(2n+1)^{d}}\E\binn{f,\ex_{n}\left[\zeta_{t}^{n}\tau_{j}^{n}\zeta_{t}^{n}\right]}\right|\le\frac{C}{n^{\frac{d}{2}\wedge1}}\E\left[\norm f_{H_{J}}^{2}\right]^{\frac{1}{2}}
\]
for each $n\ge1$, $j\in[d]$ and $t\in[0,T]$.
\end{lem}

\begin{proof}
Using Parseval's identity, (\ref{eq:connection_betwwen_ex_and_pr}),
the Cauchy-Schwarz inequality and (\ref{lem:basic_properties_of_pr})
(i), we get
\begin{align*}
\left|\frac{1}{(2n+1)^{d}}\E\binn{f,\ex_{n}\left[\zeta_{t}^{n}\tau_{j}^{n}\zeta_{t}^{n}\right]}\right|^{2} & =\left|\frac{1}{(2n+1)^{d}}\E\inn{\pr_{n}f,\zeta_{t}^{n}\tau_{j}^{n}\zeta_{t}^{n}}_{n}\right|^{2}\\
 & =\frac{1}{(2n+1)^{2d}}\left|\sum_{k\in\Znd}\E\left[\inn{\pr_{n}f,\varsigma_{k}}_{n}\inn{\varsigma_{k},\zeta_{t}^{n}\tau_{j}^{n}\zeta_{t}^{n}}_{n}\right]\right|^{2}\\
 & \le\frac{1}{(2n+1)^{2d}}\sum_{k\in\Znd}(1+|k|^{2})^{J}\E\left[\left|\inn{\pr_{n}f,\varsigma_{k}}_{n}\right|^{2}\right]\\
 & \qquad\qquad\cdot\sum_{k\in\Znd}\frac{1}{(1+|k|^{2})^{J}}\E\left[\left|\inn{\varsigma_{k},\zeta_{t}^{n}\tau_{j}^{n}\zeta_{t}^{n}}\right|^{2}\right]\\
 & \le\frac{C_{J}}{(2n+1)^{2d}}\E\left[\norm f_{H_{J}}^{2}\right]\max_{k\in\Znd}\E\left[\left|\inn{\varsigma_{k},\zeta_{t}^{n}\tau_{j}^{n}\zeta_{t}^{n}}\right|^{2}\right]
\end{align*}
since $J>\frac{d}{2}$. We next estimate for each $k\in\Znd$
\begin{align*}
 & \frac{1}{(2n+1)^{2d}}\E\left[\left|\inn{\varsigma_{k},\zeta_{t}^{n}\tau_{j}^{n}\zeta_{t}^{n}}\right|^{2}\right]=\frac{1}{(2n+1)^{2d}}\E\left[\inn{\varsigma_{k},\zeta_{t}^{n}\tau_{j}^{n}\zeta_{t}^{n}}\inn{\zeta_{t}^{n}\tau_{j}^{n}\zeta_{t}^{n},\varsigma_{k}}\right]\\
 & \qquad=\frac{1}{(2n+1)^{2d}}\sum_{x,y\in\Tnd}\varsigma_{k}(x)\varsigma_{-k}(y)\E\Big[(\eta_{t}^{n}(x)-\rho_{t}^{n}(x))(\eta_{t}^{n}(x+e_{j}^{n})-\rho_{t}^{n}(x+e_{j}^{n}))\\
 & \qquad\qquad\qquad\qquad\qquad\qquad\cdot(\eta_{t}^{n}(y)-\rho_{t}^{n}(y))(\eta_{t}^{n}(y+e_{j}^{n})-\rho_{t}^{n}(y+e_{j}^{n}))\Big].
\end{align*}

Following the observation in \citep[Theorem 6.1]{Jara_Landim:2008},
that in our setting will follow from similar computations \citep{Ferrari_Presutti:1991},
we can bound the expectation above by $\frac{C}{n^{2}}$ for distinct
$x,x+e_{j}^{n},y,y+e_{j}^{n}$, where the constant $C$ depends on
$T$ and $\sup_{n\ge1}\norm{\nabla_{n}\rho_{0}^{n}}_{n,\Cf}$. The
cardinality of the set 
\[
\bigg\{(x,y)\in\left(\Tnd\right)^{2}:\ x,x+e_{j}^{n},y,y+e_{j}^{n}\ \mbox{are not distinct}\bigg\}
\]
 is bounded by $3(2n+1)^{d}$. Thus, we can continue the estimate
by
\[
\frac{1}{(2n+1)^{2d}}\left[\frac{(2n+1)^{2d}C}{n^{2}}+3(2n+1)^{d}\right]=\frac{C}{n^{2}}+\frac{3}{(2n+1)^{d}}.
\]
Consequently, there exists a constant $C>0$ such that 
\[
\left|\frac{1}{(2n+1)^{d}}\E\binn{f,\ex_{n}\left[\zeta_{t}^{n}\tau_{j}^{n}\zeta_{t}^{n}\right]}\right|^{2}\le\frac{C}{n^{2\wedge d}}\E\left[\norm f_{H_{J}}^{2}\right].
\]
This completes the proof of the statement.
\end{proof}

\section{Generalized Ornstein-Uhlenbeck process\protect\label{sec:Generalized-Ornstein-Uhlenbeck}}

The main result of this section is the regularity of the solution
$U_{t}$, $t\ge0$, to the infinite-dimensional Kolmogorov backward
equation corresponding to the system of SPDEs (\ref{eq:heat_PDE}),
(\ref{eq:SPDE_for_OU_process}), which is defined by $U_{t}:=P_{t}^{OU}F$. 

The proof of this regularity faces several challenges due to the form
of the diffusion terms in (\ref{eq:SPDE_for_OU_process}). Firstly,
$\sqrt{\rho(1-\rho)}$ is not differentiable, which prevents from
following the usual approach to deduce the regularity of $U_{t}$
from the regularity of solutions to (\ref{eq:SPDE_for_OU_process})
with respect to their initial conditions. Secondly, the variance term
$\rho(1-\rho)$ is non-negative only for $\rho\in[0,1]$, and, as
a result, the function $U_{t}$ is well-defined only on a subset of
$\H[J]\times\H[-I]$. This is particularly problematic since the discrete
semigroup $(\hat{\rho}_{s}^{n},\hat{\zeta}_{s}^{n})$ does not necessarily
take values in this domain, since $\hat{\rho}_{s}^{n}$ is not a $[0,1]$-valued
function in general. However, this property is crucial for our main
approach based on (\ref{eq:the_main_comparison_of_generators-2}). 

To overcome the latter problem and also to avoid the discussion of
the differentiability of $U_{t}$ at boundary points of its domain,
in this section we first approximate $\rho(1-\rho)$ in the SPDE (\ref{eq:SPDE_for_OU_process})
by a smooth mollification $\Phi^{\eps}$ of the non-negative function
$\rho(1-\rho)\vee0$, such that $\sup_{\eps\in(0,1]}\norm{\Phi^{\eps}}_{\Cf^{1}}<\infty$.
This allows to approximate the function $U_{t}$ by solutions $U_{t}^{\eps}$
to Kolmogorov equations that now are well-defined on the complete
space $\H[J]\times\H[-I]$. Then, in Section \ref{sec:Comparison-of-processes},
we compare the corresponding generators on the functions $U_{t}^{\eps}$
and show that the additional mollification error can be well-controled. 

The remaining difficulty of the non-differentiability of the diffusion
coefficient $\sqrt{\rho(1-\rho)}$ is addressed in Section \ref{subsec:Differentiability-of-semigroup}
below. 

\subsection{Covariance and \Itos formula\protect\label{subsec:Covariance-and-Itos-formula}}

In this section, we fix a continuous bounded function $\Phi:\R\to[0,\infty)$
and build a Gaussian process in $\H[-I]$ for some $I$ that will
be used for the description of fluctuations of the SSEP. We first
consider the heat equation 
\begin{equation}
d\rho_{t}^{\infty}=2\pi^{2}\Delta\rho_{t}^{\infty}dt\label{eq:heat_PDE-1}
\end{equation}
in $\H$, for some $J\ge0$, with initial condition $\rho_{0}\in\H[J]$.
It is well-known that there exists a (continuous) $\H[J]$-valued
weak solution $(\rho_{t}^{\infty})_{t\ge0}$ to the heat equation
(\ref{eq:heat_PDE-1}). The semigroup associated with the PDE (\ref{eq:heat_PDE-1})
will be denoted by $P_{t}$, $t\ge0$. In particular, 
\[
\rho_{t}^{\infty}=P_{t}\rho_{0},\quad t\ge0.
\]
We next define the generalized Ornstein-Uhlenbeck process $(\zeta_{t}^{\infty})_{t\ge0}$
as the variational\footnote{See \citep[Definition 4.2.1]{Liu_Roeckner:2015} }
solution to the SPDE
\begin{align}
d\zeta_{t}^{\infty} & =2\pi^{2}\Delta\zeta_{t}^{\infty}dt+2\pi\nabla\cdot\left(\sqrt{\Phi(\rho_{t}^{\infty})}dW_{t}\right),\label{eq:SPDE_for_OU_process-1}
\end{align}
where $dW$ is a $d$-dimensional white noise. The differentiability
of the associated semigroup will follow from the differentiability
of the variance operator for the Ornstein-Uhlenbeck process whose
precise form is described in the next proposition.
\begin{prop}
\label{prop:well_possedness_of_OUP_SSEP}Let $\Phi$ be a bounded
non-negative continuous function. For each $I>\frac{d}{2}+1$, $\rho_{0}\in L_{2}(\T^{d})$
and $\zeta_{0}\in\H[-I]$ there exists a unique continuous $\H[-I]$-valued
variational solution $(\zeta_{t}^{\infty})_{t\ge0}$ to the SPDE (\ref{eq:SPDE_for_OU_process-1})
started from $\zeta_{0}$ and 
\[
\E\sup_{t\in[0,T]}\norm{\zeta_{t}^{\infty}}_{H_{-I}}^{2}<\infty
\]
for each $T>0$, where $(\rho_{t}^{\infty})_{t\ge0}$ solves the heat
equation (\ref{eq:SPDE_for_OU_process-1}) with initial condition
$\rho_{0}$. Moreover, $\zeta^{\infty}$ is a Gaussian process in
$\H[-I]$ with expectation
\begin{equation}
m_{t}(\zeta_{0})[\varphi]:=\E\left\langle \varphi,\zeta_{t}^{\infty}\right\rangle =\left\langle P_{t}\varphi,\zeta_{0}\right\rangle ,\quad\varphi\in\Cf^{\infty}(\Td),\label{eq:expectation_of_eta}
\end{equation}
 and covariance operator
\begin{align}
V_{t}(\rho)[\varphi,\psi]: & =\cov\left(\inn{\zeta_{t}^{\infty},\varphi},\inn{\psi,\zeta_{t}^{\infty}}\right)\nonumber \\
 & =2\pi^{2}\int_{0}^{t}\binn{\nabla P_{t-s}\varphi\cdot\nabla P_{t-s}\psi,\Phi\left(P_{s}\rho\right)}ds,\quad\varphi,\psi\in\Cf^{\infty}(\Td),\label{eq:eq:definition_of_covariance_operator}
\end{align}
for each $t>0$.
\end{prop}

\begin{proof}
The existence and uniqueness of the variational solution to the SPDE
(\ref{eq:SPDE_for_OU_process-1}) follows from \citep[Theorem 4.2.4]{Liu_Roeckner:2015}
and the fact that $B(t):L_{2}(\T^{d};\R^{d})\to\H[-I]$ defined by
\begin{equation}
B(t)h:=2\pi\nabla\cdot\left(\sqrt{\Phi(\rho_{t})}h\right)\label{eq:definition_of_B}
\end{equation}
is a Hilbert-Schmidt operator with Hilbert-Schmidt norm
\begin{align*}
\norm{B(t)}_{HS}^{2} & :=\sum_{l=1}^{\infty}\norm{B(t)h_{l}}_{H_{-I}}^{2}=4\pi^{2}\sum_{l=1}^{\infty}\sum_{k\in\Z^{d}}(1+|k|^{2})^{-I}\left|\inn{\sqrt{\Phi(\rho_{t})}h_{l},\nabla\varsigma_{k}}\right|^{2}\\
 & =4\pi^{2}\sum_{l=1}^{\infty}\sum_{k\in\Z^{d}}(1+|k|^{2})^{-I}\left|\inn{h_{l},\sqrt{\Phi(\rho_{t})}k\varsigma_{k}}\right|^{2}\\
 & =4\pi^{2}\sum_{k\in\Z^{d}}(1+|k|^{2})^{-I+1}\norm{\sqrt{\Phi(\rho_{t})}\varsigma_{k}}^{2}\le4\pi^{2}\norm{\Phi}_{\Cf}\sum_{k\in\Z^{d}}(1+|k|^{2})^{-I+1}<\infty,
\end{align*}
where $\{h_{l},\ l\ge1\}$ is an orthonormal basis of $(L_{2}(\Td))^{d}$.
Note that the construction of the variational solution is obtained
by Galerkin approximation leading to linear SDEs \citep[(4.48)]{Liu_Roeckner:2015}.
This implies that the process $\zeta^{\infty}$ is Gaussian in $\H[-I]$
as a limit of Gaussian processes. 

Let $T>0$ and $\varphi\in\Cf^{\infty}(\T^{d})$ be fixed. Consider
$\psi_{t}:=P_{T-t}\varphi\in\H[I+2]$ for all $t\in[0,T]$ and use
the martingale problem for $\zeta^{\infty}$ and \Itos formula to
get
\begin{align*}
\inn{\psi_{t},\zeta_{t}^{\infty}} & =\inn{\psi_{0},\zeta_{0}}+\int_{0}^{t}\binn{\partial_{s}\psi_{s},\zeta_{s}^{\infty}}ds+2\pi^{2}\int_{0}^{t}\inn{\Delta\psi_{s},\zeta_{s}^{\infty}}ds+\mbox{mart.}\\
 & =\inn{P_{T}\varphi,\zeta_{0}}+\mbox{mart.}
\end{align*}
for all $t\in[0,T]$ a.s. Thus, taking the expectation and setting
$t=T$, we get
\begin{equation}
m_{T}(\zeta_{0})[\varphi]=\E\left\langle \varphi,\zeta_{T}^{\infty}\right\rangle =\left\langle P_{T}\varphi,\zeta_{0}\right\rangle .\label{eq:computaion_of_expectation_of_OU_process}
\end{equation}

Similarly, we compute
\begin{align*}
\inn{\psi_{t},\zeta_{t}^{\infty}}^{2} & =\inn{\psi_{0},\zeta_{0}}^{2}+2\int_{0}^{t}\inn{\psi_{s},\zeta_{s}^{\infty}}\binn{\partial_{s}\psi_{s},\zeta_{s}^{\infty}}ds+4\pi^{2}\int_{0}^{t}\inn{\psi_{s},\zeta_{s}^{\infty}}\inn{\Delta\psi_{s},\zeta_{s}^{\infty}}ds\\
 & +2\pi^{2}\int_{0}^{t}\binn{\left|\nabla\psi_{s}\right|^{2},\Phi(\rho_{s}^{\infty})}ds+\mbox{mart.}\\
 & =\inn{P_{T}\varphi,\zeta_{0}}^{2}+2\pi^{2}\int_{0}^{t}\binn{\left|\nabla\psi_{s}\right|^{2},\Phi(\rho_{s}^{\infty})}ds+\mbox{mart.}
\end{align*}
Therefore, using (\ref{eq:computaion_of_expectation_of_OU_process}),
we obtain for $t=T$
\begin{align*}
\var\inn{\varphi,\zeta_{T}^{\infty}} & =\var\inn{\psi_{T},\zeta_{T}^{\infty}}=\E\left[\inn{\psi_{T},\zeta_{T}^{\infty}}^{2}\right]-\left[\E\inn{\psi_{T},\zeta_{T}^{\infty}}\right]^{2}\\
 & =\inn{P_{T}\varphi,\zeta_{0}}^{2}+2\pi^{2}\int_{0}^{T}\binn{\left|\nabla\psi_{s}\right|^{2},\Phi(\rho_{s}^{\infty})}ds-\inn{P_{T}\varphi,\zeta_{0}}^{2}\\
 & =2\pi^{2}\int_{0}^{T}\binn{\left|\nabla P_{T-s}\varphi\right|^{2},\Phi(\rho_{s}^{\infty})}ds.
\end{align*}
The expression for the covariance operator $V_{t}(\rho)$ follows
from the polarization equality. This completes the proof of the proposition.
\end{proof}
\begin{rem}
\label{rem:equation_with_general_coefficient}The statement of the
theorem remains valid if $\Phi(\rho_{t})$ is replaced by $\Phi_{t}$
for each measurable locally bounded function $\Phi:[0,\infty)\to\L$
with $\Phi_{t}\ge0$ for all $t\ge0$. 
\end{rem}

\begin{lem}
\label{lem:whide_continuity_of_zeta}Let $\Phi^{n}:[0,\infty)\to\L$,
$n\in\N_{0}$, be locally bounded functions such that $\Phi_{t}^{n}\ge0$
for all $t\ge0$, $n\in\N_{0}$ and
\[
\sup_{t\in[0,T]}\norm{\Phi_{t}^{n}-\Phi_{t}^{0}}\to0,\quad n\to\infty,
\]
for each $T>0$. Additionally assume that $\zeta^{n}\to\zeta^{0}$
in $\H[-I]$ for some $I>\frac{d}{2}+1$ and $t_{n}\to t_{0}$ in
$[0,\infty)$ as $n\to\infty$. Let also $(\zeta_{t}^{\infty,n})_{t\ge0}$
be a (variational) solution to 
\[
d\zeta_{t}^{\infty,n}=2\pi^{2}\Delta\zeta_{t}^{\infty,n}dt+2\pi\nabla\cdot\left(\sqrt{\Phi_{t}^{n}}dW_{t}\right)
\]
started from $\zeta^{n}$ for every $n\in\N_{0}$. Then $\law\zeta_{t_{n}}^{\infty,n}\to\law\zeta_{t}^{\infty,0}$
in the 2-Wasserstein topology on the space of probability measures
on $\H[-I]$ with a finite second moment as $n\to\infty$. In particular,
for each $F\in\Cf_{l}^{1}(\H[-I])$ 
\[
\E F(\zeta_{t_{n}}^{\infty,n})\to\E F(\zeta_{t}^{\infty,0})
\]
as $n\to\infty$.
\end{lem}

\begin{proof}
We will first show that $\zeta_{t_{n}}^{\infty,n}\to\zeta_{t_{0}}^{\infty,0}$
in distribution as $n\to\infty$, using \citep[Example 3.8.15]{Bogachev:1998}.
For this we will show that the means $\E\zeta_{t_{n}}^{\infty,n}$
converge to $\E\zeta_{t_{0}}^{\infty,0}$ in $\H[-I]$, the covariance
operators $V_{t_{n}}^{n}$ of $\zeta_{t_{n}}^{\infty,n}$ converge
to the covariance operator $V_{t_{0}}^{0}$ of $\zeta_{t_{0}}^{\infty,0}$
in $\MLO[2][{\H[I]}]$ and $\E[\norm{\zeta_{t_{n}}^{\infty,n}}_{\H[-I]}^{2}]\to\E[\norm{\zeta_{t_{0}}^{\infty,0}}_{\H[-I]}^{2}]$. 

By Proposition \ref{prop:well_possedness_of_OUP_SSEP} and Remark
\ref{rem:equation_with_general_coefficient}, we get
\begin{align*}
\bnorm{\E\zeta_{t_{n}}^{\infty,n}-\E\zeta_{t_{0}}^{\infty,0}}_{\H[-I]} & =\bnorm{P_{t_{n}}\zeta^{n}-P_{t_{0}}\zeta^{0}}_{H_{-I}}\\
 & \le\bnorm{P_{t_{n}}\left(\zeta^{n}-\zeta^{0}\right)}_{H_{-I}}+\bnorm{P_{t_{n}}\zeta^{0}-P_{t_{0}}\zeta^{0}}_{H_{-I}}\\
 & \le\norm{\zeta^{n}-\zeta^{0}}_{\H[-I]}+\bnorm{P_{t_{n}}\zeta^{0}-P_{t_{0}}\zeta^{0}}_{H_{-I}}\to0
\end{align*}
as $n\to\infty$. We similarly estimate
\[
\bnorm{V_{t_{n}}^{n}-V_{t_{0}}^{0}}_{\cL_{2}}\le\norm{V_{t_{n}}^{n}-V_{t_{n}}^{0}}_{\cL_{2}}+\norm{V_{t_{n}}^{0}-V_{t_{0}}^{0}}_{\cL_{2}}.
\]
The fact that $\norm{V_{t_{n}}^{0}-V_{t_{0}}^{0}}_{\cL_{2}}\to0$
follows from the continuity of $(\zeta_{t}^{\infty,0})_{t\ge0}$ in
$\H[-I]$ and \citep[Example 3.8.15]{Bogachev:1998}. Next, using
Proposition \ref{prop:well_possedness_of_OUP_SSEP} and Remark \ref{rem:equation_with_general_coefficient}
again, we estimate
\begin{align*}
\bnorm{V_{t_{n}}^{n}-V_{t_{n}}^{0}}_{\cL_{2}}^{2} & \le\bnorm{V_{t_{n}}^{n}-V_{t_{n}}^{0}}_{\cL_{2}^{HS}}^{2}\\
 & =\sum_{k,l\in\Z^{d}}(1+|k|^{2})^{-I}(1+|l|^{2})^{-I}\left|V_{t_{n}}^{n}(\tilde{\varsigma}_{k},\tilde{\varsigma}_{l})-V_{t_{n}}^{0}(\tilde{\varsigma}_{k},\tilde{\varsigma}_{l})\right|^{2}\\
 & \le4\pi^{4}t_{n}\sum_{k,l\in\Z^{d}}(1+|k|^{2})^{-I}(1+|l|^{2})^{-I}\\
 & \qquad\qquad\qquad\cdot\int_{0}^{t_{n}}\left|\binn{\nabla P_{t_{n}-s}\tilde{\varsigma}_{k}\cdot\nabla P_{t_{n}-s}\tilde{\varsigma}_{l},\Phi_{s}^{n}-\Phi_{s}^{0}}\right|^{2}ds.
\end{align*}
According to the fact that 
\[
P_{t}\tilde{\varsigma}_{k}=e^{-2\pi^{2}|k|^{2}t}\tilde{\varsigma}_{k},\quad k\in\Z^{d},
\]
we get
\[
\nabla P_{t_{n}-s}\tilde{\varsigma}_{k}\cdot\nabla P_{t_{n}-s}\tilde{\varsigma}_{l}=-e^{-2\pi^{2}(|k|^{2}+|l|^{2})(t_{n}-s)}k\cdot l\tilde{\varsigma_{k}}\tilde{\varsigma_{l}}.
\]
We now separately estimate for $k,l\in\Z^{d}$ and a bounded measurable
function $f:[0,t]\to\L$
\begin{align}
 & \int_{0}^{t_{n}}\binn{\nabla P_{t_{n}-s}\tilde{\varsigma}_{k}\cdot\nabla P_{t_{n}-s}\tilde{\varsigma}_{l},f_{s}}^{2}ds\le|k|^{2}|l|^{2}\int_{0}^{t_{n}}e^{-4\pi^{2}(|k|^{2}+|l|^{2})(t_{n}-s)}\inn{\tilde{\varsigma_{k}}\tilde{\varsigma_{l}},f_{s}}^{2}ds\nonumber \\
 & \qquad\le|k|^{2}|l|^{2}\sup_{s\in[0,t_{n}]}\norm{f_{s}}^{2}\int_{0}^{t_{n}}e^{-4\pi^{2}(|k|^{2}+|l|^{2})(t_{n}-s)}ds\label{eq:estimate_for_derivative_of_V}\\
 & \qquad\le\frac{\sup_{s\in[0,t_{n}]}\norm{f_{s}}^{2}}{4\pi^{2}}\frac{|k|^{2}|l|^{2}}{|k|^{2}+|l|^{2}}.\nonumber 
\end{align}
Hence, due to the fact that $I>\frac{d}{2}+1$, we conclude that 
\begin{align*}
\bnorm{V_{t_{n}}^{n}-V_{t_{n}}^{0}}_{\cL_{2}^{HS}}^{2} & \le C_{I}t_{n}\sup_{s\in[0,t_{n}]}\norm{\Phi_{s}^{n}-\Phi_{s}^{0}}\to0
\end{align*}
as $n\to\infty$. The convergence of the second moments $\E[\norm{\zeta_{t_{n}}^{\infty,n}}_{\H[-I]}^{2}]$
to the second moment $\E[\norm{\zeta_{t}^{\infty,0}}_{\H[-I]}^{2}]$
can be proved similarly. Hence, by \citep[Example 3.8.15]{Bogachev:1998},
$\zeta_{t_{n}}^{\infty,n}\to\zeta_{t_{0}}^{\infty,0}$ in $\H[-I]$
in distribution as $n\to\infty$. Now, using the fact that $\E[\norm{\zeta_{t_{n}}^{\infty,n}}_{\H[-I]}^{2}]\to\E[\norm{\zeta_{t}^{\infty,0}}_{\H[-I]}^{2}]$
as $n\to\infty$, we can conclude that $\law\zeta_{t_{n}}^{\infty,n}\to\law\zeta_{t_{0}}^{\infty,0}$
in the 2-Wasserstein topology on the space of probability measures
on $\H[-I]$, by \citep[Theorem I.6.9]{Villani:2009}. This easily
implies the second part of the lemma.
\end{proof}
\begin{rem}
\label{rem:identity_for_solutions_in_sobolev_spaces} According to
the definition of variational solutions, we have
\[
\rho_{t}^{\infty}=\rho_{0}+2\pi^{2}\int_{0}^{t}\Delta\rho_{s}^{\infty}ds,\quad t\ge0,
\]
in $\H[J-2]$ and 
\[
\zeta_{t}^{\infty}=\zeta_{0}+2\pi^{2}\int_{0}^{t}\Delta\zeta_{s}^{\infty}ds+2\pi\int_{0}^{t}B(s)dW_{s},\quad t\ge0,
\]
in $\H[-I-2]$.
\end{rem}

We will provide here the \Ito formula for the process $(\rho^{\infty},\zeta^{\infty})$.
Note that while \Itos formula for Hilbert space valued processes
is available in the literature, we need to obtain the resulting \Ito-correction
term in a particular form. We therefore include the result. 
\begin{lem}
\label{lem:Ito_formula_for_eta} Let $I>\frac{d}{2}+1$, $J\ge0$,
$F\in\Cf^{1,1,2}([0,\infty),\H[J-2],H_{-I-2})$, $D_{2}^{2}F$ take
values in $\MLOHS[2][{\H[-I-2]}]$ and $(\rho_{t}^{\infty},\zeta_{t}^{\infty})$,
$t\ge0$, be a solution in $\H[J]\times\H[-I]$ to (\ref{eq:heat_PDE-1}),
(\ref{eq:SPDE_for_OU_process-1}) started from $(\rho_{0},\zeta_{0})\in\H[J]\times\H[-I]$.
Then
\begin{align*}
F_{t}(\rho_{t}^{\infty},\zeta_{t}^{\infty}) & =F_{0}(\rho_{0},\zeta_{0})+2\pi\int_{0}^{t}\binn{D_{2}F_{s}(\rho_{s}^{\infty},\zeta_{s}^{\infty}),B(s)dW_{s}}\\
 & +\int_{0}^{t}\partial F_{s}(\rho_{s}^{\infty},\zeta_{s}^{\infty})ds+2\pi^{2}\int_{0}^{t}\binn{\Delta D_{1}F_{s}(\rho_{s}^{\infty},\zeta_{s}^{\infty}),\rho_{s}^{\infty}}ds\\
 & +2\pi^{2}\int_{0}^{t}\binn{\Delta D_{2}F_{s}(\rho_{s}^{\infty},\zeta_{s}^{\infty}),\zeta_{s}^{\infty}}ds\\
 & +2\pi^{2}\int_{0}^{t}\sum_{j=1}^{d}\binn{\Tr\left(\partial_{j}^{\otimes2}D_{2}^{2}F_{s}(\rho_{s}^{\infty},\zeta_{s}^{\infty})\right),\Phi(\rho_{s}^{\infty})}ds
\end{align*}
for all $t\ge0$, where $(B(t))_{t\ge0}$ is defined by (\ref{eq:definition_of_B}).
\end{lem}

\begin{proof}
We first note that according to the assumptions on $J$ and $I$,
the process $(\rho_{t}^{\infty},\zeta_{t}^{\infty}),$ $t\ge0$, has
a continuous version in $\H[J]\times\H[-I]$ and thus the identities
of Remark \ref{rem:identity_for_solutions_in_sobolev_spaces} hold
in the spaces $\H[J-2]$ and $\H[-I-2]$, respectively. Using then
the infinite-dimensional \Ito formula\footnote{see e.g. \citep[Theorem 2.10]{Gawarecki:2011}}
in the Hilbert space $\H[J-2]\times\H[-I-2]$, we get
\begin{align}
F_{t}(\rho_{t}^{\infty},\zeta_{t}^{\infty}) & =F_{0}(\rho_{0},\zeta_{0})+2\pi\int_{0}^{t}\binn{D_{2}F_{s}(\rho_{s}^{\infty},\zeta_{s}^{\infty}),B(s)dW_{s}}\nonumber \\
 & +\int_{0}^{t}\partial F_{s}(\rho_{s}^{\infty},\zeta_{s}^{\infty})ds+2\pi^{2}\int_{0}^{t}\binn{D_{1}F_{s}(\rho_{s}^{\infty},\zeta_{s}^{\infty}),\Delta\rho_{s}^{\infty}}ds\label{eq:Ito_formula_in_hilbert_space}\\
 & +2\pi^{2}\int_{0}^{t}\binn{D_{2}F_{s}(\rho_{s}^{\infty},\zeta_{s}^{\infty}),\Delta\zeta_{s}^{\infty}}ds\nonumber \\
 & +\frac{1}{2}\int_{0}^{t}\tr\left[D_{2}^{2}F_{s}(\rho_{s}^{\infty},\zeta_{s}^{\infty})B(s)B^{*}(s)\right]ds,\nonumber 
\end{align}
where $B^{*}(s):\H[-I-2]\to(L_{2}(\Td))^{d}$ is the adjoint operator
to $B(s)$ and $U(s):=D_{3}^{2}F_{s}(\rho_{s},\zeta_{s})B(s)B^{*}(s)$
is interpreted as a bounded linear operator on $\H[-I-2]$ defined
by
\[
\inn{U(s)\varsigma_{k},\varsigma_{l}}_{H_{-I-2}}=D_{2}^{2}F_{s}(\rho_{s},\zeta_{s})\left[B(s)B^{*}(s)\varsigma_{k},\varsigma_{-l}\right],\quad k,l\in\Z^{d}.
\]
We next rewrite the last term in the right hand side of (\ref{eq:Ito_formula_in_hilbert_space}).
For this, we take the orthonormal basis $\left\{ \left(1+|k|^{2}\right)^{(I+2)/2}\varsigma_{k},\ k\in\Z^{d}\right\} $
on $\H[-I-2]$ and compute
\[
\tr\left[U(s)\right]=\sum_{k\in\Z^{d}}(1+|k|^{2})^{I+2}D_{2}^{2}F_{s}(\rho_{s},\zeta_{s})\left[B(s)B^{*}(s)\varsigma_{k},\varsigma_{-k}\right].
\]
Taking also an orthonormal basis $\{h_{l}=(h_{l}^{j})_{j\in[d]},\ l\in\N\}$
on $(L_{2}(\T^{d}))^{d}$, we can expand $B^{*}(s)\varsigma_{k}$
in the Fourier series
\begin{align*}
B^{*}(s)\varsigma_{k} & =\sum_{l=1}^{\infty}\inn{B^{*}(s)\varsigma_{k},h_{l}}h_{l}=\sum_{l=1}^{\infty}\inn{\varsigma_{k},B(s)h_{l}}_{H_{-I}}h_{l}\\
 & =(1+|k|^{2})^{-I-2}\sum_{l=1}^{\infty}\inn{\varsigma_{k},B(s)h_{l}}h_{l}\\
 & =2\pi(1+|k|^{2})^{-I-2}\sum_{j=1}^{d}\sum_{l=1}^{\infty}\binn{\varsigma_{k},\partial_{j}\left(\sqrt{\Phi(\rho_{s})}h_{l}^{j}\right)}h_{l}\\
 & =-2\pi(1+|k|^{2})^{-I-2}\sum_{j=1}^{d}\sum_{l=1}^{\infty}\i k_{j}\binn{\varsigma_{k}\sqrt{\Phi(\rho_{s})},h_{l}^{j}}h_{l}\\
 & =\left(-2\pi\i(1+|k|^{2})^{-I-2}k_{j}\varsigma_{k}\sqrt{\Phi(\rho_{s})}\right)_{j\in[d]}.
\end{align*}
Thus, 
\[
B(s)B^{*}(s)\varsigma_{k}=-4\pi^{2}(1+|k|^{2})^{-I-2}\i\sum_{j=1}^{d}k_{j}\partial_{j}\left(\varsigma_{k}\Phi(\rho_{s})\right)
\]
and, consequently,
\begin{align*}
\tr\left[U(s)\right] & =-4\pi^{2}\sum_{j=1}^{d}\sum_{k\in\Z^{d}}\i k_{j}D_{2}^{2}F_{s}(\rho_{s},\zeta_{s})\left[\partial_{j}\left(\varsigma_{k}\Phi(\rho_{s})\right),\varsigma_{-k}\right]\\
 & =4\pi^{2}\sum_{j=1}^{d}\sum_{k\in\Z^{d}}D_{2}^{2}F_{s}(\rho_{s},\zeta_{s})\left[\partial_{j}\left(\varsigma_{k}\Phi(\rho_{s})\right),\partial_{j}\varsigma_{-k}\right].
\end{align*}
Using the expansion of $\Phi(\rho_{s})$ in the Fourier series
\begin{align*}
\Phi(\rho_{s}) & =\sum_{l\in\Z^{d}}\inn{\Phi(\rho_{s}),\varsigma_{l}}\varsigma_{l}=\sum_{l\in\Z^{d}}\inn{\varsigma_{-l},\Phi(\rho_{s})}\varsigma_{l}\\
 & =\sum_{l\in\Z^{d}}\inn{\varsigma_{l},\Phi(\rho_{s})}\varsigma_{-l},
\end{align*}
we get
\begin{align*}
\tr\left[U(s)\right] & =4\pi^{2}\sum_{j=1}^{d}\sum_{k\in\Z^{d}}\sum_{l\in\Z^{d}}D_{2}^{2}F_{s}(\rho_{s},\zeta_{s})\left[\partial_{j}\varsigma_{k-l},\partial_{j}\varsigma_{-k}\right]\inn{\varsigma_{l},\Phi(\rho_{s})}\\
 & =4\pi^{2}\sum_{j=1}^{d}\binn{\sum_{k\in\Z^{d}}\sum_{l\in\Z^{d}}D_{2}^{2}F_{s}(\rho_{s},\zeta_{s})\left[\partial_{j}\varsigma_{k-l},\partial_{j}\varsigma_{-k}\right]\varsigma_{l},\Phi(\rho_{s})}\\
 & =4\pi^{2}\sum_{j=1}^{d}\binn{\Tr\left(\partial_{j}^{\otimes2}D_{2}^{2}F_{s}(\rho_{s},\zeta_{s})\right),\Phi(\rho_{s})},
\end{align*}
according to Lemma \ref{lem:tr_operator}. This completes the proof
of the lemma.
\end{proof}

\subsection{Differentiability of the Ornstein-Uhlenbeck semigroup\protect\label{subsec:Differentiability-of-semigroup}}

Let $(\rho_{t}^{\infty})_{t\ge0}$ and $(\eta_{t}^{\infty})_{t\ge0}$
be solutions to (\ref{eq:heat_PDE-1}) and $(\ref{eq:SPDE_for_OU_process-1})$
with a bounded continuous function $\Phi:\R\to[0,\infty)$, respectively.
In this section, we consider these processes as functions of their
initial conditions $\rho:=\rho_{0}^{\infty}$ and $\zeta:=\zeta_{0}^{\infty}$
and study the differentiability of
\[
U_{t}^{\Phi}(\rho,\zeta):=\E F(\rho_{t}^{\infty},\zeta_{t}^{\infty})
\]
with respect to $(\rho,\zeta)$ for $F\in\Cf_{l,HS}^{1,3}(\H[J],\H[-I])$.

The fact that $U^{\Phi}$ is three times continuously differentiable
with respect to $\zeta$ directly follows from the linearity of $\zeta^{\infty}$
in $\zeta$, see the proof of Proposition \ref{prop:fokker-plank-equation}
below. Hence, the main challenge is the regularity of $U^{\Phi}$
with respect to $\rho$. The main difficulty is that the diffusion
term $\sqrt{\Phi(\rho)}$ is not differentiable, and, therefore, we
cannot follow the usual approach to conclude the differentiability
of $U^{\Phi}$ from the differentiability of the solution $\zeta^{\infty}$
to the SPDE $(\ref{eq:SPDE_for_OU_process-1})$ as function of its
initial condition. This is solved in this section by exploiting the
Gaussianity of $\zeta^{\infty}$ together with an infinite-dimensional
integration-by-parts formula. 

We start from the following auxiliary statements. 
\begin{lem}
\label{lem:differentiability_of_variance_V}Let $I>\frac{d}{2}+1$,
$J>\frac{d}{2}$, $\zeta\in\H[-I]$ be fixed and $\Phi\in\Cf_{b}^{2}(\R)$.
Let also $(\zeta_{t}^{\infty})_{t\ge0}$ be a solution to $(\ref{eq:SPDE_for_OU_process-1})$
started from $\zeta$, where $(\rho_{t}^{\infty})_{t\ge0}$ is a solution
to the heat equation (\ref{eq:heat_PDE-1}) with the initial condition
$\rho_{0}^{\infty}=\rho\in\H[J]$. Then for each $t>0$ the covariance
$V_{t}(\rho)$ of $\zeta_{t}^{\infty}$ can be extended to an element
in $\MLOHS[2][{\H[I]}]$ also denoted by $V_{t}(\rho)$. Moreover,
the map $V_{t}$ belongs to $\Cf_{b}^{1}\left(\H[J];\MLOHS[2][{\H[I]}]\right)$
and its derivative at $\rho\in\H[J]$ in direction $h\in\H[J]$ is
given by
\begin{equation}
DV_{t}(\rho)[h][\varphi,\psi]=2\pi^{2}\int_{0}^{t}\binn{\nabla P_{t-s}\varphi\cdot\nabla P_{t-s}\psi,\Phi'\left(P_{s}\rho\right)P_{s}h}ds\label{eq:derivative_of_covariance_operator}
\end{equation}
for all $\varphi,\psi\in\H[J]$ and 
\begin{equation}
\norm{DV_{t}(\rho)[h]}_{\MLOHS[2][{\H[I]}]}\le tC_{I}\norm{\Phi'}_{\Cf}\norm h_{H_{J}}.\label{eq:estimate_of_norm_of_DV}
\end{equation}
\end{lem}

\begin{proof}
Using Hölder's inequality and Proposition \ref{prop:well_possedness_of_OUP_SSEP},
we get
\[
\left|V_{t}(\rho)[\varphi,\psi]\right|\le\norm{\varphi}_{H_{I}}\norm{\psi}_{H_{I}}\E\left[\norm{\zeta_{t}^{\infty}}_{H_{-I}}^{2}\right]\le C_{\rho,I,\zeta}\norm{\varphi}_{H_{I}}\norm{\psi}_{H_{I}}
\]
for all $\varphi,\psi\in\Cf^{\infty}(\Td)$. This implies that $V_{t}(\rho)$
can be extended to a continuous multilinear operator on $\left(\H[I]\right)^{2}$.
Using Proposition \ref{prop:well_possedness_of_OUP_SSEP} again, following
the proof of Lemma \ref{lem:whide_continuity_of_zeta} and applying
the estimate (\ref{eq:estimate_for_derivative_of_V}), we can show
the boundedness of the Hilbert-Schmidt norm of $V_{t}(\rho)$ given
by
\begin{align}
\norm{V_{t}(\rho)}_{\cL_{2}^{HS}}^{2} & \le C_{I}t\sup_{s\in[0,t]}\norm{\Phi(\rho_{s}^{\infty})}^{2}\le C_{I}t\norm{\Phi}_{\Cf}^{2}.\label{eq:estimate_of_HS_norm_of_V}
\end{align}

To get the (Lipschitz) continuity of $V_{t}:\H[J]\to\MLOHS[2][{\H[I]}],$
we can also follow the proof of Lemma \ref{lem:whide_continuity_of_zeta}
and use the estimate (\ref{eq:estimate_for_derivative_of_V}) to get
for $\rho,\tilde{\rho}\in\H[J]$
\begin{align*}
\norm{V_{t}(\rho)-V_{t}(\tilde{\rho})}_{\cL_{2}^{HS}}^{2} & \le C_{I}t\sup_{s\in[0,t]}\bnorm{\Phi(P_{s}\rho)-\Phi(P_{s}\tilde{\rho})}\\
 & \le\norm{\Phi'}_{\Cf}\bnorm{P_{s}\rho-P_{s}\tilde{\rho}}\le\norm{\Phi'}_{\Cf}\norm{\rho-\tilde{\rho}}\\
 & \le\norm{\Phi'}_{\Cf}\norm{\rho-\tilde{\rho}}_{H_{J}}.
\end{align*}

We next check the differentiability of $V_{t}$ at $\rho\in\H[J]$
and show that its derivative is given by
\[
DV_{t}(\rho)[h][\varphi,\psi]=2\pi^{2}\int_{0}^{t}\binn{\nabla P_{t-s}\varphi\cdot\nabla P_{t-s}\psi,\Phi'(P_{s}\rho)P_{s}h}ds.
\]
Similarly as above, we estimate
\begin{align*}
 & \norm{V_{t}(\rho+h)-V_{t}(\rho)-DV_{t}(\rho)[h]}_{\cL_{2}^{HS}}^{2}\\
 & \qquad=\sum_{k,l\in\Z^{d}}(1+|k|^{2})^{-I}(1+|l|^{2})^{-I}\left|V_{t}(\rho+h)[\tilde{\varsigma}_{k},\tilde{\varsigma}_{l}]-V_{t}(\rho)[\tilde{\varsigma}_{k},\tilde{\varsigma}_{l}]-DV_{t}(\rho)[h][\tilde{\varsigma}_{k},\tilde{\varsigma}_{l}]\right|^{2}\\
 & \qquad\le tC_{I}\sup_{s\in[0,t]}\norm{\Phi\left(P_{s}\rho+P_{s}h\right)-\Phi\left(P_{s}\rho\right)-\Phi'(P_{s}\rho)P_{s}h}^{2}\le tC_{I}\norm{\Phi''}_{\Cf}\norm{(P_{s}h)^{2}}^{2}\\
 & \qquad\le tC_{J,I}\norm{\Phi''}_{\Cf}\norm h_{H_{J}}^{4},
\end{align*}
where we used Taylor's expansion for $\Phi$, (\ref{eq:estimate_for_derivative_of_V})
and
\[
\bnorm{(P_{s}h)^{2}}^{2}\le\bnorm{\left(P_{s}h\right)^{2}}_{\Cf}^{2}=\norm{P_{s}h}_{\Cf}^{4}\le C_{J}\norm h_{\Cf}^{4}\le C_{J}\norm h_{H_{J}}^{4}
\]
due to $J>\frac{d}{2}$.

The boundedness of $DV_{t}$ follows from (\ref{eq:derivative_of_covariance_operator})
and (\ref{eq:estimate_for_derivative_of_V}). Indeed, for each $h\in\H[J]$,
one has
\begin{align*}
\norm{DV_{t}(\rho)[h]}_{\cL_{2}^{HS}}^{2} & \le4\pi^{4}t\sum_{k,l\in\Z^{d}}(1+|k|^{2})^{-I}(1+|l|^{2})^{-I}\\
 & \qquad\qquad\cdot\int_{0}^{t}\left|\binn{\nabla P_{t-s}\tilde{\varsigma}_{k}\cdot\nabla P_{t-s}\tilde{\varsigma}_{l},\Phi'(P_{s}\rho)P_{s}h}\right|^{2}ds\\
 & \le tC_{I}\sup_{s\in[0,t]}\norm{\Phi'(P_{s}\rho)P_{s}h}^{2}\le tC_{I}\norm{\Phi'}_{\Cf}\norm{P_{s}h}^{2}\\
 & =tC_{I}\norm{\Phi'}_{\Cf}\norm h^{2}\le tC_{I}\norm{\Phi'}_{\Cf}\norm h_{H_{J}}^{2}.
\end{align*}
The continuity of $DV_{t}$ can be checked similarly. This completes
the proof of the statement.
\end{proof}
\begin{lem}
\label{lem:uniform_convergenc_of_proj_of_zeta}Under the assumptions
of Lemma \ref{lem:differentiability_of_variance_V}, one has
\[
\sup_{\rho\in\H[J]}\E\left[\norm{\pr_{n}\zeta_{t}^{\infty}-\zeta_{t}^{\infty}}_{H_{-I}}^{2}\right]\to0
\]
as $n\to\infty$. 
\end{lem}

\begin{proof}
We rewrite
\begin{align*}
\E\left[\norm{\pr_{n}\zeta_{t}^{\infty}-\zeta_{t}^{\infty}}_{H_{-I}}^{2}\right] & =\sum_{k\not\in\Znd}(1+|k|^{2})^{-I}\E\left[\inn{\zeta_{t}^{\infty},\tilde{\varsigma}_{k}}^{2}\right]\\
 & =\sum_{k\not\in\Znd}(1+|k|^{2})^{-I}\left(V_{t}(\rho)[\tilde{\varsigma}_{k},\tilde{\varsigma}_{k}]+\E\left[\left\langle \tilde{\varsigma}_{k},\zeta_{t}^{\infty}\right\rangle \right]^{2}\right).
\end{align*}
Using Proposition \ref{prop:well_possedness_of_OUP_SSEP} and following
the proof of Lemma \ref{lem:whide_continuity_of_zeta}, in particular
(\ref{eq:estimate_for_derivative_of_V}), the expression above can
be estimated as follows
\begin{align*}
 & \sum_{k\not\in\Znd}(1+|k|^{2})^{-I}\left(V_{t}(\rho)[\tilde{\varsigma}_{k},\tilde{\varsigma}_{k}]+\left(m_{t}(\zeta;\tilde{\varsigma}_{k})\right)^{2}\right)\\
 & \qquad\le\sum_{k\not\in\Znd}(1+|k|^{2})^{-I}\left(\pi^{2}t\norm{\Phi}_{\Cf}\frac{|k|^{4}}{2|k|^{2}}+\inn{\zeta,P_{t}\tilde{\varsigma}_{k}}^{2}\right)\\
 & \qquad\le\pi^{2}t\norm{\Phi}_{\Cf}\sum_{k\not\in\Znd}(1+|k|^{2})^{-I+1}+\sum_{k\not\in\Znd}(1+|k|^{2})^{-I}e^{-8\pi^{2}|k|^{2}t}\inn{\zeta,\tilde{\varsigma}_{k}}^{2}.
\end{align*}
This implies the uniform convergence of $\E\left[\norm{\pr_{n}\zeta_{t}^{\infty}-\zeta_{t}^{\infty}}_{H_{-I}}^{2}\right]$
to zero as $n\to\infty$. 
\end{proof}
Define for $A\in\MLOHS[2][{\H[-I]}]$ and $B\in\MLOHS[2][{\H[I]}]$
\[
A:B:=\sum_{k,l\in\Z^{d}}A[\tilde{\varsigma}_{k},\tilde{\varsigma}_{l}]B[\tilde{\varsigma}_{k},\tilde{\varsigma}_{l}]
\]
and note that 
\[
|A:B|\le\norm A_{\MLOHS[2][H_{-I}]}\norm B_{\MLOHS[2][H_{I}]},
\]
according to (\ref{eq:norm_of_:}).
\begin{prop}
\label{prop:differentiability_of_semigroup_for_OU}Let $I>\frac{d}{2}+1$,
$J>\frac{d}{2}$ and $\zeta\in\H[-I]$ be fixed. Let also $(\zeta_{t}^{\infty})_{t\ge0}$
be a solution to $(\ref{eq:SPDE_for_OU_process-1})$ started from
$\zeta$, where $(\rho_{t}^{\infty})_{t\ge0}$ is a solution to the
heat equation (\ref{eq:heat_PDE-1}) with the initial condition $\rho_{0}^{\infty}=\rho\in\H[J]$.
Then for each $F\in\Cf_{l}^{2}(\H[-I])$ with bounded uniformly continuous
second derivative in the space $\MLOHS[2][{\H[-I]}]$ and $t\ge0$
the function
\[
U_{t}(\rho):=\E F(\zeta_{t}^{\infty}),\quad\rho\in\H[J],
\]
belongs to $\Cf_{l}^{1}(\H[J])$ and for each $\rho\in\H[J]$ and
$t>0$
\begin{equation}
DU_{t}(\rho)[h]=\frac{1}{2}\E\left[D^{2}F(\zeta_{t}^{\infty}):DV_{t}(\rho)\left[h\right]\right],\quad h\in\H[J],\label{eq:expression_of_DU}
\end{equation}
where $V_{t}(\rho)$ is the covariance operator of $\zeta_{t}^{\infty}$
defined by (\ref{eq:eq:definition_of_covariance_operator}).
\end{prop}

\begin{proof}
Let $t>0$ be fixed. We will show the differentiability of $U_{t}$
on $\H[J]$, using the differentiability of the variance $V_{t}$
that follows from Lemma \ref{lem:differentiability_of_variance_V}.
Define the sequence of functions
\[
U^{n}(\rho):=\E F\left(\pr_{n}\zeta_{t}^{\infty}\right),\quad n\ge1,
\]
and show that they are continuously differentiable on $\H[J]$ and
their derivatives converge uniformly to a continuous function $\tilde{U}$.
By \citep[Theorem 3.6.1]{Cartan:1971}, we will conclude that $U_{t}\in\Cf^{1}(\H[J])$
and $DU=\tilde{U}.$

Setting $\xi^{n}:=\left(\inn{\zeta_{t}^{\infty},\tilde{\varsigma}_{k}}-m_{k}\right)_{k\in\Znd}$
for $m_{k}=\E\inn{\zeta_{t}^{\infty},\tilde{\varsigma}_{k}}$, we
can represent $U^{n}$ as follows
\[
U^{n}(\rho)=\E f_{n}\left(m^{n}+\xi^{n}\right),
\]
where
\[
f_{n}(z)=F\left(\chi_{n}(z)\right),\quad z\in\R^{\Znd},
\]
$\chi_{n}(z):=\sum_{k\in\Z_{n}^{d}}z_{k}\tilde{\varsigma}_{k}$ and
$m^{n}:=(m_{k})_{k\in\Znd}$. Note that $\xi^{n}:=\left(\inn{\zeta_{t}^{\infty},\tilde{\varsigma}_{k}}-m_{k}\right)_{k\in\Znd}$
is a centered Gaussian vector with covariance matrix
\[
V^{n}(\rho):=\left(V_{t}(\rho)[\tilde{\varsigma}_{k},\tilde{\varsigma}_{l}]\right)_{k,l\in\Z_{n}^{d}}
\]
that is non-negatively defined and symmetric. By the differentiability
of $F$, the function $f_{n}$ belongs to $\Cf_{l}^{2}(\R^{\Znd})$
and 
\begin{equation}
\frac{\partial f_{n}}{\partial z_{k}}=DF(\chi_{n})[\tilde{\varsigma}_{k}],\quad\frac{\partial^{2}f_{n}}{\partial z_{k}\partial z_{l}}=D^{2}F(\chi_{n})[\tilde{\varsigma}_{k},\tilde{\varsigma}_{l}],\quad k,l\in\Znd.\label{eq:derivative_of_f_n}
\end{equation}
Using the spectral decomposition theorem, there exists a square-root
$\sqrt{V^{n}(\rho)}$ of $V^{n}(\rho)$, that is a (unique) non-negatively
defined symmetric matrix such that $(\sqrt{V^{n}(\rho)})^{2}=V^{n}(\rho)$.
Thus, we can define $\xi^{n}=\sqrt{V^{n}(\rho)}\tilde{\xi}^{n}$ for
a standard Gaussian vector $\tilde{\xi^{n}}=(\tilde{\xi}_{k})_{k\in\Znd}$.
Therefore, the differentiability of $U^{n}$ will follow from the
differentiability of $\rho\mapsto\E f_{n}\big(m^{n}+\sqrt{V^{n}(\rho)}\tilde{\xi}^{n}\big)$.

Let $\Her_{n}$ denote the Hilbert space of symmetric matrices $(A_{k,l})_{k,l\in\Znd}$
with real entries and be equipped with the inner product 
\[
A:B:=\sum_{k,l\in\Znd}A_{k,l}B_{k,l}.
\]
The open subset of positively defined matrices from $\Her_{n}$ will
be denoted by $\Her_{n}^{+}.$ Note that the square-root function
$\sqrt{\cdot}:\Her_{n}^{+}\to\Her_{n}^{+}$ is continuously differentiable
and its derivative $(D\sqrt{A})[B]$ in a direction $B\in\Her_{n}$
satisfies
\begin{equation}
\left(D\sqrt{A}\right)[B]\sqrt{A}+\sqrt{A}\left(D\sqrt{A}\right)[B]=B,\label{eq:derivative_of_square_root}
\end{equation}
according to the expression for the derivative of the product $A=\sqrt{A}\sqrt{A}$.
We next consider for each $\delta>0$ a continuously differentiable
function $G_{\delta}(A):=\delta I+A$, $A\in\Her_{n}^{-\delta}$,
with values in $\Her_{n}^{+}$, where $I$ is the identity matrix
and 
\[
\Her_{n}^{-\delta}:=\left\{ A\in\Her_{n}:\ Ax\cdot x>-\delta|x|^{2},\quad x\in\R^{\Znd}\backslash\{0\}\right\} 
\]
is an open subset of $\Her_{n}$. Then $G_{\delta}\in\Cf^{1}(\Her_{n}^{-\delta};\Her_{n}^{+})$
and, consequently, the function
\[
K_{n,\delta}(A):=\E f_{n}\left(m^{n}+\sqrt{G_{\delta}(A)}\tilde{\xi}^{n}\right),\quad A\in\Her_{n}^{-\delta},
\]
is continuously differentiable with derivative in a direction $B\in\Her_{n}$
given by
\[
DK_{n,\delta}(A)[B]=\E\left[Df_{n}\left(m^{n}+\sqrt{G_{\delta}(A)}\tilde{\xi}^{n}\right)\cdot\left(\left(D\sqrt{G_{\delta}}(A)\right)[B]\tilde{\xi}^{n}\right)\right]
\]
due to \citep[Theorem 2.2.1]{Cartan:1971} and the dominated convergence
theorem. Using the integration-by-parts formula for a Gaussian vector
(see Lemma \ref{lem:integration-by-parts_for_gauss}), we get
\[
DK_{n,\delta}(A)[B]=\E\left[D^{2}f_{n}\left(m^{n}+\sqrt{G_{\delta}(A)}\tilde{\xi}^{n}\right):\left(\left(D\sqrt{G_{\delta}}(A)\right)[B]\sqrt{G_{\delta}(A)}\right)\right].
\]
Next, by the equality $A:(BR)=A:(RB)$ for $A,B,R\in\Her_{n}$ and
(\ref{eq:derivative_of_square_root}), we have
\begin{align}
DK_{n,\delta}(A)[B] & =\frac{1}{2}\E\left[D^{2}f_{n}\left(m^{n}+\sqrt{G_{\delta}(A)}\tilde{\xi}^{n}\right):\left(\left(D\sqrt{G_{\delta}}(A)\right)[B]\sqrt{G_{\delta}(A)}\right)\right]\nonumber \\
 & +\frac{1}{2}\E\left[D^{2}f_{n}\left(m^{n}+\sqrt{G_{\delta}(A)}\tilde{\xi}^{n}\right):\left(\sqrt{G_{\delta}(A)}\left(D\sqrt{G_{\delta}}(A)\right)[B]\right)\right]\label{eq:derivative_DK_delta}\\
 & =\frac{1}{2}\E\left[D^{2}f_{n}\left(m^{n}+\sqrt{G_{\delta}(A)}\tilde{\xi}^{n}\right):B\right]\nonumber 
\end{align}
for all $A\in\Her_{n}^{-\delta}$ and $B\in\Her_{n}$. By the differentiability
of the composition and the expression (\ref{eq:derivative_DK_delta}),
we conclude that the function $\E f_{n}\left(m^{n}+\sqrt{\delta I+V^{n}}\tilde{\xi}^{n}\right)$
is continuously differentiable and
\begin{align*}
DU^{n,\delta}(\rho) & :=D\E f_{n}\left(m^{n}+\sqrt{\delta I+V^{n}(\rho)}\tilde{\xi}^{n}\right)\\
 & =\frac{1}{2}\E\left[D^{2}f_{n}\left(m^{n}+\sqrt{\delta I+V^{n}(\rho)}\tilde{\xi}^{n}\right):DV^{n}(\rho)\right],\quad\rho\in\H[J],
\end{align*}
for all $\delta>0$. Now, taking $\delta\to0+$, and using \citep[Theorem 3.6.1]{Cartan:1971}
and the dominated convergence theorem, we get that $U^{n}\in\Cf^{1}\left(\H\right)$
and 
\begin{align*}
DU^{n}(\rho)[h] & =D\E f_{n}\left(m^{n}+\sqrt{V^{n}(\rho)}\tilde{\xi}^{n}\right)[h]\\
 & =\frac{1}{2}\E\left[D^{2}f_{n}\left(m^{n}+\xi^{n}\right):DV^{n}(\rho)[h]\right],\quad\rho,h\in\H[J].
\end{align*}
Note that the assumptions of \citep[Theorem 3.6.1]{Cartan:1971} require
the uniform convergence of the sequence $DU^{n,\delta}$ to $DU^{n}$
as $\delta\to0$. We will show this property for a more complicated
sequence of derivatives at the end of this proof. The uniform convergence
in the present case can be obtained similarly. 

In order to show the differentiability of $U_{t}$, we will use \citep[Theorem 3.6.1]{Cartan:1971}
again. We first note that 
\[
U^{n}(\rho)\to U_{t}(\rho)
\]
as $n\to\infty$ for each $\rho\in\H$, by the dominated convergence
theorem and the fact that $\pr_{n}\zeta_{t}^{\infty}\to\zeta_{t}^{\infty}$
a.s. in $\H[-I]$ as $n\to\infty$. We will next rewrite the derivative
$DU_{n}$ via the derivatives $D^{2}F$ and $DV_{t}$ in the corresponding
spaces. Using (\ref{eq:derivative_of_f_n}) and 
\[
DV_{k,l}^{n}(\rho)[h]=DV_{t}(\rho)[h][\tilde{\varsigma}_{k},\tilde{\varsigma}_{l}],
\]
we obtain
\begin{align*}
DU^{n}(\rho)[h] & =\frac{1}{2}\sum_{k,l\in\Znd}\E\left[D^{2}F(\pr_{n}\zeta_{t}^{\infty})[\tilde{\varsigma_{k}},\tilde{\varsigma_{l}}]DV_{t}(\rho)[h][\tilde{\varsigma}_{k},\tilde{\varsigma}_{l}]\right]\\
 & =\frac{1}{2}\E\left[D^{2}F(\pr_{n}\zeta_{t}^{\infty}):\pr_{n}^{\otimes2}DV_{t}(\rho)[h]\right]
\end{align*}
for all $\rho,h\in\H[J]$. We next note that $D^{2}F(\zeta)\in\MLOHS[2][{\H[-I]}]$
and $DV_{t}(\rho)[h]\in\MLOHS[2][{\H[I]}]$, by Lemma \ref{lem:differentiability_of_variance_V}.
Hence, $D^{2}F(\zeta):DV_{t}(\rho)[h]$ is well defined for each $\zeta\in\H[-I]$,
$\rho\in\H[J]$ and $h\in\H[J]$. We will show that $DU^{n}\to DU$
uniformly. Using (\ref{eq:norm_of_:}), we get
\begin{align*}
 & \left|DU^{n}(\rho)[h]-DU(\rho)[h]\right|^{2}\le\frac{1}{2}\E\left[\left|D^{2}F(\pr_{n}\zeta_{t}^{\infty}):\left(\pr_{n}^{\otimes2}DV_{t}(\rho)[h]-DV_{t}(\rho)[h]\right)\right|^{2}\right]\\
 & \qquad+\frac{1}{2}\E\left[\left|\left(D^{2}F(\pr_{n}\zeta_{t}^{\infty})-D^{2}F(\zeta_{t}^{\infty})\right):V_{t}(\rho)[h]\right|^{2}\right]\\
 & \qquad\le\frac{1}{2}\bnorm{\pr_{n}^{\otimes2}DV_{t}(\rho)[h]-DV_{t}(\rho)[h]}_{\MLOHS[2][{\H[I]}]}^{2}\E\left[\bnorm{D^{2}F(\pr_{n}\zeta_{t}^{\infty})}_{\MLOHS[2][{\H[-I]}]}^{2}\right]\\
 & \qquad+\frac{1}{2}\bnorm{DV_{t}(\rho)[h]}_{\MLOHS[2][{\H[I]}]}^{2}\E\left[\bnorm{D^{2}F(\pr_{n}\zeta_{t}^{\infty})-D^{2}F(\zeta_{t}^{\infty})}_{\MLOHS[2][{\H[-I]}]}^{2}\right]
\end{align*}
 for all $h,\rho\in\H[J]$. By (\ref{eq:estimate_of_norm_of_DV}),
\begin{align*}
\bnorm{DV(\rho)[h]}_{\MLOHS[2][{\H[I]}]}^{2} & \le tC_{I}\norm{\Phi'}_{\Cf}\norm h_{H_{J}}^{2},\quad\rho,h\in\H[J].
\end{align*}
Moreover, similarly to the proof of (\ref{eq:estimate_of_norm_of_DV}),
we conclude 
\begin{align*}
\bnorm{\tilde{V}^{n}(\rho)[h]-V(\rho)[h]}_{\MLOHS[2][{\H[I]}]}^{2} & =\sum_{k,l\not\in\Znd}(1+|k|^{2})^{-I}(1+|l|^{2})^{-I}\left|DV(\rho)[h][\tilde{\varsigma}_{k},\tilde{\varsigma}_{l}]\right|^{2}\\
 & \le\sum_{k,l\not\in\Znd}(1+|k|^{2})^{I+1}(1+|l|^{2})^{-I+1}t\norm{\Phi'}_{\Cf}\norm h_{H_{J}}^{2}\\
 & =\eps_{n}t\norm{\Phi'}_{\Cf}\norm h_{H_{J}}^{2},\quad\rho,h\in\H[J],
\end{align*}
where $\eps_{n}:=\sum_{k,l\not\in\Znd}(1+|k|^{2})^{-I+1}(1+|l|^{2})^{-I+1}\to0$
as $n\to\infty$, due to $I>\frac{d}{2}+1$.

We next fix arbitrary $\eps>0$ and choose $\delta>0$ such that 
\[
\norm{D^{2}F(\zeta)-D^{2}F(\zeta')}_{\MLOHS[2][{\H[-I]}]}<\eps
\]
for all $\zeta,\zeta'\in\H[-I]$ satisfying $\norm{\zeta-\zeta'}_{\H[-I]}<\delta$,
according to the uniform continuity of $D^{2}F$. We also take $N\in\N$
such that for all $n\ge N$
\[
\sup_{\rho\in\H[J]}\E\left[\norm{\pr_{n}\zeta_{t}^{\infty}-\zeta_{t}^{\infty}}_{H_{-I}}^{2}\right]<\eps,
\]
by Lemma \ref{lem:uniform_convergenc_of_proj_of_zeta}. Then, using
Chebyshev's inequality, we get
\begin{align*}
 & \E\left[\bnorm{D^{2}F(\pr_{n}\zeta_{t}^{\infty})-D^{2}F(\zeta_{t}^{\infty})}_{\MLOHS[2][{\H[-I]}]}^{2}\right]\\
 & \qquad=\E\left[\bnorm{D^{2}F(\pr_{n}\zeta_{t}^{\infty})-D^{2}F(\zeta_{t}^{\infty})}_{\MLOHS[2][{\H[-I]}]}^{2}\I_{\left\{ \norm{\pr_{n}\zeta_{t}^{\infty}-\zeta_{t}^{\infty}}_{H_{-I}}\ge\delta\right\} }\right]\\
 & \qquad+\E\left[\bnorm{D^{2}F(\pr_{n}\zeta_{t}^{\infty})-D^{2}F(\zeta_{t}^{\infty})}_{\MLOHS[2][{\H[-I]}]}^{2}\I_{\left\{ \norm{\pr_{n}\zeta_{t}^{\infty}-\zeta_{t}^{\infty}}_{H_{-I}}\le\delta\right\} }\right]\\
 & \qquad\le\frac{4}{\delta^{2}}\sup_{\zeta\in\H[-I]}\norm{D^{2}F(\zeta)}_{\MLOHS[2][{\H[-I]}]}^{2}\E\left[\norm{\pr_{n}\zeta_{t}^{\infty}-\zeta_{t}^{\infty}}_{H_{-I}}^{2}\right]+\eps^{2}\\
 & \qquad\le\frac{4}{\delta^{2}}\sup_{\zeta\in\H[-I]}\norm{D^{2}F(\zeta)}_{\MLOHS[2][{\H[-I]}]}^{2}\eps+\eps^{2}.
\end{align*}
This shows that 
\[
\sup_{\rho\in\H[J]}\E\left[\bnorm{D^{2}F(\pr_{n}\zeta_{t}^{\infty})-D^{2}F(\zeta_{t}^{\infty})}_{\MLOHS[2][{\H[-I]}]}^{2}\right]\to0
\]
as $n\to\infty$. Consequently, 
\[
\sup_{\rho\in\H[J]}\bnorm{DU^{n}(\rho)-DU(\rho)}_{H_{-J}}=\sup_{\rho\in\H[J]}\sup_{\norm h_{J}\le1}\left|DU^{n}(\rho)[h]-DU(\rho)[h]\right|\to0.
\]

The boundedness of $DU$ follows from the expression (\ref{eq:expression_of_DU}),
(\ref{eq:norm_of_:}) and Lemma \ref{lem:differentiability_of_variance_V}.
This completes the proof of the proposition.
\end{proof}
We next define for a function $F\in\Cf_{l,HS}^{1,2}(\H[J],\H[-I])$
the differential operator
\begin{align*}
\cG^{OU,\Phi}F(\rho,\zeta) & =2\pi^{2}\binn{\Delta D_{1}F(\rho,\zeta),\rho}+2\pi^{2}\binn{\Delta D_{2}F(\rho,\zeta),\zeta}\\
 & +2\pi^{2}\sum_{j=1}^{d}\binn{\Tr\left(\partial_{j}^{\otimes2}D_{2}^{2}F(\rho,\zeta)\right),\Phi(\rho)},\quad\rho\in\H[J+2],\ \ \zeta\in\H[-I+2].
\end{align*}

\begin{prop}
\label{prop:fokker-plank-equation} Let $I>\frac{d}{2}+1$, $J>\frac{d}{2}$,
$\Phi\in\Cf_{b}^{2}(\R)$, $F\in\Cf_{l}^{2,4}(\H[J],\H[-I])$ and
$D_{2}^{2}F$ be bounded and uniformly continuous in $\MLOHS[2][{\H[-I]}]$.
Let also
\[
U_{t}(\rho,\zeta):=\E F(\rho_{t}^{\infty},\zeta_{t}^{\infty}),\quad t\ge0,
\]
for $\rho\in\H[J]$, $\zeta\in\H[-I]$, where $(\rho_{t}^{\infty},\zeta_{t}^{\infty})$,
$t\ge0$, is a solution in $\H[J]\times\H[-I]$ to (\ref{eq:heat_PDE-1}),
(\ref{eq:SPDE_for_OU_process-1}) started from $(\rho,\zeta)$. Then
the function $U$ belongs to $\Cf^{0,1,3}([0,\infty),H_{J},H_{-I})$
and for each $T>0$
\begin{equation}
\sup_{t\in[0,T]}\norm{U_{t}}_{\Cf_{l,HS}^{1,3}}\le\Cf_{I,T}\left(\norm{\Phi'}_{\Cf}+1\right)\norm F_{\Cf_{l,HS}^{1,3}}.\label{eq:uniform_estimate_of_norm_of_U}
\end{equation}
Moreover, if $I>\frac{d}{2}+3$, then for each $\rho\in\H[J+2]$ and
$\zeta\in\H[-I+2]$ the map $t\mapsto U_{t}(\rho,\zeta)$ is continuously
differentiable, 
\begin{equation}
\partial U_{t}(\rho,\zeta)=\cG^{OU,\Phi}U_{t}(\rho,\zeta),\quad t>0,\ \rho\in\H[J+2],\ \zeta\in\H[-I+2],\label{eq:kolmogorov_equation}
\end{equation}
and $\partial U\in\Cf\left([0,\infty)\times\H[J+2]\times H_{-I+2}\right)$.
\end{prop}

\begin{proof}
To prove the proposition, we will split the dependence of $\rho_{t}^{\infty}$
and $\zeta_{t}^{\infty}$ on the initial condition for the heat equation
(\ref{eq:heat_PDE-1}), extending $U_{t}$ by
\[
\tilde{U}_{t}(\tilde{\rho},\rho,\zeta):=\E F(\tilde{\rho}_{t}^{\infty},\zeta_{t}^{\infty}),
\]
where $(\tilde{\rho}_{t}^{\infty})_{t\ge0}$ is a solution to (\ref{eq:heat_PDE-1})
started from $\tilde{\rho}\in\H[J]$ and $(\zeta_{t}^{\infty})_{t\ge0}$
is a solution to (\ref{eq:SPDE_for_OU_process-1}) started from $\zeta\in\H[-I]$
with the diffusion coefficient depending on the solution $(\rho_{t}^{\infty})_{t\ge0}$
to the heat equation (\ref{eq:heat_PDE-1}) started from $\rho\in\H[J]$.
Then
\[
U_{t}(\rho,\zeta)=\tilde{U}_{t}(\rho,\rho,\zeta)
\]
for all $t\ge0$, $\rho\in\H[J]$ and $\zeta\in\H[-I]$. 

The continuity of $\tilde{U}$ directly follows from the mean-value
theorem (see \citep[Theorem 3.3.2]{Cartan:1971}), the continuity
$(t,\tilde{\rho})\mapsto\tilde{\rho}_{t}^{\infty}$ as a map from
$[0,\infty)\times\H[J]$ to $\H[J]$ and the continuity of $(t,\rho,\zeta)\mapsto\law\zeta_{t}^{\infty}$
in the 2-Wasserstein topology as a map from $[0,\infty)\times\H[J]\times\H[-I]$
to the space of probability distributions on $\H[-I]$ with a finite
second moment, by Lemma \ref{lem:whide_continuity_of_zeta}. 

We next show that $\tilde{U_{t}}$ is differentiable on $\H[-I]$
with respect to the third variable and its derivative in a direction
$h\in\H[-I]$ equals
\begin{equation}
D_{3}\tilde{U}_{t}(\tilde{\rho},\rho,\zeta)[h]=\E\left[D_{2}F(\tilde{\rho}_{t}^{\infty},\zeta_{t}^{\infty})[P_{t}h]\right]\label{eq:derivative_of_tilde_U}
\end{equation}
for all $\tilde{\rho},\rho\in\H[J]$, $\zeta\in\H[-I]$ and $t\ge0$.
Using the differentiability of $F$, we get
\begin{align*}
 & \left|\tilde{U}_{t}(\tilde{\rho},\rho,\zeta+h)-\tilde{U}_{t}(\tilde{\rho},\rho,\zeta)-D_{3}\tilde{U}_{t}(\tilde{\rho},\rho,\zeta)[h]\right|\\
 & \qquad=\left|\E\left[F(\tilde{\rho}_{t}^{\infty},\zeta_{t}^{\infty,h})-F(\tilde{\rho}_{t}^{\infty},\zeta_{t}^{\infty})-D_{2}F(\tilde{\rho}_{t}^{\infty},\zeta_{t}^{\infty})[P_{t}h]\right]\right|\\
 & \qquad=\left|\E\left[F(\tilde{\rho}_{t}^{\infty},\zeta_{t}^{\infty}+P_{t}h)-F(\tilde{\rho}_{t}^{\infty},\zeta_{t}^{\infty})-D_{2}F(\tilde{\rho}_{t}^{\infty},\zeta_{t}^{\infty})[P_{t}h]\right]\right|,
\end{align*}
where $(\zeta_{t}^{\infty,h})_{t\ge0}$ is a solution to (\ref{eq:SPDE_for_OU_process-1})
started from $\zeta+h$ and, by the linearity of (\ref{eq:SPDE_for_OU_process-1}),
$\zeta_{t}^{\infty,h}=\zeta_{t}^{\infty}+P_{t}h$. Consequently, we
can estimate the right hand side of the equality above by $\frac{1}{2}\norm{D_{2}^{2}F}_{\Cf}\norm{P_{t}h}_{\H[-I]}^{2}$,
according to \citep[Theorem 5.6.1]{Cartan:1971}. The continuity $D_{3}\tilde{U}$
can be proved similarly to the continuity of $\tilde{U}$. Similarly,
we can also prove that $\tilde{U}$ is continuously differentiable
with respect to $\zeta$ to the third order and continuously differentiable
with respect to $\tilde{\rho}$. Moreover, the derivatives have a
similar structure as in (\ref{eq:derivative_of_tilde_U}). Hence they
are uniformly bounded.

The continuous differentiability of $\tilde{U}$ with respect to $\rho$
and the boundedness of its derivative follows from Proposition \ref{prop:differentiability_of_semigroup_for_OU}.
Thus, $U\in\Cf_{l}^{0,1,3}([0,\infty),H_{J},H_{-I})$, by \citep[Proposition 2.6.2]{Cartan:1971}.
The fact that $D_{2}^{2}U_{t}(\rho,\zeta)\in\MLOHS[2][{\H[-I]}]$
follows from the estimate
\begin{align*}
 & \norm{D_{2}^{2}U_{t}(\rho,\zeta)}_{\MLOHS[2][{\H[-I]}]}^{2}=\sum_{k,l\in\Z^{d}}(1+|k|^{2})^{I}(1+|l|^{2})^{I}\left|\E\left[D_{2}^{2}F(\rho_{t}^{\infty},\zeta_{t}^{\infty})[P_{t}\tilde{\varsigma}_{k},P_{t}\tilde{\varsigma}_{l}]\right]\right|^{2}\\
 & \qquad=\sum_{k,l\in\Z^{d}}(1+|k|^{2})^{I}(1+|l|^{2})^{I}e^{-4\pi^{2}t\left(|k|^{2}+|l|^{2}\right)}\left|\E\left[D_{2}^{2}F(\rho_{t}^{\infty},\zeta_{t}^{\infty})[\tilde{\varsigma}_{k},\tilde{\varsigma}_{l}]\right]\right|^{2}\\
 & \qquad\le\sum_{k,l\in\Z^{d}}(1+|k|^{2})^{I}(1+|l|^{2})^{I}\E\left[\left|D_{2}^{2}F(\rho_{t}^{\infty},\zeta_{t}^{\infty})[\tilde{\varsigma}_{k},\tilde{\varsigma}_{l}]\right|^{2}\right]\\
 & \qquad=\E\left[\bnorm{D_{2}^{2}F(\rho_{t}^{\infty},\zeta_{t}^{\infty})}_{\MLOHS[2][{\H[-I]}]}^{2}\right]\le\sup_{\rho\in\H[J],\zeta\in\H[-I]}\bnorm{D_{2}^{2}F(\rho,\zeta)}_{\MLOHS[2][{\H[-I]}]}^{2}.
\end{align*}

The bound (\ref{eq:uniform_estimate_of_norm_of_U}) follows from the
latter inequality, direct estimates of the derivatives $D_{1}\tilde{U}$,
$D_{3}^{m}\tilde{U}$, $m\in[3]$, that satisfy expressions similar
to (\ref{eq:derivative_of_tilde_U}), Lemma \ref{lem:differentiability_of_variance_V}
and Proposition \ref{prop:differentiability_of_semigroup_for_OU}.

Let $\rho\in\H[J+2]$, $\zeta\in\H[-I+2]$ and $(\rho_{t}^{\infty},\zeta_{t}^{\infty})$,
$t\ge0$, be a solution to (\ref{eq:heat_PDE-1}), (\ref{eq:SPDE_for_OU_process-1})
started from $(\rho,\zeta)$. By Proposition \ref{prop:well_possedness_of_OUP_SSEP},
the process $(\rho^{\infty},\zeta^{\infty})$ takes values in $H_{J+2}\times\H[-I+2]$.
Using the Markov property of $(\rho^{\infty},\zeta^{\infty})$ and
Lemma \ref{lem:Ito_formula_for_eta}, we get for each $t\ge0$ and
$\eps>0$
\begin{align*}
U_{t+\eps}(\rho,\zeta)-U_{t}(\rho,\zeta) & =\E\left[U_{t}(\rho_{\eps}^{\infty},\zeta_{\eps}^{\infty})\right]-U_{t}(\rho,\zeta)\\
 & =2\pi^{2}\int_{0}^{\eps}\E\binn{\Delta D_{2}U_{t}(\rho_{s}^{\infty},\zeta_{s}^{\infty}),\zeta_{s}^{\infty}}ds\\
 & +2\pi^{2}\int_{0}^{t}\E\binn{\Delta D_{1}U_{t}(\rho_{s}^{\infty},\zeta_{s}^{\infty}),\rho_{s}^{\infty}}ds\\
 & +2\pi^{2}\int_{0}^{t}\sum_{j=1}^{d}\E\binn{\Tr\left(\partial_{j}^{\otimes2}D_{2}^{2}U_{t}(\rho_{s}^{\infty},\zeta_{s}^{\infty})\right),\Phi(\rho_{s}^{\infty})}ds.
\end{align*}
By the continuity of $(\rho^{\infty},\zeta^{\infty})$ in $\H[J+2]\times\H[-I+2]$,
the fact that $U\in\Cf_{l}^{0,1,2}([0,\infty),\H[J],\H[-I])$, the
estimate (\ref{eq:uniform_estimate_of_norm_of_U}) and Lemmas \ref{lem:tr_operator},
\ref{lem:norm_of_derivative_of_multilinear_operator} with the observation
that $\Phi(\rho_{t}^{\infty})$, $t\ge0$, is continuous in $\L$,
we get
\[
\lim_{\eps\to0+}\frac{U_{t+\eps}(\rho,\zeta)-U_{t}(\rho,\zeta)}{\eps}=\cG^{OU,\Phi}U_{t}(\rho,\zeta).
\]
 Taking into account that the right derivative of $(U_{t}(\rho,\zeta))_{t\ge0}$
with respect to $t$ is continuous, we conclude that $(U_{t}(\rho,\zeta))_{t\ge0}$
is continuously differentiable (in $t$) and the equality (\ref{eq:kolmogorov_equation})
holds. The continuity of $\partial U$ follows from (\ref{eq:kolmogorov_equation}).
\end{proof}

\section{Berry-Esseen bound for the initial fluctuations\protect\label{sec:Berry-Esseen-bound}}

The main result of this section is a quantified CLT for the fluctuations
of the random initialization of the SSEP $(\eta_{t}^{n})_{t\ge0}$.
Recall that $\eta_{0}^{n}$ has the distribution $\nu_{\rho_{0}^{n}}^{n}$
that is the product measure on $\EEP$ with marginals given by $\nu_{\rho_{0}^{n}}^{n}\left\{ \eta(x)=1\right\} =\rho_{0}^{n}(x)$,
$x\in\Tnd$, for a function $\rho_{0}^{n}:\Tnd\to[0,1]$. We define
the multilinear operator 
\begin{equation}
A_{\rho}[\varphi,\psi]=\inn{\rho(1-\rho)\varphi,\psi},\quad\varphi,\psi\in\H[I],\label{eq:covariance_operator_A_rho}
\end{equation}
for $\rho\in\L$ taking values in $[0,1]$ and $I>\frac{d}{2}$, and
note that it is a trace class operator since 
\begin{align*}
\sum_{k\in\Z^{d}}\frac{1}{(1+|k|^{2})^{I}}A[\tilde{\varsigma}_{k},\tilde{\varsigma}_{k}] & =\sum_{k\in\Z^{d}}\frac{1}{(1+|k|^{2})^{I}}\inn{\rho(1-\rho)\tilde{\varsigma}_{k},\tilde{\varsigma}_{k}}\\
 & \le\sum_{k\in\Z^{d}}\frac{1}{(1+|k|^{2})^{I}}\left(\norm{\rho}+\norm{\rho}^{2}\right)<\infty.
\end{align*}
Thus, by \citep[Proposition 3.15]{Hairer:2009}, there exists a centered
Gaussian random variable $\zeta$ in $\H[-I]$ with covariance $A_{\rho}$,
that is,
\[
\E\left[\inn{\zeta,\varphi}\inn{\psi,\zeta}\right]=A_{\rho}[\varphi,\psi],\quad\varphi,\psi\in\H[I].
\]

In the next statement we obtain a rate of convergence for the fluctuation
density field $\eta_{0}^{n}$ of the SSEP, started with distribution
$\nu_{\rho_{0}^{n}}^{n}$, to a Gaussian random variable with covariance
operator $A_{\rho_{0}}$. Since in this section, we do not work with
processes but only with their initial conditions, we will drop the
time-dependence in the notation throughout this section.
\begin{prop}
\label{prop:berry_essen_bound} Let $I>\frac{d}{2}+1$ and $\rho\in\Cf^{1}(\T^{d})$.
Assume that $\zeta$ is a centered Gaussian random variable in $\H[-I]$
with covariance operator $A_{\rho}$. Let also $\rho_{n}\in\Ln$,
$\eta_{n}$ have distribution $\nu_{\rho_{n}}^{n}$ and $\zeta_{n}=(2n+1)^{d/2}\left(\eta_{n}-\rho_{n}\right)$
for each $n\ge1$. Then for each $F\in\Cf_{l,HS}^{3}(\H[-I])$ and
$n\ge1$ 
\begin{align*}
\left|\E F(\ex_{n}\zeta_{n})-\E F(\zeta)\right| & \le C_{I}\left(\frac{1}{n^{1\wedge\frac{d}{2}}}(1+\norm{\nabla\rho}_{\Cf})+\norm{\rho_{n}-\rho}_{n}\right)\norm F_{\Cf_{l,HS}^{3}}.
\end{align*}
\end{prop}

\begin{proof}
Using the triangle inequality, it is enough to estimate $\E F(\ex_{n}\zeta_{n})-\E F(\pr_{n}\zeta)$
and $\E F(\pr_{n}\zeta)-\E F(\zeta)$. By the mean value theorem (see
\citep[Theorem 3.3.2]{Cartan:1971}), we obtain 
\[
\left|\E F(\pr_{n}\zeta)-\E F(\zeta)\right|\le\norm{DF}_{\Cf}\E\norm{\pr_{n}\zeta-\zeta}_{H_{-I}}.
\]
Note that $\zeta$ has a version that belongs to $H_{-I+1}$ due to
the fact that $I-1>\frac{d}{2}$ and \citep[Proposition 3.15]{Hairer:2009}.
Thus, $\E[\norm{\zeta}_{-I+1}]<\infty$. Then using Lemma \ref{lem:basic_properties_of_pr},
we get
\[
\E\norm{\pr_{n}\zeta-\zeta}_{H_{-I}}\le\frac{C_{I}}{n}\E\norm{\pr_{n}\zeta-\zeta}_{H_{-I+1}}\le\frac{C_{I}}{n}\E\norm{\zeta}_{H_{-I+1}}\le\frac{C_{I}}{n}.
\]

We next estimate $R_{I}^{n}:=\left|\E F(\pr_{n}\zeta)-\E F(\ex_{n}\zeta_{n})\right|$
by adopting Stein's method, see e.g., \citep{Meckes:2009,Reinert_Roellin:2009}
and the survey paper \citep{Ross_2011}. While in these contributions,
Stein's method is developed for finite-dimensional random variables,
the dimension of $\zeta_{n}$ diverges to infinity for $n\to\infty$.
Therefore, we need to carefully control the dependency of the occurring
constants, and to control them uniformly with respect to the dimension. 

Let $n\ge1$ be fixed. An important step in the estimation of $R_{I}^{n}$
is the identification of the (finite-dimensional) random variable
$\ex_{n}\zeta_{n}$ taking values in the Sobolev space $\H[-I]$ with
a random variable $X$ taking values in a Euclidean space, and to
then build an exchangeable pair $(X,X').$ This will allow to apply
the general finite-dimensional result from \citep[Theorem 3]{Meckes:2009}.
We will identify $\ex_{n}\zeta_{n}$ with its coordinates with respect
to the basis $\tilde{\varsigma}_{k}'':=\frac{1}{(1+|k|^{2})^{I/2}}\tilde{\varsigma}_{k}$,
$k\in\Z^{d}$, of $\H[-I]$ by defining $X_{k}:=\inn{\zeta_{n},\tilde{\varsigma}_{k}''}_{n}$
for $k\in\Znd$. Then $X=(X_{k})_{k\in\Znd}$ is a random variable
in $\R^{\Znd}$. In particular, $|X|^{2}=\norm{\ex_{n}\zeta_{n}}_{\H[-I]}^{2}.$ 

The standard approach for the construction of an exchangeable pair
for a random vector with independent coordinates is to replace a randomly
chosen coordinate by an independent one with the same distribution.
Note that the coordinates of $X$ are not independent. However, we
have the independence of the fluctuations $\zeta_{n}(x)$, $x\in\Tnd$.
Therefore, we will replace $\zeta_{n}(x)$, $x\in\Tnd$, by an independent
copy for a randomly chosen $x$. Let $\tilde{\zeta}_{n}$ be an independent
copy of $\zeta_{n}$ and $\gamma$ be a uniformly distributed random
variable on $\Tnd$ that is independent of $\zeta_{n}$ and $\tilde{\zeta}_{n}$.
Define
\[
\zeta_{n}'(x):=\zeta_{n}(x)\I_{\{\gamma\not=x\}}+\tilde{\zeta}_{n}(x)\I_{\{\gamma=x\}},\quad x\in\Tnd,
\]
and 
\[
\zeta_{n}'(x):=\zeta_{n}(x)\I_{\{\gamma\not=x\}}+\tilde{\zeta}_{n}(x)\I_{\{\gamma=x\}},\quad x\in\Tnd,
\]
and 
\begin{align*}
X'_{k} & :=\inn{\zeta'_{n},\tilde{\varsigma}_{k}''}_{n}=X_{k}+\frac{1}{(2n+1)^{d}}\left(\tilde{\zeta}_{n}(\gamma)-\zeta_{n}(\gamma)\right)\tilde{\varsigma}_{k}''(\gamma)
\end{align*}
for each $k\in\Znd$. Trivially, $(X,X')$ is an exchangeable pair,
that is, $(X,X')$ and $(X',X)$ have the same distribution.

We also need to replace the function $F:\H[-I]\to\R$ by a function
$f_{n}:\R^{\Tnd}\to\R$ such that $F(\ex_{n}\zeta_{n})=f_{n}(X)$.
Trivially, we have to take $f_{n}:=F\circ\kappa_{n}$, where $\kappa_{n}(z):=\sum_{k\in\Z^{d}}z_{k}\tilde{\varsigma}_{k}'$
for $z=(z_{k})_{k\in\Znd}$ and $\tilde{\varsigma}_{k}':=\left(1+|k|^{2}\right)^{I/2}\tilde{\varsigma}_{k}$,
$k\in\Znd$. In particular, $f\in\Cf^{3}(\R^{\Tnd})$ and
\[
\frac{\partial f_{n}}{\partial z_{k}}=DF\left(\kappa_{n}\right)\left[\tilde{\varsigma}_{k}'\right]\quad\mbox{and}\quad\frac{\partial^{2}f_{n}}{\partial z_{k}\partial z_{l}}=D^{2}F\left(\kappa_{n}\right)\left[\tilde{\varsigma}_{k}',\tilde{\varsigma}_{l}'\right]
\]
for all $k,l\in\Znd$. 

Next, for every $k\in\Znd$ we compute
\begin{align*}
\E\left[X'_{k}-X_{k}|X\right] & =\frac{1}{(2n+1)^{d}}\E\left[\left(\tilde{\zeta}_{n}(\gamma)-\zeta_{n}(\gamma)\right)\tilde{\varsigma}_{k}''(\gamma)|\zeta\right]\\
 & =\frac{1}{(2n+1)^{2d}}\sum_{x\in\Znd}\left(-\zeta_{n}(x)\right)\tilde{\varsigma}_{k}''(x)=-\frac{1}{(2n+1)^{d}}X.
\end{align*}
Moreover, for each $k,l\in\Znd$
\begin{align*}
 & \E\left[\left(X'_{k}-X_{k}\right)(X'_{l}-X_{l})|X\right]\\
 & \qquad=\frac{1}{(2n+1)^{2d}}\sum_{x,y\in\Znd}\E\left[(\zeta'_{n}(x)-\zeta_{n}(x))(\zeta'_{n}(y)-\zeta_{n}(y))|\zeta\right]\tilde{\varsigma}''_{k}(x)\tilde{\varsigma}''_{l}(y)\\
 & \qquad=\frac{1}{(2n+1)^{2d}}\sum_{x,y\in\Znd}\E\left[(\tilde{\zeta}_{n}(x)-\zeta_{n}(x))(\tilde{\zeta}_{n}(y)-\zeta_{n}(y))\I_{\{\gamma=x,\gamma=y\}}|\zeta\right]\tilde{\varsigma}''_{k}(x)\tilde{\varsigma}''_{l}(y).\\
 & \qquad=\frac{1}{(2n+1)^{3d}}\sum_{x\in\Znd}\E\left[(\tilde{\zeta}_{n}(x)-\zeta_{n}(x))^{2}|\zeta\right]\tilde{\varsigma}''_{k}(x)\tilde{\varsigma}''_{l}(x).
\end{align*}
Due to the equality
\[
\E\left[(\tilde{\zeta}_{n}(x)-\zeta_{n}(x))^{2}|\zeta\right]=\E\left[(\zeta_{n}(x))^{2}\right]+\zeta_{n}^{2}(x)=(2n+1)^{d}\rho_{n}(x)(1-\rho_{n}(x))+\zeta_{n}^{2}(x),
\]
we get
\begin{align*}
\E\left[\left(X'_{k}-X_{k}\right)(X'_{l}-X_{l})|X\right] & =\frac{1}{(2n+1)^{2d}}\sum_{x\in\Znd}\rho_{n}(x)(1-\rho_{n}(x))\tilde{\varsigma}''_{k}(x)\tilde{\varsigma}''_{l}(x)\\
 & +\frac{1}{(2n+1)^{3d}}\sum_{x\in\Znd}\zeta_{n}^{2}(x)\tilde{\varsigma}''_{k}(x)\tilde{\varsigma}''_{l}(x)\\
 & =\frac{1}{(2n+1)^{d}}\binn{\rho_{n}(1-\rho_{n})\tilde{\varsigma}_{k}'',\tilde{\varsigma}_{l}''}_{n}+\frac{1}{(2n+1)^{2d}}\binn{\zeta_{n}^{2}\tilde{\varsigma}_{k}'',\tilde{\varsigma}_{l}''}_{n}.
\end{align*}
Note that the entries of the covariance matrix $\Sigma=(\Sigma_{k,l})_{k,l\in\Znd}$
of the random vector 
\[
Z=\left(\inn{\pr_{n}\zeta,\tilde{\varsigma}_{k}''}\right)_{k\in\Znd}=\left(\inn{\zeta,\tilde{\varsigma}_{k}''}\right)_{k\in\Znd}
\]
 are given by
\[
\Sigma_{k,l}=\inn{\rho(1-\rho)\tilde{\varsigma}_{k}'',\tilde{\varsigma}_{l}''}.
\]
 We thus rewrite 
\begin{align*}
\E\left[\left(X'_{k}-X_{k}\right)(X'_{l}-X_{l})|X\right] & =\frac{2}{(2n+1)^{d}}\binn{\rho(1-\rho)\tilde{\varsigma}_{k}'',\tilde{\varsigma}_{l}''}\\
 & +\frac{1}{(2n+1)^{d}}\left[\binn{\left((\eta_{n}-\rho_{n})^{2}-\rho_{n}(1-\rho_{n})\right)\tilde{\varsigma}_{k}'',\tilde{\varsigma}_{l}''}_{n}\right]\\
 & +\frac{2}{(2n+1)^{d}}\left[\binn{\rho_{n}(1-\rho_{n})\tilde{\varsigma}_{k}'',\tilde{\varsigma}_{l}''}_{n}-\binn{\rho(1-\rho)\tilde{\varsigma}_{k}'',\tilde{\varsigma}_{l}''}\right]\\
 & =\frac{2}{(2n+1)^{d}}\Sigma_{k,l}+\frac{1}{(2n+1)^{d}}E{}_{k,l}^{*},
\end{align*}
where 
\begin{align*}
E_{k,l}^{*} & =\binn{\left((\eta_{n}-\rho_{n})^{2}-\rho_{n}(1-\rho_{n})\right)\tilde{\varsigma}_{k}'',\tilde{\varsigma}_{l}''}_{n}\\
 & +2\left[\binn{\rho_{n}(1-\rho_{n})\tilde{\varsigma}_{k}'',\tilde{\varsigma}_{l}''}_{n}-\binn{\rho(1-\rho)\tilde{\varsigma}_{k}'',\tilde{\varsigma}_{l}''}\right]\\
 & =:E_{k,l}^{1*}+2E_{k,l}^{2*}.
\end{align*}
Using \citep[Theorem 3]{Meckes:2009}, we get
\begin{align*}
\left|\E F(\ex_{n}\zeta_{n})-\E F(\pr_{n}\zeta)\right| & =\left|\E f_{n}(X)-\E f_{n}(Z)\right|\\
 & \le\frac{(2n+1)^{d}}{4(2n+1)^{d}}\norm{D^{2}f_{n}}_{\Cf(\MLOHS[2][\R^{\Tnd}])}\E\norm{E^{*}}_{\MLOHS[2][\R^{\Tnd}]}\\
 & +\frac{(2n+1)^{d}}{9}\norm{D^{3}f_{n}}_{\Cf(\MLO[3][\R^{\Tnd}])}\E\left|X'-X\right|^{3},
\end{align*}
We next estimate each term in the right hand side of the inequality
above. We start from
\begin{align*}
\norm{D^{2}f_{n}}_{\Cf(\MLOHS[2][\R^{\Tnd}])}^{2} & =\sup_{z\in\R^{\Znd}}\sum_{k,l\in\Znd}\left(\frac{\partial^{2}f_{n}}{\partial z_{k}\partial z_{l}}(z)\right)^{2}\\
 & =\sup_{z\in\R^{\Znd}}\sum_{k,l\in\Znd}\left(D^{2}F\left(\kappa_{n}(z)\right)\left[\tilde{\varsigma}_{k}',\tilde{\varsigma}_{l}'\right]\right)^{2}\\
 & =\sup_{z\in\R^{\Znd}}\sum_{k,l\in\Znd}(1+|k|^{2})^{I}(1+|l|^{2})^{I}\left(D^{2}F\left(\kappa_{n}(z)\right)\left[\tilde{\varsigma}_{k},\tilde{\varsigma}_{l}\right]\right)\\
 & \le\sup_{g\in\H[-I]}\norm{D^{2}F(g)}_{\MLOHS[2][{\H[-I]}]}^{2}=\norm{D^{2}F}_{C(\MLOHS[2][H_{-I}])}^{2}.
\end{align*}
Using Hölder's inequality and then Jensen's inequality, we get
\[
\E\left[\norm{E^{*}}_{\MLOHS[2][\R^{\Tnd}]}\right]^{2}\le2\E\left[\norm{E^{1*}}_{\MLOHS[2][\R^{\Tnd}]}^{2}\right]+8\E\left[\norm{E^{2*}}_{\MLOHS[2][\R^{\Tnd}]}^{2}\right].
\]
Rewriting
\begin{align*}
 & \E\left[\norm{E^{1*}}_{\MLOHS[2][\R^{\Tnd}]}^{2}\right]=\sum_{k,l\in\Znd}\E\left[\binn{\left((\eta_{n}-\rho_{n})^{2}-\rho_{n}(1-\rho_{n})\right)\tilde{\varsigma}_{k}'',\tilde{\varsigma}_{l}''}_{n}^{2}\right]\\
 & \qquad=\frac{1}{(2n+1)^{2d}}\sum_{k,l\in\Znd}\E\left[\sum_{x\in\Tnd}\left((\eta_{n}(x)-\rho_{n}(x))^{2}-\rho_{n}(x)(1-\rho_{n}(x))\right)\tilde{\varsigma}_{k}''(x)\tilde{\varsigma}_{l}''(x)\right]^{2}
\end{align*}
and using the independence of $\eta_{n}(x)$, $x\in\Tnd$, and the
equality $\E\left[(\eta_{n}(x)-\rho_{n}(x))^{2}\right]=\rho_{n}(x)(1-\rho_{n}(x)),$
we get
\begin{align*}
\E\left[\norm{E^{1*}}_{\MLOHS[2][\R^{\Tnd}]}^{2}\right] & =\frac{1}{(2n+1)^{2d}}\sum_{k,l\in\Znd}\frac{1}{(1+|k|^{2})^{I}(1+|l|^{2})^{I}}\\
 & \qquad\qquad\cdot\sum_{x\in\Tnd}\E\left[\left((\eta_{n}(x)-\rho_{n}(x))^{2}-\rho_{n}(x)(1-\rho_{n}(x))\right)^{2}\right]\tilde{\varsigma}_{k}^{2}(x)\tilde{\varsigma}_{l}^{2}(x)\\
 & \le\frac{16}{(2n+1)^{2d}}\sum_{k,l\in\Znd}\frac{1}{(1+|k|^{2})^{I}(1+|l|^{2})^{I}}\\
 & \qquad\qquad\cdot\sum_{x\in\Tnd}\E\left[\left((\eta_{n}(x)-\rho_{n}(x))^{2}-\rho_{n}(x)(1-\rho_{n}(x))\right)^{2}\right]\\
 & \le\frac{C_{I}}{(2n+1)^{d}}
\end{align*}
due to the boundedness of $\eta_{n}$, $\rho_{n}$ and the fact that
$I>\frac{d}{2}.$ We now consider 
\begin{align*}
\E\left[\norm{E^{2*}}_{\MLOHS[2][\R^{\Tnd}]}^{2}\right] & =\sum_{k,l\in\Znd}\left(\binn{\rho_{n}(1-\rho_{n})\tilde{\varsigma}_{k}'',\tilde{\varsigma}_{l}''}_{n}-\binn{\rho(1-\rho)\tilde{\varsigma}_{k}'',\tilde{\varsigma}_{l}''}\right)^{2}\\
 & =\sum_{k,l\in\Znd}\frac{1}{(1+|k|^{2})^{I}(1+|l|^{2})^{I}}\left(\binn{\rho_{n}(1-\rho_{n})\tilde{\varsigma}_{k},\tilde{\varsigma}_{l}}_{n}-\binn{\rho(1-\rho)\tilde{\varsigma}_{k},\tilde{\varsigma}_{l}}\right)^{2}.
\end{align*}
To estimate the sum in the right hand side, we rewrite for $\varphi\in\Cf(\Td)$
\begin{align}
 & \left|\binn{\rho_{n}(1-\rho_{n}),\varphi}_{n}-\binn{\rho(1-\rho),\varphi}\right|\nonumber \\
 & \qquad=\left|\frac{1}{(2n+1)^{d}}\sum_{x\in\Tnd}\rho_{n}(x)(1-\rho_{n}(x))\varphi(x)-\int_{\Td}\rho(y)(1-\rho(y))\varphi(y)dy\right|\label{eq:difference_for_rho_n_and_rho}\\
 & \qquad=\left|\int_{\Td}\bar{\rho}_{n}(y)(1-\bar{\rho}_{n}(y))\bar{\varphi}_{n}(y)dy-\int_{\Td}\rho(y)(1-\rho(y))\varphi(y)dy\right|,\nonumber 
\end{align}
where
\[
\bar{\rho}_{n}=\sum_{x\in\Tnd}\rho_{n}(x)\I_{\pi_{x}^{n}}\quad\text{and}\quad\bar{\varphi}_{n}=\sum_{x\in\Tnd}\varphi(x)\I_{\pi_{x}^{n}}
\]
for $\pi_{x}^{n}=\prod_{j=1}^{d}\left[x_{j},x_{j}+\frac{2\pi}{2n+1}\right)$.
Using the triangle inequality, we can bound the right hand side of
(\ref{eq:difference_for_rho_n_and_rho}) by
\begin{align*}
 & \int_{\Td}\left|\bar{\rho}_{n}(y)-\rho(y)\right|(1-\bar{\rho}_{n}(y))\left|\bar{\varphi}_{n}(y)\right|dy\\
 & \qquad+\int_{\Td}\rho(y)\left|\bar{\rho}_{n}(y)-\rho(y)\right|\left|\bar{\varphi}_{n}(y)\right|dy\\
 & \qquad+\int_{\Td}\rho(y)(1-\rho(y))\left|\varphi(y)-\bar{\varphi}_{n}(y)\right|dy\\
 & \qquad\le2\norm{\varphi}_{\Cf}\int_{\Td}\left|\bar{\rho}_{n}(y)-\rho(y)\right|dy+\frac{C}{n}\norm{\nabla\varphi}_{\Cf}.
\end{align*}
This implies that 
\begin{align*}
\E\left[\norm{E^{2*}}_{\MLOHS[2][\R^{\Tnd}]}^{2}\right] & \le\sum_{k,l\in\Znd}\frac{1}{(1+|k|^{2})^{I}(1+|l|^{2})^{I}}\\
 & \qquad\qquad\qquad\cdot\left(\norm{\tilde{\varsigma}_{k}\tilde{\varsigma}_{l}}_{\Cf}\int_{\Td}\left|\bar{\rho}_{n}(y)-\rho(y)\right|dy+\frac{C}{n}\norm{\nabla\left(\tilde{\varsigma}_{k}\tilde{\varsigma}_{l}\right)}_{\Cf}\right)^{2}\\
 & \le C_{I}\left(\frac{1}{n}+\int_{\Td}\left|\bar{\rho}_{n}(y)-\rho(y)\right|dy\right)^{2}\\
 & \le C_{I}\left(\frac{1}{n}(1+\norm{\nabla\rho}_{\Cf})+\norm{\rho_{n}-\rho}_{n}\right)^{2}.
\end{align*}

We now estimate
\begin{align*}
\norm{D^{3}f_{n}}_{\Cf(\MLO[3][\R^{\Tnd}]} & =\sup_{z\in\R^{\Znd}}\sup_{|a^{i}|\le1}\left|D^{3}f_{n}(z)(a_{1},a_{2},a_{3})\right|\\
 & =\sup_{z\in\R^{\Znd}}\sup_{|a^{i}|\le1}\left|\sum_{k,l,i}\frac{\partial^{3}f_{n}}{\partial z_{k}\partial z_{l}\partial z_{i}}(z)a_{k}^{1}a_{l}^{2}a_{i}^{3}\right|\\
 & =\sup_{z\in\R^{\Znd}}\sup_{|a^{i}|\le1}\left|\sum_{k,l,i}D^{3}F(\kappa_{n}(z))\left[\tilde{\varsigma}_{k}',\tilde{\varsigma}_{l}',\tilde{\varsigma}_{i}'\right]a_{k}^{1}a_{l}^{2}a_{i}^{3}\right|\\
 & =\sup_{z\in\R^{\Znd}}\sup_{|a^{i}|\le1}\left|D^{3}F(\kappa_{n}(z))\left[\iota_{n}(a^{1}),\iota_{n}(a^{2}),\iota_{n}(a^{3})\right]\right|,
\end{align*}
where 
\[
\iota_{n}(a)=\sum_{k\in\Znd}a_{k}\tilde{\varsigma}_{k}'\in\L
\]
for $a\in\R^{\Znd}$. Due to the identity
\[
\norm{\iota(a)}_{H_{-I}}^{2}=\sum_{k\in\Znd}\frac{1}{(1+|k|^{2})^{I}}\left(1+|k|^{2}\right)^{I}a_{k}^{2}=\sum_{k\in\Znd}a_{k}^{2}=|a|^{2},
\]
we get
\[
\norm{D^{3}f_{n}}_{\Cf}\le\sup_{z\in\R^{\Znd}}\sup_{\norm{g_{i}}_{H_{-I}}\le1}\left|D^{3}F(\chi_{n}(z))\left[g_{1},g_{2},g_{3}\right]\right|\le\norm{D^{3}F}_{\Cf(\MLO[3][{\H[-I]}])}.
\]

It only remains to estimate 
\begin{align*}
\E\left[|X'-X|^{3}\right] & =\frac{1}{(2n+1)^{3d}}\E\left[\left(\sum_{k\in\Znd}\left|\tilde{\zeta}_{n}(\gamma)-\zeta_{n}(\gamma)\right|^{2}\left|\tilde{\varsigma}_{k}''(\gamma)\right|^{2}\right)^{\frac{3}{2}}\right]\\
 & =\frac{1}{(2n+1)^{3d}}\E\left[\left|\tilde{\zeta}_{n}(\gamma)-\zeta_{n}(\gamma)\right|^{3}\left(\sum_{k\in\Znd}\frac{1}{(1+|k|^{2})^{I}}\right)^{\frac{3}{2}}\right]\\
 & =\frac{C_{I}}{(2n+1)^{3d}}\E\left[\left|\tilde{\zeta}_{n}(\gamma)-\zeta_{n}(\gamma)\right|^{3}\right]\\
 & \le\frac{C_{I}}{(2n+1)^{3d/2}}
\end{align*}
due to the bound $\left|\zeta_{n}(\gamma)-\tilde{\zeta}_{n}(\gamma)\right|\le2(2n+1)^{d/2}$
for all $x\in\Tnd$. Combining all estimates together, we get the
statement of the proposition.
\end{proof}

\section{Proof of Theorem \ref{thm:main_result}\protect\label{sec:Comparison-of-processes}}

The goal of this section is to prove Theorem \ref{thm:main_result}.
We will do so under more general assumptions on $\rho_{0}^{n}$ than
in the statement of the result. Namely, we assume that the initial
conditions $\rho_{0}^{n}$ are arbitrary functions from $\Ln$ taking
values in $[0,1]$ such that $\sup_{n\ge1}\norm{\ex_{n}\rho_{0}^{n}}_{\H[\tilde{J}]}<\infty$.
Additionally, let $J>\frac{d}{2}\vee2$, $\tilde{I}>\frac{d}{2}+1,$
$I>\tilde{I}+\frac{d}{2}+2$ and $\tilde{J}>(\tilde{I}\vee(\frac{d}{2}+4))+\frac{d}{2}+1$.
We will show that for each $T>0$ there exists a constant $C$ independent
of $F$ and $n$ such that 
\[
\sup_{t\in[0,T]}\left|\E F(\hat{\rho}_{t}^{n},\hat{\zeta}_{t}^{n})-\E F(\rho_{t}^{\infty},\eta_{t}^{\infty})\right|\le C\norm F_{\Cf_{l,HS}^{1,3}}\left(\frac{1}{n^{\frac{d}{2}\wedge1}}+\norm{\hat{\rho}_{0}^{n}-\rho_{0}}_{\H[J]}\right).
\]
Using the inequality above and Lemma \ref{lem:estimate_of_phi-ex_phi},
this immediately yields Theorem \ref{thm:main_result}.

We first assume that $F\in\Cf_{l,HS}^{2,4}(\H[J],\H[-I])$ and $D_{2}^{2}F$
is uniformly continuous in $\MLOHS[2][{\H[-I]}]$. 

Let $t\in(0,T]$ be fixed. To compare the difference $\E\left[F(\rho_{t}^{\infty},\zeta_{t}^{\infty})\right]-\E[F(\hat{\rho}_{t}^{n},\hat{\zeta}_{t}^{n})]$,
we will use the expression (\ref{eq:the_main_comparison_of_generators-2}).
Since $U_{t-s}(\hat{\rho}_{s}^{n},\hat{\zeta}_{s}^{n})$ is not well-defined
there if $\hat{\rho}_{s}^{n}$ takes values outside $[0,1],$ we will
first replace the process $\zeta^{\infty}$ by a solutions to the
SPDE (\ref{eq:SPDE_for_OU_process-1}) with $\Phi$ being a mollification
of $f(x):=x(1-x)\vee0$, $x\in\R$. More precisely, we take a non-negative
function $\phi\in\Cf^{2}(\R)$ such that $\supp\phi\in[-1,1]$ and
$\int_{\R}\phi(x)dx=1$. Then for each $\eps>0$ we define $\phi_{\eps}:=\frac{1}{\eps}\phi(\eps\cdot)$
and $\Phi_{\eps}:=\phi_{\eps}*f$. Let
\[
U_{t}^{\eps}(\rho,\zeta):=\E F(\rho_{t}^{\infty},\zeta_{t}^{\infty,\eps}),
\]
where $(\rho^{\infty},\zeta^{\infty,\eps})$ is a solution to (\ref{eq:heat_PDE-1}),
(\ref{eq:SPDE_for_OU_process-1}) in $\H[J]\times\H[-\tilde{I}]$
started from $(\rho_{0},\zeta_{0})$ with $\Phi$ replaced by $\Phi_{\eps}$.
Since $I>\frac{d}{2}+3$ and $J>\frac{d}{2}+1$, we can use Proposition
\ref{prop:differentiability_of_semigroup_for_OU} to conclude that
$U^{\eps},\partial U^{\eps},D_{1}U^{\eps}\in\Cf([0,\infty)\times\H[J+2]\times\H[-I+2])$
and $U_{t}^{\eps}\in\Cf_{l,HS}^{1,3}(\H[J],\H[-I])$ for each $t>0$.
Thus, by Lemma \ref{lem:Ito_formula_for_eta_t_dependence} and Proposition
\ref{prop:fokker-plank-equation}, we get
\begin{align*}
\E F(\hat{\rho}_{t}^{n},\hat{\zeta}_{t}^{n}) & =\E U_{t-t}^{\eps}(\hat{\rho}_{t}^{n},\hat{\zeta}_{t}^{n})\\
 & =\E U_{t}^{\eps}(\hat{\rho}_{0}^{n},\hat{\zeta}_{0}^{n})+\int_{0}^{t}\E\left[\hat{\cG}^{FF}U_{t-s}^{\eps}(\hat{\rho}_{s}^{n},\hat{\zeta}_{s}^{n})-\partial U_{t-s}^{\eps}(\hat{\rho}_{t}^{n},\hat{\zeta}_{t}^{n})\right]ds\\
 & =\E U_{t}^{\eps}(\hat{\rho}_{0}^{n},\hat{\zeta}_{0}^{n})+\int_{0}^{t}\E\left[\hat{\cG}^{FF}U_{t-s}^{\eps}(\hat{\rho}_{s}^{n},\hat{\zeta}_{s}^{n})-\cG^{OU,\Phi_{\eps}}U_{t-s}^{\eps}(\hat{\rho}_{t}^{n},\hat{\zeta}_{t}^{n})\right]ds.
\end{align*}
Applying Proposition \ref{prop:expansion_of_cG_FEP}, we obtain
\begin{align*}
 & \E F(\hat{\rho}_{t}^{n},\hat{\zeta}_{t}^{n})-\E U_{t}^{\eps}(\hat{\rho}_{0}^{n},\hat{\zeta}_{0}^{n})\\
 & \qquad=\frac{2\pi^{2}}{(2n+1)^{d}}\sum_{j=1}^{d}\int_{0}^{t}\E\binn{\Tr\left(\partial_{j}^{\otimes2}D_{2}^{2}U_{t-s}^{\eps}(\hat{\rho}_{s}^{n},\hat{\zeta}_{s}^{n})\right),\ex_{n}\left[\zeta_{s}^{n}\tau_{j}^{n}\zeta_{s}^{n}\right]}ds\\
 & \qquad+4\pi^{2}\int_{0}^{t}\sum_{j=1}^{d}\E\Bigg[\binn{\Tr\left(\partial_{j}^{\otimes2}D_{2}^{2}U_{t-s}^{\eps}(\hat{\rho}_{s}^{n},\hat{\zeta}_{s}^{n})\right),\hat{\rho}_{s}^{n}(1-\hat{\rho}_{s}^{n})}\\
 & \qquad\qquad\qquad\qquad-\binn{\Tr\left(\partial_{j}^{\otimes2}D_{2}^{2}U_{t-s}^{\eps}(\hat{\rho}_{s}^{n},\hat{\zeta}_{s}^{n})\right),\Phi_{\eps}\left(\hat{\rho}_{s}^{n}\right)}\Bigg]ds\\
 & \qquad+\int_{0}^{t}\E R_{t-s}^{n}(\rho_{s}^{n},\zeta_{s}^{n})ds,
\end{align*}
where 
\[
\left|R_{s}^{n}(\rho,\zeta)\right|\le\frac{C_{J,I,\tilde{I},T}}{n^{\frac{d}{2}\wedge1}}\norm{U_{t-s}^{\eps}}_{\Cf_{l,HS}^{1,3}}\left(1+\norm{\hat{\rho}}_{\Cf^{\lceil d/2\rceil+4}}^{2}+\norm{\hat{\rho}}_{C^{\lceil\tilde{I}\rceil}}\right)\left(1+\norm{\hat{\zeta}}_{H_{-\tilde{I}}}\right)
\]
for all $\rho\in\EHE$ and $\zeta=(2n+1)^{d/2}(\eta-\rho)$, $\eta\in\EEP$.
Consequently,
\begin{align}
 & \Bigg|\E F(\hat{\rho}_{t}^{n},\hat{\zeta}_{t}^{n})-\E U_{t}^{\eps}(\hat{\rho}_{0}^{n},\hat{\zeta}_{0}^{n})\Bigg|\nonumber \\
 & \qquad\le\frac{2\pi^{2}}{(2n+1)^{d}}\sum_{j=1}^{d}\int_{0}^{t}\left|\E\binn{\Tr\left(\partial_{j}^{\otimes2}D_{2}^{2}U_{t-s}^{\eps}(\hat{\rho}_{s}^{n},\hat{\zeta}_{s}^{n})\right),\ex_{n}\left[\zeta_{s}^{n}\tau_{j}^{n}\zeta_{s}^{n}\right]}\right|ds\nonumber \\
 & \qquad+2\pi^{2}\int_{0}^{t}\sum_{j=1}^{d}\E\left|\binn{\Tr\left(\partial_{j}^{\otimes2}D_{2}^{2}U_{t-s}^{\eps}(\hat{\rho}_{s}^{n},\hat{\zeta}_{s}^{n})\right),\hat{\rho}_{s}^{n}(1-\hat{\rho}_{s}^{n})-\Phi_{\eps}\left(\hat{\rho}_{s}^{n}\right)}\right|ds\label{eq:comparison_of_semigroup_for_SSEP}\\
 & \qquad+\frac{C_{J,I,\tilde{I},T}}{n^{\frac{d}{2}\wedge1}}\int_{0}^{t}\norm{U_{t-s}^{\eps}}_{\Cf_{l,HS}^{1,3}}\left(1+\norm{\hat{\rho}_{s}^{n}}_{\Cf^{\lceil d/2\rceil+4}}^{2}+\norm{\hat{\rho}_{s}^{n}}_{C^{\lceil\tilde{I}\rceil}}\right)\left(1+\E\norm{\hat{\zeta}_{s}^{n}}_{H_{-\tilde{I}}}\right)ds.\nonumber 
\end{align}

We next note that the function $f_{s}^{n,j,\eps}:=\Tr\left(\partial_{j}^{\otimes2}D_{2}^{2}U_{t-s}^{\eps}(\hat{\rho}_{s}^{n},\hat{\zeta}_{s}^{n})\right)$
belongs to $\H[\tilde{I}]$ due to $\tilde{I}+1+\frac{d}{2}<I$ and
\begin{align}
\norm{f_{s}^{n,j,\eps}}_{H_{\tilde{I}}} & \le C\norm{\partial_{j}^{\otimes2}D_{2}^{2}U_{t-s}^{\eps}(\hat{\rho}_{s}^{n},\hat{\zeta}_{s}^{n})}_{\MLOHS[2][{\H[-I+1]}]}\label{eq:estimate_of_f_n_for_tr}\\
 & \le C\norm{D_{2}^{2}U_{t-s}^{\eps}(\hat{\rho}_{s}^{n},\hat{\zeta}_{s}^{n})}_{\MLOHS[2][{\H[-I]}]}\le C_{I,T}(\norm{\Phi_{\eps}'}_{\Cf}+1)\norm F_{\Cf_{l,HS}^{1,3}},\nonumber 
\end{align}
according to Proposition \ref{prop:fokker-plank-equation} and Lemmas
\ref{lem:tr_operator} and \ref{lem:norm_of_derivative_of_multilinear_operator}.
Thus, by Lemma \ref{lem:estimate_of_eta_tau_eta} (recall that $\tilde{I}>\frac{d}{2}$),
the first term of (\ref{eq:comparison_of_semigroup_for_SSEP}) can
be estimated by 
\[
\frac{\tilde{C}}{n^{\frac{d}{2}\wedge1}}\E\left[\norm{f_{s}^{n,j,\eps}}_{\H[\tilde{I}]}^{2}\right]^{1/2}\le\frac{\tilde{C}}{n^{\frac{d}{2}\wedge1}}C_{I,T}(\norm{\Phi_{\eps}'}_{\Cf}+1)\norm F_{\Cf_{l,HS}^{1,3}},
\]
where the constant $\tilde{C}$ depends on $\tilde{J}$, $d$ and
$\sup_{n\ge1}\norm{\nabla\rho_{0}^{n}}_{n,\Cf}$. Note that the finiteness
of $\sup_{n\ge1}\norm{\nabla\rho_{0}^{n}}_{n,\Cf}$ follows from 
\begin{align}
\bnorm{\nabla_{n}\rho_{0}^{n}}_{n,\Cf} & \le\sum_{j=1}^{d}\bnorm{\partial_{n,j}\rho_{0}^{n}}_{n,\Cf}\le\sum_{j=1}^{d}\bnorm{\ex_{n}\partial_{n,j}\rho_{0}^{n}}_{\Cf}\nonumber \\
 & \le\sum_{j=1}^{d}\norm{\ex_{n}\partial_{n,j}\rho_{0}^{n}}_{\H[\tilde{J}-1]}\le\sum_{j=1}^{d}\norm{\ex_{n}\rho_{0}^{n}}_{\H[\tilde{J}]}\label{eq:estimate_of_grad_n_rho_n}\\
 & =d\norm{\hat{\rho}_{0}^{n}}_{\H[\tilde{J}]}\nonumber 
\end{align}
and the assumption (i) of the theorem, where starting from the second
inequality in the estimate above we used the interpolation property
(\ref{eq:ex_is_interpolation}) of $\ex_{n}$, the Sobolev embedding
theorem and then Lemma \ref{lem:estimate_of_norm_ex_partial}. 

We next estimate the second term of the right hand side of (\ref{eq:comparison_of_semigroup_for_SSEP}).
We note that $\left|\Phi_{\eps}(x)-f(x)\right|\le\eps$ for all $x\in\R$.
Thus, 
\begin{align}
 & \int_{0}^{t}\sum_{j=1}^{d}\left|\binn{\Tr\left(\partial_{j}^{\otimes2}D_{2}^{2}U_{t-s}^{\eps}(\hat{\rho}_{s}^{n},\hat{\zeta}_{s}^{n})\right),\hat{\rho}_{s}^{n}(1-\hat{\rho}_{s}^{n})-\Phi_{\eps}\left(\hat{\rho}_{s}^{n}\right)}\right|ds\nonumber \\
 & \qquad\qquad\le\int_{0}^{t}\sum_{j=1}^{d}\left|\binn{f_{s}^{n,j,\eps},f\left(\hat{\rho}_{s}^{n}\right)-\Phi_{\eps}\left(\hat{\rho}_{s}^{n}\right)}\right|ds\label{eq:bound_tr_term_in_the_main_proof}\\
 & \qquad\qquad+\int_{0}^{t}\sum_{j=1}^{d}\left|\binn{f_{s}^{n,j,\eps},\hat{\rho}_{s}^{n}(1-\hat{\rho}_{s}^{n})\I_{\left\{ \hat{\rho}_{s}^{n}\not\in[0,1]\right\} }}\right|ds.\nonumber 
\end{align}
The first term of the right hand side of (\ref{eq:bound_tr_term_in_the_main_proof})
can be estimated by $\eps C_{I,T}(\norm{\Phi_{\eps}'}_{\Cf}+1)\norm F_{\Cf_{l,HS}^{1,3}}$,
according to the Cauchy-Schwarz inequality and (\ref{eq:estimate_of_f_n_for_tr}).
Next, for each $s\in[0,t]$ we have
\begin{align*}
\left|\binn{f_{s}^{n,j,\eps},\hat{\rho}_{s}^{n}(1-\hat{\rho}_{s}^{n})\I_{\left\{ \hat{\rho}_{s}^{n}\not\in[0,1]\right\} }}\right| & \le\norm{f_{s}^{n,j,\eps}}_{\Cf}\bnorm{\hat{\rho}_{s}^{n}\I_{\left\{ \hat{\rho}_{s}^{n}\not\in[0,1]\right\} }}\bnorm{(1-\hat{\rho}_{s}^{n})\I_{\left\{ \hat{\rho}_{s}^{n}\not\in[0,1]\right\} }}\\
 & \le\norm{f_{s}^{n,j,\eps}}_{\tilde{I}}\bnorm{\hat{\rho}_{s}^{n}\I_{\left\{ \hat{\rho}_{s}^{n}<0\right\} }}\bnorm{1-\hat{\rho}_{s}^{n}}.
\end{align*}
Consider the convex function $\psi(x):=|x|\I_{\{x<0\}}$, $x\in\R$,
and note that it satisfies the triangle inequality $\psi(x+y)\le\psi(x)+\psi(y)$,
$x,y\in\R.$ Thus,
\begin{align*}
\bnorm{\hat{\rho}_{s}^{n}\I_{\left\{ \hat{\rho}_{s}^{n}<0\right\} }} & =\bnorm{\psi\left(\hat{\rho}_{s}^{n}\right)}\le\bnorm{\psi\left(\hat{\rho}_{s}^{n}-\rho_{s}^{\infty}\right)}+\bnorm{\psi\left(\rho_{s}^{\infty}\right)}\\
 & \le\bnorm{\hat{\rho}_{s}^{n}-\rho_{s}^{\infty}}+0,
\end{align*}
since $\rho_{s}^{\infty}\ge0$. Now, using the triangle inequality,
Corollary \ref{cor:comparison_discrete_and_continuous_heat_equations}
and Lemma \ref{lem:basic_properties_of_pr}, we get
\begin{align*}
\bnorm{\hat{\rho}_{s}^{n}-\rho_{s}^{\infty}} & \le\bnorm{\hat{\rho}_{s}^{n}-\pr_{n}\rho_{s}^{\infty}}+\bnorm{\pr_{n}\rho_{s}^{\infty}-\rho_{s}^{\infty}}\\
 & \le C_{T}\bnorm{\hat{\rho}_{0}^{n}-\rho_{0}}+\frac{C_{T}}{n}\bnorm{\rho_{0}}_{H_{2}}.
\end{align*}
Consequently,
\[
\bnorm{\hat{\rho}_{s}^{n}\I_{\left\{ \hat{\rho}_{s}^{n}<0\right\} }}\le C_{T}\bnorm{\hat{\rho}_{0}^{n}-\rho_{0}}+\frac{C_{T}}{n}\bnorm{\rho_{0}}_{H_{2}}
\]
for all $s\in[0,t]$. Note that
\[
\norm{1-\hat{\rho}_{s}^{n}}=\bnorm{1-\rho_{s}^{n}}_{n}\le1,
\]
according to the maximum principle. This shows that the second term
in the right hand side of (\ref{eq:bound_tr_term_in_the_main_proof})
is estimated by
\[
C_{I,T}(\norm{\Phi_{\eps}'}_{\Cf}+1)\norm F_{\Cf_{l,HS}^{1,3}}\left(\bnorm{\hat{\rho}_{0}^{n}-\rho_{0}}+\frac{1}{n}\bnorm{\rho_{0}}_{H_{2}}\right).
\]

To estimate the third term of the right hand side of (\ref{eq:comparison_of_semigroup_for_SSEP}),
we use Proposition \ref{prop:fokker-plank-equation} to control $\norm{U_{t-s}^{\eps}}_{\Cf_{l,HS}^{1,3}}$
by $C_{I,T}(\norm{\Phi_{\eps}'}_{\Cf}+1)\norm F_{\Cf_{l,HS}^{1,3}}$.
Next, recall that the sequence $\norm{\hat{\rho}_{0}^{n}}_{\H[\tilde{J}]}$,
$n\ge1$, is bounded. Since trivially $\norm{\hat{\rho}_{t}^{n}}_{\H[\tilde{J}]}\le\norm{\hat{\rho}_{0}^{n}}_{\H[\tilde{J}]}$
for all $t\ge0$, we get that $\norm{\hat{\rho}_{s}^{n}}_{\Cf^{\lceil d/2\rceil+4}}$
and $\norm{\hat{\rho}_{s}^{n}}_{\Cf^{\lceil\tilde{I}\rceil}}$ are
uniformly bounded in $n\ge1$ and $s\in[0,t]$, due to the Sobolev
embedding theorem and the fact that $\tilde{J}>\lceil d/2\rceil+4+\frac{d}{2}$
and $\tilde{J}>\lceil\tilde{I}\rceil+\frac{d}{2}$. Using Lemma \ref{lem:estimate_of_sobolev_norm_of_eta}
and (\ref{eq:estimate_of_grad_n_rho_n}), we get
\[
\E\left[\norm{\ex_{n}\zeta_{t}^{n}}_{H_{-\tilde{I}}}^{2}\right]<C_{\tilde{I}}\left(1+2\pi^{2}dt\norm{\hat{\rho}_{0}^{n}}_{\H[\tilde{J}]}\right).
\]
 This completes the proof of the fact that 
\begin{align}
 & \left|\E F(\hat{\rho}_{t}^{n},\hat{\zeta}_{t}^{n})-\E U_{t}^{\eps}(\hat{\rho}_{0}^{n},\hat{\zeta}_{0}^{n})\right|\nonumber \\
 & \qquad\le C(\norm{\Phi_{\eps}'}_{\Cf}+1)\norm F_{\Cf_{l,HS}^{1,3}}\left(\frac{1}{n^{\frac{d}{2}\wedge1}}\left(1+\bnorm{\rho_{0}}_{H_{2}}\right)+\bnorm{\hat{\rho}_{0}^{n}-\rho_{0}}+\eps\right),\label{eq:main_inequality_for_rho_n_initial_condition}
\end{align}
where the constant $C$ depends on $J,\tilde{J},I,\tilde{I},T$ and
$\sup_{n\ge1}\norm{\hat{\rho}_{0}^{n}}_{\H[\tilde{J}]}$. 

We next estimate the difference $\E U_{t}^{\eps}(\rho_{0},\zeta_{0})-\E U_{t}^{\eps}(\hat{\rho}_{0}^{n},\hat{\zeta}_{0}^{n})$.
By the triangle inequality and the mean-value theorem, we get
\begin{align}
\left|\E U_{t}^{\eps}(\rho_{0},\zeta_{0})-\E U_{t}^{\eps}(\hat{\rho}_{0}^{n},\hat{\zeta}_{0}^{n})\right| & \le\left|\E U_{t}^{\eps}(\rho_{0},\zeta_{0})-\E U_{t}^{\eps}(\rho_{0},\hat{\zeta}_{0}^{n})\right|\nonumber \\
 & +\left|\E U_{t}^{\eps}(\rho_{0},\hat{\zeta}_{0}^{n})-\E U_{t}^{\eps}(\hat{\rho}_{0}^{n},\hat{\zeta}_{0}^{n})\right|\label{eq:estimate_EU_for_for_initial_datas}\\
 & \le\left|\E U_{t}^{\eps}(\rho_{0},\hat{\zeta}_{0}^{n})-\E U_{t}^{\eps}(\hat{\rho}_{0}^{n},\hat{\zeta}_{0}^{n})\right|\nonumber \\
 & +\norm{D_{1}U_{t}^{\eps}}_{\Cf}\norm{\rho_{0}-\hat{\rho}_{0}^{n}}_{\H[J]}.\nonumber 
\end{align}
Recall that 
\[
\norm{U_{t}^{\eps}}_{\Cf_{l,HS}^{1,3}}\le C_{I,T}(\norm{\Phi_{\eps}'}_{\Cf}+1)\norm F_{\Cf_{l,HS}^{1,3}},
\]
according to Proposition \ref{prop:fokker-plank-equation}. Moreover,
by Proposition \ref{prop:berry_essen_bound}, 
\begin{align*}
\left|\E U_{t}^{\eps}(\rho_{0},\zeta_{0})-\E U_{t}^{\eps}(\rho_{0},\hat{\zeta}_{0}^{n})\right| & \le C_{I}\left(\frac{1}{n^{1\wedge\frac{d}{2}}}(1+\norm{\nabla\rho_{0}}_{\Cf})+\norm{\rho_{0}^{n}-\rho_{0}}_{n}\right)\norm{U_{t}^{\eps}}_{\Cf_{l,HS}^{3}}.
\end{align*}
We can also estimate $\norm{\nabla\rho_{0}}_{\Cf}\le\norm{\rho_{0}}_{\H[J]}$
and
\[
\norm{\rho_{0}^{n}-\rho_{0}}_{n}\le\norm{\hat{\rho}_{0}^{n}-\rho_{0}}_{\Cf}\le\norm{\hat{\rho}_{0}^{n}-\rho_{0}}_{\H[J]}.
\]
Consequently, 
\begin{align}
 & \left|\E U_{t}^{\eps}(\rho_{0},\zeta_{0})-\E U_{t}^{\eps}(\hat{\rho}_{0}^{n},\hat{\zeta}_{0}^{n})\right|\nonumber \\
 & \qquad\le C_{I,T}(\norm{\Phi_{\eps}'}_{\Cf}+1)\norm F_{\Cf_{l,HS}^{1,3}}\left(\frac{1}{n^{1\wedge\frac{d}{2}}}(1+\norm{\rho_{0}}_{\H[J]})+\norm{\hat{\rho}_{0}^{n}-\rho_{0}}_{\H[J]}\right).\label{eq:main_inequality_for_U_eps}
\end{align}
Combining the inequalities (\ref{eq:main_inequality_for_rho_n_initial_condition}),
(\ref{eq:main_inequality_for_U_eps}) and using the uniform bound
of $\norm{\Phi_{\eps}'}_{\Cf}$ in $\eps$, we get that there exists
a constant $C$ independent of $n$, $\eps$, $t$ and $F$ such that
\begin{equation}
\left|\E F(\hat{\rho}_{t}^{n},\hat{\zeta}_{t}^{n})-\E F(\rho_{t}^{\infty},\zeta_{t}^{\infty,\eps})\right|\le Ct\norm F_{\Cf_{l,HS}^{1,3}}\left(\frac{1}{n^{1\wedge\frac{d}{2}}}+\norm{\hat{\rho}_{0}^{n}-\rho_{0}}_{\H[J]}+\eps\right).\label{eq:main_estimate_with_eps}
\end{equation}

Now, making $\eps\to0+$ and using Lemma \ref{lem:whide_continuity_of_zeta},
we get the required estimate for $F\in\Cf_{l,HS}^{2,4}(\H[J],\H[-I])$
with uniformly continuous second order derivative $D_{2}F$ in $\MLOHS[2][{\H[-I]}]$.
Since the constant $C$ is independent of $F$ in the inequality (\ref{eq:main_estimate_with_eps}),
we can cover the case $F\in\Cf_{l,HS}^{1,3}(\H[J],\H[-I])$ by an
pointwise approximation argument. This completes the proof of Theorem
\ref{thm:main_result}.

\subsubsection*{Acknowledgements }

Benjamin Gess acknowledges support by the Max Planck Society through
the Research Group \textquoteleft\textquoteleft Stochastic Analysis
in the Sciences\textquotedbl . This work was funded by the Deutsche
Forschungsgemeinschaft (DFG, German Research Foundation) via -- SFB
1283/2 2021 -- 317210226, and co-funded by the European Union (ERC,
FluCo, grant agreement No. 101088488). Views and opinions expressed
are however those of the author(s)only and do not necessarily reflect
those of the European Union or of the European Research Council. Neither
the European Union nor the granting authority can be held responsible
for them. Vitalii Konarovskyi also thanks the Max Planck Institute
for Mathematics in the Sciences for its warm hospitality, where a
part of this research was carried out.

\appendix

\section{Notation and basic facts\protect\label{subsec:Preliminaries}}

The goal of this section is to introduce the basic notation that are
used throughout this work. 

\subsection{Continuous spaces}

Recall that $\T^{d}$ denotes the $d$-dimensional torus $\left(\R/\left\{ 2\pi k-\pi:\ k\in\Z\right\} \right)^{d}$.
Let $\Cf(E)$ be the space of continuous functions on $E$ and $\Cf^{m}(E)$
be a subspace of $\Cf(E)$ consisting of all $m$-times continuously
differentiable functions for $m\in\N\cup\{\infty\}$, where $E=\Td$
or $\R^{d}$. We equip $\Cf(\Td)$ and $\Cf^{m}(\Td)$ with the standard
uniform norms denoted by $\norm{\cdot}_{\Cf}$ and $\norm{\cdot}_{\Cf^{m}}$,
respectively. For $f\in\Cf^{m}(\Td)$ we write $\partial_{j}^{m}f$
for its partial derivative of $m$-th order with respect to the $j$-th
coordinate. As usual, we also set
\[
\Delta f:=\sum_{j=1}^{d}\partial_{j}^{2}f\quad\text{and}\quad\nabla f:=\left(\partial_{j}f\right)_{j\in[d]}.
\]
The set of all functions from $\Cf^{m}(\R^{d})$ that have bounded
derivatives to the $m$-th order is denoted by $\Cf_{l}^{m}(\R^{d})$.
The subset of $\Cf_{l}^{m}(\R^{d})$ consisting of all bounded functions
is denoted by $\Cf_{b}^{m}(\R^{d})$.

\textbf{Sobolev spaces.} Let $\L$ denote the Hilbert space of square-integrable
real-valued functions on $\T^{d}$ with respect to the Lebesgue measure.
The inner product on $\L$ associated with the normalized Lebasgue
measure is denoted by $\inn{\cdot,\cdot}$ and the corresponding norm
by $\norm{\cdot}$. To define a basis on $\L$ we split $\Z^{d}\backslash\{0\}$
on two disjoint subsets $\Z_{1}^{d}$ and $\Z_{2}^{d}$ such that
$\Z_{1}^{d}=-\Z_{2}^{d}$ and take
\[
\tilde{\varsigma}_{k}=\begin{cases}
2\cos k\cdot x, & k\in\Z_{1}^{d},\\
2\sin k\cdot x, & k\in\Z_{2}^{d},\\
1, & k=0,
\end{cases}
\]
for all $k\in\Z^{d}$. We also consider the complex-valued functions
\[
\varsigma_{k}=\left(e^{\i k\cdot x},\ x\in\T^{d}\right),\ \ k\in\Z^{d},
\]
that form an orthonormal basis in the Hilbert space of all square-integrable
complex-valued functions on $\T^{d}$ equipped with the standard inner
product, denoted also by $\inn{\cdot,\cdot}$. Since for each $f\in\L$
and $n\in\N$ 
\[
\pr_{n}f:=\sum_{k\in\Znd}\inn{f,\varsigma_{k}}\varsigma_{k}=\sum_{k\in\Znd}\inn{f,\tilde{\varsigma}_{k}}\tilde{\varsigma}_{k}\in\L,
\]
where $\Znd=\left\{ -n,\ldots,n\right\} ^{d}$, the set of functions
$\{\tilde{\varsigma}_{k},k\in\Z^{d}\}$ is an orthonormal basis in
$\L$. To simplify many computations later on, we will prefer to work
with $\{\varsigma_{k},k\in\Z\}$. 

For $J\ge0$ we define the Sobolev space
\[
\H:=\left\{ f\in\L:\ \norm f_{H_{J}}^{2}:=\sum_{k\in\Z^{d}}\left(1+|k|^{2}\right)^{J}\left|\inn{f,\varsigma_{k}}\right|^{2}<\infty\right\} 
\]
and $\H[-J]$ as the completion of $\L$ with respect to the norm
\[
\norm f_{H_{-J}}^{2}:=\sum_{k\in\Z^{d}}\left(1+|k|^{2}\right)^{-J}\left|\inn{f,\varsigma_{k}}\right|^{2}.
\]
It is well-known that $\H[J]\subset\L\subset\H[-J]$ for each $J>0$
and $\H[-J]$ is the dual space to $\H[J]$ with respect to the relation
$\inn{\cdot,\cdot}$. We also note that the operators $\partial_{j}$
and $\Delta$ can be naturally defined on $\H[J]$ for each $J\in\R$.
Moreover, $\partial_{j}:\H\to\H[J-1]$ and $\Delta:\H[J]\to\H[J-2]$
are bounded linear operators and
\[
\norm f_{H_{J}}=\bnorm{(1+\Delta)^{J}f}_{}
\]
 for each $J\in\R$.

\textbf{Multilinear operators.} Let $(E_{i},\norm{\cdot}_{E_{i}})$,
$i\in[2]$, be arbitrary Banach spaces. The set of all continuous
symmetric multilinear operators from $E_{1}^{m}$ to $E_{2}$ is denoted
by $\MLO[m][E_{1};E_{2}]$. We equip $\MLO[m][E_{1};E_{2}]$ with
the norm
\[
\normMLO[m][A]:=\sup_{\norm{x_{j}}_{E_{1}}\le1}\norm{A[x_{1},\ldots,x_{m}]}_{E_{2}},
\]
which makes it a Banach space (see e.g. \citep[Section 1.8]{Cartan:1971}
for more details). If $E_{2}=\R$, we simply write $\MLO[m][E_{1}]$
instead of $\MLO[m][E_{1};\R]$. If $E_{1}$ is a separable Hilbert
space with an orthonormal basis $\{z_{l},l\in\N\}$ and $E_{2}=\R$,
we define the space of Hilbert-Schmidt multilinear operators by
\[
\MLOHS[m][E_{1}]:=\left\{ A\in\MLO[m][E_{1}]:\norm A_{\cL_{m}^{HS}}^{2}:=\sum_{(l_{j})\in\N^{m}}\left|A[z_{l_{1}},\ldots,z_{l_{m}}]\right|^{2}<\infty\right\} .
\]
Note that the space $\MLOHS[m][E_{1}]$ can be defined iteratively
as the space of all Hilbert-Schmidt operators from $E_{1}$ to $\MLOHS[m-1][E_{1}]$
for $m\ge2$, where $\MLOHS[1][E_{1}]$ is identified with $E_{1}$
via the Riesz representation theorem, and then $\norm{\cdot}_{\cL_{m}^{HS}}$
coincides with the usual Hilbert-Schmidt norm. In particular, for
$E_{1}=\H[J]$ and $m=2$ one has
\begin{align}
\norm A_{\cL_{2}^{HS}}^{2} & :=\sum_{k,l\in\Z^{d}}(1+|k|^{2})^{-J}(1+|l|^{2})^{-J}\left|A[\varsigma_{k},\varsigma_{l}]\right|^{2}\nonumber \\
 & =\sum_{k,l\in\Z^{d}}(1+|k|^{2})^{-J}(1+|l|^{2})^{-J}\left|A[\tilde{\varsigma}_{k},\tilde{\varsigma}_{l}]\right|^{2}.\label{eq:hilbert_schmidt_norm}
\end{align}
A simple computation shows that $\MLO[m][{\H[J]}]$ is continuously
embedded into $\MLOHS[m][{\H[I]}]$ for $I>J+d/2$, i.e. the restriction
of $A\in\MLO[m][{\H[J]}]$ to $\left(\H[I]\right)^{m}$ belongs to
$\MLOHS[m][{\H[I]}]$ and
\begin{equation}
\norm A_{\cL_{m}^{HS}(H_{I})}\le C_{I,J,m}\norm A_{\cL_{m}(H_{J})},\label{eq:hilbert_schmidt_embedding}
\end{equation}
where the constant $C_{I,J,m}$ depends on $I,J$ and $m$.

For each $J\in\R$ and $A\in\MLO[2][{\H[J]}]$ we define the symmetric
multilinear operator
\[
\partial_{j}^{\otimes2}A[f,g]=A[\partial_{j}f,\partial_{j}g],\quad f,g\in\H[J+1],
\]
that belongs to $\MLO[2][{\H[J+1]}]$. Moreover, it is easily seen
that
\[
\bnorm{\partial_{j}^{\otimes2}A}_{\MLO[2][H_{J+1}]}\le\norm A_{\MLO[2][H_{J}]}.
\]

We will need a bounded linear operator $\Tr:\MLOHS[2][{\H[-J]}]\to\H[I]$
such that $\Tr K_{a}=a(x,x)$, where $K_{a}$ denotes the kernel multilinear
operator for a kernel $a:\left(\Td\right)^{2}\to\R$. Since the $\delta_{x}$-function
belongs to $\H[-J]$ for $J>\frac{d}{2}$, we define the function
$\Tr A:\Td\to\R$ by
\[
\Tr A(x)=A[\delta_{x},\delta_{x}].
\]
It is continuous and, by Lemma \ref{lem:tr_operator} below, belongs
to $\H[I]$ for each $I<J-\frac{d}{2}$. 

\textbf{Derivatives on Banach spaces.} Let $\Cf(E_{1};E_{2})$ be
the space of continuous functions from a Banach space $E_{1}$ to
a Banach space $E_{2}$. The subspace of $\Cf(E_{1};E_{2})$ of $m$-times
continuously Frechet differentiable functions\footnote{See \citep[Section 5]{Cartan:1971}}
is denoted by $\Cf^{m}(E_{1};E_{2})$. The subspace of $\Cf^{m}(E_{1};E_{2})$
of all bounded functions together with their derivatives to the $m$-th
order is denoted by $\Cf_{b}^{m}(E_{1};E_{2})$. We will simply write
$\Cf(E_{1})$, $\Cf^{m}(E_{1})$, $\Cf_{b}^{m}(E_{1})$ instead of
$\Cf(E_{1};\R)$, $\Cf^{m}(E_{1};\R)$, $\Cf_{b}^{m}(E_{1};\R)$,
respectively. Note that for each $k\in[m]:=\{1,\ldots,m\}$ the $k$-th
derivative $D^{k}F(x)$ of $F\in\Cf^{m}(E_{1};E_{2})$ at $x\in E_{1}$
can be identified with a continuous symmetric multilinear operator
from $\cL_{k}(E_{1};E_{2}).$ The set of functions $F$ from $\Cf^{m}(E_{1})$
whose derivatives $D^{k}F$ are bounded (in $\norm{\cdot}_{\cL_{k}}$-norm)
functions for all $k\in[m]$ is denoted by $\Cf_{l}^{m}(E_{1})$.
Note that functions from $\Cf_{l}^{m}(E_{1})$ are not bounded in
general but they are of linear growth. The semi-norm on $\Cf_{l}^{m}(E_{1})$
is defined by
\[
\norm F_{\Cf_{l}^{m}}:=\sum_{k=1}^{m}\sup_{x\in E_{1}}\norm{D^{k}F(x)}_{\cL_{k}}.
\]
If additionally $m\ge2$ and $D^{2}F$ is an $\MLOHS[2][E_{1}]$-valued
bounded function, we write $D^{2}F\in\Cf_{l,HS}^{m}(E_{1})$ and define
\begin{equation}
\norm F_{\Cf_{l,HS}^{m}}:=\norm F_{\Cf_{l}^{m}}+\sup_{x\in E_{1}}\norm{D^{2}F(x)}_{\cL_{2}^{HS}}.\label{eq:norm_on_C_l_HS}
\end{equation}
We often identify $DF(x)$ with an element from $\H[-J]$ for each
$F\in\Cf^{1}(\H[J])$ using the dual relation $\inn{\cdot,\cdot}$
between $\H[J]$ and $\H[-J]$.
\begin{rem}
Note that $F\in\Cf_{l}^{m_{1},m_{2}}(\H[J],\H[-I'])\subset\Cf_{l,HS}^{m_{1},m_{2}}(\H[J],\H[-I])$
for $I'>I+\frac{d}{2}$, according to (\ref{eq:hilbert_schmidt_embedding})
and Lemma \ref{lem:derivative_in_different_sobolev_spaces} below.
Thus, the assumption on the boundedness of the Hilbert-Schmidt norm
of $D_{2}F$ can be replaced with the differentiability of $F$ in
a larger Sobolev space.
\end{rem}

Set $x^{\times m}:=(x,\ldots,x)\in E_{1}^{m}$ for $x\in E_{1}$.
A function $F\in\Cf^{m+1}(E_{1};E_{2})$ with bounded derivative $D^{m+1}F$
can be expanded into the Taylor series
\begin{equation}
F(x)=\sum_{k=0}^{m}\frac{1}{k!}D^{k}F(x_{0})\left[(x-x_{0})^{\times k}\right]+R_{m}(x,x_{0}),\quad x\in E_{1},\label{eq:taylors_formula}
\end{equation}
where 
\[
\norm{R_{m}(x,x_{0})}_{E_{2}}\le\frac{1}{(m+1)!}\norm{D^{m+1}F}_{\cL_{m+1}}\norm{x-x_{0}}_{E_{1}}^{m+1},
\]
according to \citep[Theorem 5.6.2]{Cartan:1971}.

The subspace of $\Cf(E_{1}\times\ldots\times E_{j})$ of all functions
that are $m_{i}$-times continuously differentiable with respect to
the $i$-th variable will be denoted by $\Cf^{m_{1},\ldots,m_{2}}(E_{1},\ldots,E_{j})$
and $D_{i}^{k}$, $i\in[j]$, will denote the corresponding partial
derivatives of the $k$-th order. We similarly introduce $\Cf_{l}^{m_{1},\ldots,m_{j}}(E_{1},\ldots,E_{j})$,
$\Cf^{1,m_{1},\ldots,m_{j}}([0,\infty),E_{1},\ldots,E_{j})$ and $\norm F_{\Cf_{l}^{m_{1},\ldots,m_{j}}}.$
If $F\in\Cf\left([0,\infty)\times E_{1}\times\ldots\times E_{j}\right)$
and it is differentiable with respect to the first (time) variable,
we use a special notation $\partial F$ for its time derivative. In
this case, all other derivatives, if they exist, are denoted by $D_{1},D_{2},...,D_{j}$,
respectively. Note that $\Cf^{m,m}(E_{1},E_{2})=C^{m}(E_{1}\times E_{2})$,
according to \citep[Proposition 2.6.2]{Cartan:1971}.

\subsection{Discrete spaces\protect\label{subsec:Discrete-spaces}}

We recall that $\T_{n}^{d}:=\left\{ \frac{2\pi k}{2n+1}:\ k\in\Z_{n}^{d}\right\} $
is the $d$-dimensional torus\footnote{The choice of the scale for the torus is motivated by our argument
that relies on the discrete/continuous Fourier expansion. In particular,
to simplify the notation, we removed the constant $2\pi$ from the
exponent in the standard Fourier basis by rescaling the torus. The
odd number of points in any direction will allow us easily to jump
between complex-valued exponential basis and the real-valued $\cos$-$\sin$
basis.}, and is considered as a subset of $\T^{d}$. The space of functions
from $\Tnd$ to $\R$ equipped with the inner product 
\[
\inn{f,g}_{n}=\frac{1}{\left(2n+1\right)^{d}}\sum_{x\in\Tnd}f(x)g(x)
\]
is denoted by $\Ln$. The corresponding norm on $\Ln$ and the maximum
norm are denoted by $\norm{\cdot}_{n}$ and $\norm{\cdot}_{n,\Cf}$,
respectively.

Following \citep[Section 5.6]{Kipnis_Landim:1999}, we can write 
\begin{equation}
f(x)=\sum_{k\in\Z_{n}^{d}}\inn{f,\varsigma_{k}}_{n}\varsigma_{k}(x)=\sum_{k\in\Znd}\inn{f,\tilde{\varsigma}_{k}}_{n}\tilde{\varsigma}_{k},\quad x\in\T_{n}^{d},\label{eq:expansion_of_f_in_Ln}
\end{equation}
for each function $f\in\Ln$ due to the equality $\inn{\varsigma_{k},\varsigma_{l}}_{n}=\delta_{k,l}$
for $k,l\in\Znd$, where $\delta_{k,l}$ is the Kronecker-Delta. 

The discrete differential operators on $\Ln$ are defined by
\begin{align*}
\partial_{n,j}f(x) & :=\frac{2n+1}{2\pi}\left(f(x+e_{j}^{n})-f(x)\right),\quad x\in\Tnd,\\
\nabla_{n}f & :=\left(\partial_{n,j}f\right)_{j\in[d]}
\end{align*}
and
\[
\Delta_{n}f(x)=\frac{(2n+1)^{2}}{4\pi^{2}}\sum_{j=1}^{d}\left(f(x+e_{j})+f(x-e_{j})-2f(x)\right),\quad x\in\Tnd,
\]
where $e_{j}=e_{j}^{n}$ denote the canonical vectors, that is, $e_{j}^{n}=\left(\frac{2\pi}{2n+1}\I_{\{i=j\}}\right)_{i\in[d]}$
and we used the normalization constant $\frac{1}{2\pi}$ for the sake
of conformity with the continuous derivatives. A simple computation
shows that
\begin{equation}
\inn{\Delta_{n}f,g}_{n}=\inn{f,\Delta_{n}g}_{n}\quad\text{and}\quad\inn{\partial_{n,j}f,g}_{n}=-\inn{f,\partial_{n,j}g}_{n}\label{eq:integration_by_parts_in_discrete_space}
\end{equation}
for each $f,g\in\Ln$. We also note that for each $k\in\Z_{n}^{d}$
the equalities
\begin{equation}
\partial_{n,j}\varsigma_{k}=\mu_{k,j}^{n}\varsigma_{k}\quad\mbox{and}\quad\Delta_{n}\varsigma_{k}=-\lambda_{k}^{n}\varsigma_{k}\label{eq:eignevalues_for_discrete_basis}
\end{equation}
hold with $\mu_{k,j}^{n}:=\frac{(2n+1)}{2\pi}\left(e^{\i\frac{2\pi k_{j}}{2n+1}}-1\right)$
and $\lambda_{k}^{n}:=\frac{(2n+1)^{2}}{2\pi^{2}}\sum_{j=1}^{d}\left[1-\cos\frac{2\pi k_{j}}{2n+1}\right].$

For a function $f:\left(\Tnd\right)^{2}\to\R$, we define
\begin{align*}
\partial_{n,j}^{\otimes2}f(x_{1},x_{2}) & =\left(\partial_{n,j}g(x_{1},\cdot)\right)(x_{2}),\quad(x_{1},x_{2})\in\left(\Tnd\right)^{2},
\end{align*}
where $g(x_{1},x_{2})=\left(\partial_{n,j}g(\cdot,x_{2})\right)(x_{1})$,
and $\Tr f:\Tnd\to\R$ by
\[
\Tr f(x)=f(x,x),\quad x\in\Tnd.
\]

\subsubsection{Projection and extension operators}

Recall the expansion (\ref{eq:expansion_of_f_in_Ln}) for $f\in\Ln$.
Since the right hand side of this expansion is a well-defined smooth
real-valued function on $\Td$, we will use it for the interpolation
of $f$. More precisely, for $f\in\Ln$ define
\[
\ex_{n}f(x)=\sum_{k\in\Znd}\inn{f,\varsigma_{k}}_{n}\varsigma_{k}(x)=\sum_{k\in\Znd}\inn{f,\tilde{\varsigma}_{k}}_{n}\tilde{\varsigma}_{k}(x),\quad x\in\Td.
\]
By (\ref{eq:expansion_of_f_in_Ln}), we have
\begin{equation}
\ex_{n}f(x)=f(x),\quad x\in\Tnd.\label{eq:ex_is_interpolation}
\end{equation}
Considering a function $f$ defined on $\Td$, we will write $\ex_{n}f$
for $\ex_{n}$ applied to the restriction of $f$ to $\Tnd$.

We will also need a kind of inverse operation to $\ex_{n}$ that will
allow to transform elements from $\H[J]$ to functions on $\Tnd$
for every $J\in\R$. For the sake of this, we will use the usual projection
operator 
\[
\pr_{n}g=\sum_{k\in\Z_{n}^{d}}\inn{g,\varsigma_{k}}\varsigma_{k}=\sum_{k\in\Znd}\inn{g,\tilde{\varsigma}_{k}}\tilde{\varsigma}_{k}.
\]
For every $g\in\H[J]$, the function $\pr_{n}g$ is well-defined and
smooth on $\T^{d}$. Therefore, its restriction to $\Tnd$ is well-defined
as well, and is also denoted by $\pr_{n}g$. 

The equality
\begin{equation}
\inn{\ex_{n}f,g}=\inn{f,\pr_{n}g}_{n}\label{eq:connection_betwwen_ex_and_pr}
\end{equation}
 easily follows from the definitions of $\ex_{n}$ and $\pr_{n}$
for every $f\in\Ln$ and $g\in\H[J]$. Thus, it will be often used
to replace the discrete inner product by the continuous one and vice
versa. In particular, the equality (\ref{eq:connection_betwwen_ex_and_pr})
implies that 
\begin{equation}
\inn{\ex_{n}f,\varsigma_{k}}=\begin{cases}
\inn{f,\varsigma_{k}}_{n}, & \mbox{if}\ k\in\Z_{n}^{d},\\
0, & \mbox{otherwise,}
\end{cases}\label{eq:connection_between_ex_and_varsigma}
\end{equation}
for each $k\in\Z^{d}$. One can also easily see that 
\begin{equation}
\pr_{n}\ex_{n}f=f\quad\mbox{and}\quad\ex_{n}\pr_{n}g=\pr_{n}g\label{eq:composition_of_pr_and_ex}
\end{equation}
for all $f\in\Ln$ and $g\in\H[J]$. Thus, combining (\ref{eq:connection_betwwen_ex_and_pr})
and (\ref{eq:composition_of_pr_and_ex}), we obtain
\begin{equation}
\inn{\ex_{n}f_{1},\ex_{n}f_{2}}=\inn{f_{1},f_{2}}_{n}\quad\text{and}\quad\inn{\pr_{n}g_{1},\pr_{n}g_{2}}=\inn{\pr_{n}g_{1},\pr_{n}g_{2}}_{n}\label{eq:connection_between_continuous_and_discrete_inner_products}
\end{equation}
for all $f_{1},f_{2}\in\Ln$ and $g_{1},g_{2}\in\H[J]$. With some
abuse of notation, we set $\hat{f}:=\ex_{n}f$ and $\check{g}:=\pr_{n}g$.

For $A\in\cL_{m}(H_{J})$ we similarly define $\pr_{n}^{\otimes m}A:\left(\Tnd\right)^{m}\to\R$
by
\[
\pr_{n}^{\otimes m}A=\sum_{k\in\left(\Z_{n}^{d}\right)^{m}}A[\tilde{\varsigma}_{\times k}]\tilde{\varsigma}_{k},
\]
where 
\[
\tilde{\varsigma}_{k}=\bigotimes_{j=1}^{m}\tilde{\varsigma}_{k_{j}}\quad\text{and}\quad\tilde{\varsigma}_{\times k}=(\tilde{\varsigma}_{k_{j}})_{j\in[m]}
\]
for $k=(k_{j})_{j\in[m]}$. Similarly to $\pr_{n}f$, we will also
consider $\pr_{n}^{\otimes}A$ as a smooth function on $\Td$, that
is defined by the same expression. For $f=(f_{j})_{j\in[m]}\in(\Ln)^{m}$,
let also 
\[
\ex_{n}^{\times m}f=\left(\ex_{n}f_{j}\right)_{j\in[m]}.
\]
A simple computation yields the equality 
\begin{equation}
A\left[\ex_{n}^{\times m}f\right]=\binn{\pr_{n}^{\otimes m}A,f^{\otimes m}}_{n},\label{eq:connection_between_ex_and_pr_extension}
\end{equation}
where $f^{\otimes m}(x):=\prod_{j=1}^{m}f_{j}(x_{j})$, $x=(x_{j})_{j\in[m]}\in\left(\T_{n}^{d}\right)^{m}$
and $\inn{\cdot,\cdot}_{n}$ is the discrete inner product on $L_{2}(\T_{n}^{md}).$
We will also identify $\pr_{n}^{\otimes m}A$ with the symmetric multilinear
operator 
\[
\pr_{n}^{\otimes m}A[f]=\binn{\pr_{n}^{\otimes m}A,f^{\otimes m}},\quad f=(f_{j})_{j\in[m]}\in H_{J}^{m}.
\]

\subsection{Further notation and comments}

The natural filtration generated by a \cdl process $X_{t}$, $t\ge0$,
is denoted by $(\F_{t}^{X})_{t\ge0}.$ The distribution of a random
variable $\xi$ in a Banach space is denoted by $\law\xi$.

A constant $C$ in estimates below will be changed from line to line.
Parameters on which $C$ depends will be listed as its subscripts,
e.g. $C_{J,I}$ will mean that the constant depends on parameters
$J,I$. Since the dimension $d$ is fixed, we will not further point
out the dependence on $d$ in constants.

\section{Some operators on Sobolev spaces\protect\label{sec:Some-operators-on}}

In this section, we will prove some basic properties of $\pr_{n}$,
$\ex_{n}$ and multilinear operators on Sobolev spaces.

\subsection{$\protect\pr_{n}$ and $\protect\ex_{n}$ operators\protect\label{subsec:pr_and_ex_operators}}

Recall that $\varsigma_{k}(x)=e^{\i k\cdot x},\ x\in\T^{d},\ k\in\Z^{d}$. 
\begin{lem}
\label{lem:estimate_of_discrete_eigenvalues}For each $n\in\N$, $j\in[d]$
and $k\in\Z_{n}^{d}$ the equalities 
\begin{equation}
\partial_{n,j}\varsigma_{k}=\mu_{k,j}^{n}\varsigma_{k}\quad\mbox{and}\quad\Delta_{n}\varsigma_{k}=-\lambda_{k}^{n}\varsigma_{k}\label{eq:eigenvalues_for_discrete_operators}
\end{equation}
hold with $\mu_{k,j}^{n}=\frac{(2n+1)}{2\pi}\left(e^{\i\frac{2\pi k_{j}}{2n+1}}-1\right)$
and $\lambda_{k}^{n}=\frac{(2n+1)^{2}}{2\pi^{2}}\sum_{j=1}^{d}\left[1-\cos\frac{2\pi k_{j}}{2n+1}\right].$
Moreover, 
\begin{equation}
\frac{|k_{j}|}{\sqrt{3}}\le|\mu_{k,j}^{n}|\le|k_{j}|\quad\mbox{and}\quad\frac{|k|^{2}}{3}\le\lambda_{k}^{n}\le|k|^{2}.\label{eq:estimates_of_discrete_eigenvalues}
\end{equation}
\end{lem}

\begin{proof}
The equalities (\ref{eq:eigenvalues_for_discrete_operators}) directly
follows from simple computations. The inequalities (\ref{eq:estimates_of_discrete_eigenvalues})
follows from 
\[
\frac{x^{2}}{3}\le\left|e^{\i x}-1\right|^{2}=\left(\cos x-1\right)^{2}+\sin^{2}x\le x^{2}
\]
and 
\[
\frac{x^{2}}{6}\le\left|1-\cos x\right|\le\frac{x^{2}}{2}
\]
for all $x\in[-\pi,\pi]$.
\end{proof}
We next recall that for each $f\in\Ln$ and $g\in\H[J]$
\[
\ex_{n}f=\sum_{k\in\Z_{n}^{d}}\inn{f,\varsigma_{k}}_{n}\varsigma_{k}\quad\mbox{and}\quad\pr_{n}g=\sum_{k\in\Z_{n}^{d}}\inn{g,\varsigma_{k}}\varsigma_{k}
\]
that are smooth functions on $\Td$. Moreover, the equality
\begin{equation}
\inn{\ex_{n}f,g}=\inn{f,\pr_{n}g}_{n}\label{eq:connection_between_descrete_inner_prod_and_continuous_one}
\end{equation}
holds. It directly follows from the fact that $\inn{\varsigma_{k},\varsigma_{l}}=\delta_{k,l}$
for all $k,l\in\Z^{d}$ and $\inn{\varsigma_{k},\varsigma_{l}}_{n}=\delta_{k,l}$
for all $k,l\in\Znd$. We next collect the basic properties of the
operator $\pr_{n}$.
\begin{lem}
\label{lem:basic_properties_of_pr}The following statements holds.
\begin{enumerate}
\item[(i)]  For each $J\in\R$ and $g\in\H[J]$
\[
\pr_{n}g\to g\quad\text{in}\ \H[J]\quad\text{and}\quad\norm{\pr_{n}g}_{H_{J}}\le\norm g_{H_{J}}.
\]
\item[(ii)]  Let $m\ge0$ and $J>m+\frac{d}{2}$. Then every function $g\in\H[J]$
has $m$ times continuously differentiable version, denoted also by
$g$, such that 
\[
\norm g_{\Cf^{m}}\le C_{m,J}\norm g_{H_{J}}.
\]
\item[(iii)]  For each $J\ge0$, $m:=\lceil J\rceil$ and every $g\in\Cf^{m}(\Td)$
\[
\norm g_{H_{J}}\le C_{m}\norm g_{\Cf^{m}}.
\]
\item[(iv)]  For each $J,I\in\R$, $J<I$, $g\in\H[I]$ and $n\ge1$
\[
\norm{g-\pr_{n}g}_{H_{J}}\le\frac{1}{n^{I-J}}\norm{g-\pr_{n}g}_{H_{I}}.
\]
In particular, for each $m\in\N_{0}$, $p\ge0$ and $J>m+p+\frac{d}{2}$
one has
\[
\norm{g-\pr_{n}g}_{\Cf^{m}}\le\frac{C_{m,p,J}}{n^{p}}\norm g_{H_{J}}.
\]
\end{enumerate}
\end{lem}

\begin{proof}
The statement (i) directly follows from the definitions of $\pr_{n}g$
and the norm in $\H[J]$.

The statement (ii) is the well-known Sobolev embedding theorem.

Using integration-by-parts, we next estimate
\begin{align*}
\norm g_{H_{J}}^{2} & =\sum_{k\in\Z^{d}}\left(1+|k|^{2}\right)^{J}\left|\inn{g,\varsigma_{k}}\right|^{2}\le\sum_{k\in\Z^{d}}\left(1+|k|^{2}\right)^{m}\left|\inn{g,\varsigma_{k}}\right|^{2}\\
 & \le C_{m}\sum_{k\in\Z^{d}}\left(1+|k|^{2m}\right)\left|\inn{g,\varsigma_{k}}\right|^{2}\\
 & \le C_{m}\sum_{k\in\Z^{d}}\left(1+\sum_{j=1}^{d}|k_{j}|^{2m}\right)\left|\inn{g,\varsigma_{k}}\right|^{2}\\
 & =C_{m}\sum_{k\in\Z^{d}}\left|\inn{g,\varsigma_{k}}\right|^{2}+C_{m}\sum_{j=1}^{d}\sum_{k\in\Z^{d}}\left|\inn{\partial_{j}^{m}g,\varsigma_{k}}\right|^{2}\\
 & =C_{m}\norm g^{2}+C_{m}\sum_{j=1}^{d}\norm{\partial_{j}^{m}g}^{2}\le C_{m}\norm g_{\Cf^{m}}.
\end{align*}
This implies (iii). 

According to the definition of $\pr_{n}g$, we have
\begin{align*}
\norm{g-\pr_{n}g}_{H_{J}}^{2} & =\sum_{k\not\in\Z_{n}^{d}}\left(1+|k|^{2}\right)^{J}\left|\inn{g,\varsigma_{k}}\right|^{2}\\
 & =\sum_{k\not\in\Z_{n}^{d}}\frac{\left(1+|k|^{2}\right)^{I}}{(1+|k|^{2})^{I-J}}\left|\inn{g,\varsigma_{k}}\right|^{2}\le\frac{1}{n^{2(I-J)}}\norm{g-\pr_{n}g}_{H_{I}}^{2}.
\end{align*}
The second part of (iv) directly follows from the first one and (ii).
The proof of the lemma is complete.
\end{proof}
\begin{lem}
\label{lem:continuity_of_pr_and_ex}The linear maps $\pr_{n}:\H[J]\to\Ln$
and $\ex_{n}:\Ln\to\H[J]$ are continuous for each $J\in\R$. Moreover,
$\pr_{n}\ex_{n}=\id$ and $\ex_{n}\pr_{n}=\pr_{n}$, where $\id$
denotes the identity operator and $\pr_{n}$ in the right hand side
of the second equality is considered as a map from $\H[J]$ to $\H[J]$.
\end{lem}

\begin{proof}
We first show the continuity of $\ex_{n}$, that will follow from
its boundedness. Take $f\in\Ln$ and estimate
\begin{align*}
\norm{\ex_{n}f}_{H_{J}}^{2} & =\sum_{k\in\Z^{d}}(1+|k|^{2})^{J}\left|\inn{\ex_{n}f,\varsigma_{k}}\right|^{2}\\
 & =\sum_{k\in\Z^{d}}(1+|k|^{2})^{J}\left|\inn{f,\pr_{n}\varsigma_{k}}_{n}\right|^{2}\\
 & =\sum_{k\in\Znd}(1+|k|^{2})^{J}\left|\inn{f,\varsigma_{k}}_{n}\right|^{2}\\
 & \le\left[(1+|n|^{2})^{J}\vee1\right]\sum_{k\in\Znd}\left|\inn{f,\varsigma_{k}}_{n}\right|^{2}\\
 & =\left[(1+|n|^{2})^{J}\vee1\right]\norm f_{n}^{2}.
\end{align*}
Thus, $\ex_{n}$ is a bounded linear operator. 

The boundedness of $\pr_{n}$ follows from the estimate
\begin{align*}
\norm{\pr_{n}g}_{n} & =\sup_{f\in\Ln}\frac{\inn{\pr_{n}g,f}_{n}}{\norm f_{n}}=\sup_{f\in\Ln}\frac{\inn{g,\ex_{n}f}}{\norm f_{n}}\\
 & \le\sup_{f\in\Ln}\norm g_{H_{J}}\frac{\norm{\ex_{n}f}_{H_{-J}}}{\norm f_{n}}\le\norm g_{H_{J}}\left[(1+|n|^{2})^{-J}\vee1\right].
\end{align*}
for each $g\in\H[J]$, where we used the boundedness of $\ex_{n}$
in $\H[-J]$. 

Now for $f\in\Ln$ and $g\in\H[J]$ we get
\[
\pr_{n}\ex_{n}f=\pr_{n}\sum_{k\in\Z_{n}^{d}}\inn{f,\varsigma_{k}}_{n}\varsigma_{k}=f
\]
and
\begin{align*}
\ex_{n}\pr_{n}g & =\ex_{n}\sum_{k\in\Znd}\inn{g,\varsigma_{k}}\varsigma_{k}=\sum_{k\in\Znd}\inn{g,\varsigma_{k}}\ex_{n}\varsigma_{k}\\
 & =\sum_{k,l\in\Znd}\inn{g,\varsigma_{k}}\inn{\varsigma_{k},\varsigma_{l}}_{n}\varsigma_{l}=\pr_{n}g.
\end{align*}
This completes the proof of the lemma.
\end{proof}
\begin{cor}
\label{cor:connection_between_descrete_and_cont_inner_product} Let
$f_{1},f_{2}\in\Ln$ and $g_{1},g_{2}\in\H[J]$. Then $\inn{f_{1},f_{2}}_{n}=\inn{\ex_{n}f_{1},\ex_{n}f_{2}}$
and $\inn{\pr_{n}g_{1},\pr_{n}g_{2}}_{n}=\inn{\pr_{n}g_{1},\pr_{n}g_{2}}$
for each $n\ge1$.
\end{cor}

\begin{proof}
By Lemma \ref{lem:continuity_of_pr_and_ex}, $f_{1}=\pr_{n}\ex_{n}f_{1}$.
Thus, $\inn{f_{1},f_{2}}_{n}=\inn{\pr_{n}\ex_{n}f_{1},f_{2}}_{n}=\inn{\ex_{n}f_{1},\ex_{n}f_{2}}$
due to (\ref{eq:connection_between_descrete_inner_prod_and_continuous_one}).
The second equality follows from the first one by taking $f_{i}=\pr_{n}g_{i}$
and using the fact that $\ex_{n}\pr_{n}=\pr_{n}$.
\end{proof}
\begin{rem}
The last two equalities in the proof of Lemma \ref{lem:continuity_of_pr_and_ex}
implies that for each $f\in\Ln$ and $g\in\H[J]$ 
\[
\pr_{n}\ex_{n}f(x)=\ex_{n}f(x)\quad\mbox{and}\quad\ex_{n}\pr_{n}g(x)=\pr_{n}g(x)
\]
for all $x\in\Td$.
\end{rem}

We will next focus on the approximating properties of the operator
$\ex_{n}$. Recall that considering a function $f:\Td\to\R$, we write
$\ex_{n}f$ for the operator $\ex_{n}$ applied to the restriction
of $f$ to the set $\Tnd$.
\begin{lem}
\label{lem:estimate_of_phi-ex_phi}Let $J\ge0$ and $m\in\N$ such
that $2m>J+1+\frac{d}{2}$. Then for each $f\in\Cf^{2m+1}(\T^{d})$
one has
\[
\norm{\ex_{n}f-f}_{H_{J}}\le\frac{C_{J,m}}{n}\norm f_{\Cf^{2m+1}}
\]
for all $n\ge1$. 
\end{lem}

\begin{proof}
Using the triangle inequality, we first get
\[
\norm{\ex_{n}f-f}_{H_{J}}\le\norm{\ex_{n}f-\pr_{n}f}_{H_{J}}+\norm{\pr_{n}f-f}_{H_{J}},
\]
where the second term in the right hand side of the estimate above
can be bounded by
\[
\norm{\pr_{n}f-f}_{H_{J}}\le\frac{1}{n}\norm{\pr_{n}f-f}_{H_{J+1}}\le\frac{2}{n}\norm f_{H_{J+1}}\le\frac{C_{m}}{n}\norm f_{\Cf^{2m+1}},
\]
according to Lemma \ref{lem:basic_properties_of_pr} and the fact
that $\lceil J\rceil+1\le2m+1$. The square of the first term can
be rewritten as
\begin{align*}
\norm{\ex_{n}f-\pr_{n}f}_{H_{J}}^{2} & =\sum_{k\in\Z_{n}^{d}}\left(1+|k|^{2}\right)^{J}\left|\inn{f,\varsigma_{k}}_{n}-\inn{f,\varsigma_{k}}\right|^{2}.
\end{align*}
Thus, we will need to estimate the difference of discrete and continuous
Fourier coefficients. Using the integration-by-parts formula and Lemma
\ref{lem:estimate_of_discrete_eigenvalues}, we get
\begin{align*}
\left|\inn{f,\varsigma_{k}}_{n}-\inn{f,\varsigma_{k}}\right| & =\left|\frac{1}{\left(\lambda_{k}^{n}\right)^{m}}\inn{f,\Delta_{n}^{m}\varsigma_{k}}_{n}-\frac{1}{|k|^{2m}}\inn{f,\Delta^{m}\varsigma_{k}}\right|\\
 & =\left|\frac{1}{\left(\lambda_{k}^{n}\right)^{m}}\inn{\Delta_{n}^{m}f,\varsigma_{k}}_{n}-\frac{1}{|k|^{2m}}\inn{\Delta^{m}f,\varsigma_{k}}\right|\\
 & \le\left|\frac{1}{\left(\lambda_{k}^{n}\right)^{m}}-\frac{1}{|k|^{2m}}\right|\left|\inn{\Delta_{n}^{m}f,\varsigma_{k}}\right|\\
 & +\frac{1}{|k|^{2m}}\left|\inn{\Delta_{n}^{m}f,\varsigma_{k}}_{n}-\inn{\Delta^{m}f,\varsigma_{k}}\right|
\end{align*}
for $k\in\Z_{n}^{d}\setminus\{0\}$.

Note that 
\[
\left|\frac{1}{\left(\lambda_{k}^{n}\right)^{m}}-\frac{1}{|k|^{2m}}\right|=\frac{1}{\left(\lambda_{k}^{n}\right)^{m}}\left|1-\left(\frac{(2n+1)^{2}}{2\pi^{2}|k|^{2}}\sum_{j=1}^{d}\left[1-\cos\frac{2\pi k_{j}}{2n+1}\right]\right)^{m}\right|.
\]
By Taylor's formula 
\[
\cos x=1-\frac{x^{2}}{2}+\frac{\cos\theta(x)}{4!}x^{4},
\]
where $\theta:\R\to\R$ is a function, we get for each $j\in[d]$
\begin{align*}
\frac{(2n+1)^{2}}{2\pi^{2}|k|^{2}}\sum_{j=1}^{d}\left[1-\cos\frac{2\pi k_{j}}{2n+1}\right] & =\frac{(2n+1)^{2}}{2\pi^{2}|k|^{2}}\sum_{j=1}^{d}\left[\frac{2\pi^{2}k_{j}^{2}}{(2n+1)^{2}}+\frac{\cos\theta_{j}(n)}{4!}\frac{16\pi^{4}k_{j}^{4}}{(2n+1)^{4}}\right]\\
 & =1+\frac{\pi^{2}}{3|k|^{2}(2n+1)^{2}}\sum_{j=1}^{d}\cos\theta_{j}(n)k_{j}^{4}=:1+z_{k}^{n},
\end{align*}
where $\theta_{j}(n):=\theta(2\pi k_{j}/(2n+1))$ and $|z_{k}^{n}|\le\frac{C|k|^{2}}{n^{2}}$
for all $k\in\Znd\backslash\{0\}$, $n\ge1$ and a constant $C>0$
is independent of $n$ and $k$. Consequently, using Taylor's formula
again for the function $x\mapsto(1+x)^{m}$, we obtain
\begin{align}
\left|\frac{1}{\left(\lambda_{k}^{n}\right)^{m}}-\frac{1}{|k|^{2m}}\right| & =\frac{1}{\left(\lambda_{k}^{n}\right)^{m}}\left|1-\left(1+z_{k}^{n}\right)^{m}\right|\le\frac{C_{m}|k|^{2}}{\left(\lambda_{k}^{n}\right)^{m}n^{2}}\le\frac{C_{m}}{|k|^{2m-2}n^{2}}\label{eq:estimate_for_eigenvalues}
\end{align}
for each $k\in\Z_{n}^{d}\setminus\{0\}$ and $n\in\N$, where we used
Lemma \ref{lem:estimate_of_discrete_eigenvalues} in the last step. 

By Taylor's formula, there exists a constant $C>0$ such that 
\[
\left|\inn{\Delta_{n}^{m}f,\varsigma_{k}}_{n}-\inn{\Delta^{m}f,\varsigma_{k}}\right|\le\frac{C}{n}\norm f_{\Cf^{2m+1}}.
\]
Combining the obtained estimates together, we conclude
\begin{align*}
\left|\inn{f,\varsigma_{k}}_{n}-\inn{f,\varsigma_{k}}\right| & \le\frac{C_{m}}{n^{2}|k|^{2m-2}}\left|\inn{\Delta_{n}^{m}f,\varsigma_{k}}\right|+\frac{C}{n|k|^{2m}}\norm f_{\Cf^{2m+1}}\\
 & \le\frac{C_{m}}{n|k|^{2m-1}}\norm f_{\Cf^{2m+1}}
\end{align*}
for all $k\in\Z_{n}^{d}\setminus\{0\}$ and $n\in\N$. Similarly,
we can estimate
\[
\left|\inn{f,\varsigma_{0}}_{n}-\inn{f,\varsigma_{0}}\right|\le\frac{C}{n}\norm f_{\Cf^{1}}.
\]
Consequently, 
\begin{align*}
\norm{\ex_{n}f-\pr_{n}f}_{H_{J}}^{2} & \le\frac{C^{2}}{n^{2}}\norm f_{\Cf^{1}}^{2}+\frac{C_{m}^{2}}{n^{2}}\norm f_{\Cf^{2m+1}}^{2}\sum_{k\in\Z_{n}^{d}\backslash\{0\}}\frac{\left(1+|k|^{2}\right)}{|k|^{4m-2}}^{J}\\
 & \le\frac{C_{J,m}}{n^{2}}\norm f_{\Cf^{2m+1}}^{2},
\end{align*}
since $2m-1-J>\frac{d}{2}$. This completes the proof of the statement.
\end{proof}
\begin{lem}
\label{lem:exchange_of_ex_and_square} For all $f\in\Ln$ and $n\ge1$
one has
\[
\bnorm{\ex_{n}f^{2}-\left(\ex_{n}f\right)^{2}}\le\frac{C_{m}}{n}\bnorm{\ex_{n}f}_{\Cf^{\lceil d/2\rceil+4}}^{2}.
\]
\end{lem}

\begin{proof}
We first note that $f(x)=\ex_{n}f(x)$ for all $x\in\Tnd$, according
to (\ref{eq:ex_is_interpolation}). Since $\ex_{n}f^{2}$ is only
defined by values of $f^{2}$ on $\Tnd$, $\ex_{n}f^{2}=\ex_{n}\left(\ex_{n}f\right)^{2}$.
Hence, we can estimate
\begin{align*}
\bnorm{\ex_{n}f^{2}-\left(\ex_{n}f\right)^{2}} & =\bnorm{\ex_{n}\left(\ex_{n}f\right)^{2}-\left(\ex_{n}f\right)^{2}}\\
 & \le\frac{C_{m}}{n}\bnorm{\left(\ex_{n}f\right)^{2}}_{\Cf^{2m+1}}\le\frac{C_{m}}{n}\bnorm{\ex_{n}f}_{\Cf^{2m+1}}^{2}
\end{align*}
due to Lemma \ref{lem:estimate_of_phi-ex_phi} with $J=0$ and $m\in\N$
satisfying $\frac{d}{2}+2<2m+1\le\lceil d/2\rceil+4$. This completes
the proof of the statement.
\end{proof}
\begin{lem}
\label{lem:estimate_of_norm_ex_partial} For each $J\in\R$, $n\in\N$,
$j\in[d]$ and $f\in\Ln$, one has
\[
\norm{\ex_{n}\partial_{n,j}f}_{H_{J}}\le\norm{\ex_{n}f}_{H_{J+1}}.
\]
\end{lem}

\begin{proof}
Using Lemma \ref{lem:estimate_of_discrete_eigenvalues}, we estimate
\begin{align*}
\norm{\ex_{n}\partial_{n,j}f}_{H_{J}}^{2} & =\sum_{k\in\Z^{d}}\left(1+|k|^{2}\right)^{J}\left|\inn{\ex_{n}\partial_{n,j}f,\varsigma_{k}}\right|^{2}=\sum_{k\in\Z_{n}^{d}}\left(1+|k|^{2}\right)^{J}\left|\inn{\partial_{n,j}f,\varsigma_{k}}_{n}\right|^{2}\\
 & =\sum_{k\in\Z_{n}^{d}}\left(1+|k|^{2}\right)^{J}\left|\inn{f,\partial_{n,j}\varsigma_{k}}_{n}\right|^{2}=\sum_{k\in\Z_{n}^{d}}\left(1+|k|^{2}\right)^{J}\left|\mu_{k,j}^{n}\right|^{2}\left|\inn{f,\varsigma_{k}}_{n}\right|^{2}\\
 & \le\sum_{k\in\Z_{n}^{d}}\left(1+|k|^{2}\right)^{J+1}\left|\inn{f,\varsigma_{k}}_{n}\right|^{2}=\norm{\ex_{n}f}_{H_{J+1}}^{2}.
\end{align*}
\end{proof}
\begin{lem}
\label{lem:norm_of_ex_of_product}Let $J\in\N_{0}$. Then for each
$f\in\Cf^{J}(\T^{d})$ and $g\in\H[J]$ one has
\[
\norm{\ex_{n}(fg)}_{H_{J}}\le C_{J}\norm{\ex_{n}f}_{\Cf^{J}}\norm{\ex_{n}g}_{H_{J}}.
\]
\end{lem}

\begin{proof}
Using Lemma \ref{lem:estimate_of_discrete_eigenvalues} and the integration-by-parts
formula, we estimate
\begin{align*}
\norm{\ex_{n}(fg)}_{H_{J}}^{2} & \le\sum_{k\in\Znd}\left(1+|k|^{2}\right)^{J}\left|\inn{fg,\varsigma_{k}}_{n}\right|^{2}\\
 & \le3^{J}\sum_{k\in\Znd}\left(1+\sum_{j=1}^{d}|\mu_{k,j}^{n}|^{2}\right)^{J}\left|\inn{fg,\varsigma_{k}}_{n}\right|^{2}\\
 & \le C_{J}\sum_{k\in\Znd}\left(1+\sum_{j=1}^{d}|\mu_{k,j}^{n}|^{2J}\right)\left|\inn{fg,\varsigma_{k}}_{n}\right|^{2}\\
 & =C_{J}\left[\norm{fg}_{n}^{2}+\sum_{j=1}^{d}\sum_{k\in\Znd}\left|\inn{\partial_{n,j}^{J}(fg),\varsigma_{k}}_{n}\right|^{2}\right]\\
 & =C_{J}\left[\norm{fg}_{n}^{2}+\sum_{j=1}^{d}\bnorm{\partial_{n,j}^{J}(fg)}_{n}^{2}\right].
\end{align*}
Iterating the equality $\partial_{n,j}(fg)=\partial_{n,j}f\tau_{j}^{n}g+f\partial_{n,j}g$,
where $\tau_{j}^{n}f(x)=f(x+e_{j}^{n})$, we get 
\[
\partial_{n,j}^{J}(fg)=\sum_{l=0}^{J}\binom{J}{l}\left[\partial_{n,j}^{l}f\right]\left[(\tau_{j}^{n})^{l}\partial_{n,j}^{J-l}g\right].
\]
Thus, using the fact that $\norm{\tau_{n,j}g}_{n}=\norm g_{n}$, we
obtain
\[
\norm{\ex_{n}(fg)}_{H_{J}}^{2}\le C_{J}\left[\norm f_{n,\Cf}^{2}\norm g_{n}^{2}+\sum_{j=1}^{d}\sum_{l=0}^{J}\bnorm{\partial_{n,j}^{l}f}_{n,\Cf}^{2}\bnorm{\partial_{n,j}^{J-l}g}_{n}^{2}\right].
\]
We next note that $\ex_{n}f(x)=f(x)$ for all $x\in\Tnd$. Thus, $\norm f_{n,\Cf}\le\norm{\ex_{n}f}_{\Cf}$.
Moreover, applying Taylor's formula to $\ex_{n}f$, we get
\begin{align*}
\norm{\partial_{n,j}^{l}f}_{\Cf}=\norm{\partial_{n,j}^{l}\ex_{n}f}_{\Cf}\le C_{l}\norm{\partial_{j}^{l}\ex_{n}f}_{\Cf} & .
\end{align*}
Consequently, we can continue the estimate as follows
\[
\norm{\ex_{n}(fg)}_{H_{J}}^{2}\le C_{J}\left[\norm{\ex_{n}f}_{\Cf}^{2}\norm{\ex_{n}g}^{2}+\sum_{j=1}^{d}\sum_{l=0}^{J}\bnorm{\partial_{j}^{l}\ex_{n}f}_{\Cf}^{2}\bnorm{\ex_{n}\partial_{n,j}^{J-l}g}^{2}\right].
\]
The statement now follows from Lemma \ref{lem:estimate_of_norm_ex_partial}.
\end{proof}
\begin{lem}
\label{lem:norm_of_ex_shift}Let $J\in\R$, $n\in\N$, $j\in[d]$
and $f\in\Ln$. Then 
\[
\norm{\ex_{n}\tau_{j}^{n}f}_{H_{J}}=\norm{\ex_{n}f}_{H_{J}}.
\]
\end{lem}

\begin{proof}
The statement directly follows from the following computation
\begin{align*}
\norm{\ex_{n}\tau_{j}^{n}f}_{H_{J}}^{2} & =\sum_{k\in\Znd}\left(1+|k|^{2}\right)^{J}\left|\inn{\tau_{j}^{n}f,\varsigma_{k}}_{n}\right|^{2}\\
 & =\sum_{k\in\Znd}\left(1+|k|^{2}\right)^{J}\left|\binn{f,\left(\tau_{k}^{n}\right)^{-1}\varsigma_{k}}_{n}\right|^{2}\\
 & =\sum_{k\in\Znd}\left(1+|k|^{2}\right)^{J}\left|\binn{f,\varsigma_{k}}_{n}\right|^{2}=\norm{\ex_{n}f}_{H_{J}}.
\end{align*}
\end{proof}
\begin{lem}
\label{lem:discrete_and_continuous_laplace_operator}There exists
a constant $C>0$ such that for each $J\in\R$ and $g\in\H[J+2]$
the inequality
\[
\bnorm{\ex_{n}\Delta_{n}\pr_{n}g-\pr_{n}\Delta g}_{H_{J}}\le\frac{C}{n}\norm g_{H_{J+2}}
\]
holds.
\end{lem}

\begin{proof}
Using integration-by-parts formula Lemma \ref{lem:estimate_of_discrete_eigenvalues}
and (\ref{eq:connection_betwwen_ex_and_pr}), we compute

\begin{align*}
\bnorm{\ex_{n}\Delta_{n}\pr_{n}g-\pr_{n}\Delta g}_{H_{J}}^{2} & =\sum_{k\in\Z^{d}}(1+|k|^{2})^{J}\left|\inn{\ex_{n}\Delta_{n}\pr_{n}g,\varsigma_{k}}-\inn{\pr_{n}\Delta g,\varsigma_{k}}\right|^{2}\\
 & =\sum_{k\in\Znd}(1+|k|^{2})^{J}\left|\inn{\Delta_{n}\pr_{n}g,\varsigma_{k}}_{n}-\inn{\Delta g,\varsigma_{k}}\right|^{2}\\
 & =\sum_{k\in\Znd}(1+|k|^{2})^{J}\left|\lambda_{k}^{n}\inn{\pr_{n}g,\varsigma_{k}}_{n}-|k|^{2}\inn{g,\varsigma_{k}}\right|^{2}\\
 & =\sum_{k\in\Znd}(1+|k|^{2})^{J}\left|\lambda_{k}^{n}-|k|^{2}\right|^{2}\left|\inn{g,\varsigma_{k}}\right|^{2},
\end{align*}
where we used the equality $\inn{\pr_{n}g,\varsigma_{k}}_{n}=\inn{g,\ex_{n}\varsigma_{k}}=\inn{g,\varsigma_{k}}$
in the last step. Using Taylor's expansion, we get 
\[
\lambda_{k}^{n}=\frac{(2n+1)^{2}}{2\pi^{2}}\sum_{j=1}^{d}\left[1-\cos\frac{2\pi k_{j}}{2n+1}\right]=|k|^{2}+\frac{1}{(2n+1)^{2}}\sum_{j=1}^{d}k_{j}^{4}a_{k}^{n},
\]
where the family $|a_{k}^{n}|$, $k\in\Znd$, $n\ge1$, is bounded
by an universal constant. Thus, 
\begin{align*}
\bnorm{\ex_{n}\Delta_{n}\pr_{n}g-\pr_{n}\Delta g}_{H_{J}}^{2} & \le\frac{C}{(2n+1)^{2}}\sum_{k\in\Znd}(1+|k|^{2})^{J}\left|k\right|^{4}\left|\inn{g,\varsigma_{k}}\right|^{2}\\
 & \le\frac{C}{n^{2}}\norm g_{H_{J+2}}^{2}.
\end{align*}
This completes the proof of the lemma.
\end{proof}
\begin{cor}
\label{cor:comparison_discrete_and_continuous_heat_equations}Let
$(\rho_{t}^{\infty})_{t\ge0}$ and $(\rho_{t}^{n})_{t\ge0}$, $n\ge1$,
be solutions to (\ref{eq:heat_PDE-1}) and (\ref{eq:descrete_heat_equation}),
respectively. Let also $\rho_{0}^{\infty}\in\H[2]$. Then for each
$T>0$ there exists a constant $C_{T}>0$ such that
\[
\norm{\ex_{n}\rho_{t}^{n}-\pr_{n}\rho_{t}^{\infty}}\le C_{T}\norm{\ex_{n}\rho_{0}^{n}-\rho_{0}^{\infty}}+\frac{C_{T}}{n}\norm{\rho_{0}^{\infty}}_{\H[2]}
\]
for all $t\in[0,T]$ and $n\ge1$.
\end{cor}

\begin{proof}
Using Corollary \ref{cor:connection_between_descrete_and_cont_inner_product},
we easily get
\begin{align*}
\norm{\ex_{n}\rho_{t}^{n}-\pr_{n}\rho_{t}^{\infty}}^{2} & =\norm{\rho_{t}^{n}-\pr_{n}\rho_{t}^{\infty}}_{n}^{2}=\\
 & =\norm{\rho_{0}^{n}-\pr_{n}\rho_{0}^{\infty}}_{n}^{2}+4\pi^{2}\int_{0}^{t}\inn{\rho_{s}^{n}-\pr_{n}\rho_{s}^{\infty},\Delta_{n}\rho_{s}^{n}-\pr_{n}\Delta\rho_{s}^{\infty}}_{n}ds.
\end{align*}
Integrating-by-parts and using the Cauchy-Schwarz inequality, we can
estimate
\begin{align*}
\inn{\rho_{s}^{n}-\pr_{n}\rho_{s}^{\infty},\Delta_{n}\rho_{s}^{n}-\pr_{n}\Delta\rho_{s}^{\infty}}_{n} & =\inn{\rho_{s}^{n}-\pr_{n}\rho_{s}^{\infty},\Delta_{n}\rho_{s}^{n}-\Delta_{n}\pr_{n}\rho_{s}^{\infty}}_{n}\\
 & +\inn{\rho_{s}^{n}-\pr_{n}\rho_{s}^{\infty},\Delta_{n}\pr_{n}\rho_{s}^{\infty}-\pr_{n}\Delta\rho_{s}^{\infty}}_{n}\\
 & \le-\sum_{j=1}^{d}\bnorm{\partial_{n,j}\left(\rho_{s}^{n}-\pr_{n}\rho_{s}^{\infty}\right)}_{n}^{2}\\
 & +\bnorm{\rho_{s}^{n}-\pr_{n}\rho_{s}^{\infty}}_{n}\bnorm{\Delta_{n}\pr_{n}\rho_{s}^{\infty}-\pr_{n}\Delta\rho_{s}^{\infty}}_{n}\\
 & \le\frac{1}{2}\bnorm{\rho_{s}^{n}-\pr_{n}\rho_{s}^{\infty}}_{n}^{2}+\frac{1}{2}\bnorm{\Delta_{n}\pr_{n}\rho_{s}^{\infty}-\pr_{n}\Delta\rho_{s}^{\infty}}_{n}^{2}.
\end{align*}
According to Corollary \ref{cor:connection_between_descrete_and_cont_inner_product}
and Lemma \ref{lem:discrete_and_continuous_laplace_operator}, the
bounds
\begin{align*}
\bnorm{\Delta_{n}\pr_{n}\rho_{s}^{\infty}-\pr_{n}\Delta\rho_{s}^{\infty}}_{n}^{2} & =\bnorm{\ex_{n}\Delta_{n}\pr_{n}\rho_{s}^{\infty}-\pr_{n}\Delta\rho_{s}^{\infty}}^{2}\\
 & \le\frac{C}{n^{2}}\norm{\rho_{s}^{\infty}}_{H_{2}}^{2}\le\frac{C}{n^{2}}\norm{\rho_{0}^{\infty}}_{H_{2}}^{2}
\end{align*}
hold. Consequently, we obtain 
\[
\norm{\ex_{n}\rho_{t}^{n}-\pr_{n}\rho_{t}^{\infty}}^{2}\le\norm{\rho_{0}^{n}-\pr_{n}\rho_{0}^{\infty}}_{n}^{2}+2\pi^{2}\int_{0}^{t}\bnorm{\rho_{s}^{n}-\pr_{n}\rho_{s}^{\infty}}_{n}^{2}ds+\frac{Ct}{n^{2}}\norm{\rho_{0}^{\infty}}_{H_{2}}
\]
for all $t\ge0$. Using Grönwall's inequality, we conclude
\[
\norm{\ex_{n}\rho_{t}^{n}-\pr_{n}\rho_{t}^{\infty}}^{2}\le C_{T}\norm{\ex_{n}\rho_{0}^{n}-\pr_{n}\rho_{0}^{\infty}}^{2}+\frac{C_{T}}{n^{2}}\norm{\rho_{0}^{\infty}}_{H_{2}}^{2}
\]
that completes the proof of the corollary.
\end{proof}

\subsection{Multilinear operators on Sobolev spaces\protect\label{subsec:Multilinear-operators-on-H}}

Recall that $\MLO[m][{\H[J]}]$ denotes the space of all continuous
multilinear operators from $\left(\H\right)^{m}$ to $\R$ equipped
with the norm
\[
\normMLO[m][A]=\sup_{\norm{f_{j}}_{H_{J}}\le1}\left|A[f_{1},\ldots,f_{m}]\right|,
\]
and the subset of $\MLO[m][{\H[J]}]$ consisting of multilinear operators
with finite Hilbert-Schmidt norm (\ref{eq:hilbert_schmidt_norm})
is denoted by $\MLOHS[m][{\H[J]}]$. 

Since for each $J<I$ one has $\H[I]\subset\H[J]$ and $\norm{\cdot}_{H_{J}}\le\norm{\cdot}_{H_{I}}$,
the space $\MLO[m][{\H[J]}]$ is continuously embedded into $\MLO[m][{\H[I]}]$
and $\norm{\cdot}_{\MLO[m][{\H[I]}]}\le\norm{\cdot}_{\MLO[m][{\H[J]}]}$.
We next show the continuous embedding of $\MLOHS[m][{\H[J]}]$ into
$\MLO[m][{\H[I]}]$.
\begin{lem}
\label{lem:norm_of_multilinear_operator}For each $I,J\in\R$ with
$I>J+\frac{d}{2}$ one has
\[
\norm A_{\MLOHS[m][{\H[I]}]}\le C_{I-J,m}\norm A_{\MLO[m][{\H[J]}]}
\]
for all $A\in\MLOHS[m][{\H[I]}]$. In particular, the space $\MLOHS[m][{\H[I]}]$
is continuously embedded into $\MLO[m][{\H[J]}]$.
\end{lem}

\begin{proof}
The statement follows from the straightforward estimate 

\begin{align*}
\norm A_{\MLOHS[m][{\H[I]}]}^{2} & =\sum_{k_{1},\ldots,k_{m}\in\Z^{d}}\prod_{j=1}^{m}\left(1+|k_{j}|^{2}\right)^{-I}\left|A[\tilde{\varsigma}_{k_{1}},\ldots,\tilde{\varsigma}_{k_{m}}]\right|^{2}\\
 & \le\norm A_{\MLO[m][{\H[J]}]}\sum_{k_{1},\ldots,k_{m}\in\Z^{d}}\prod_{j=1}^{m}\left(1+|k_{j}|^{2}\right)^{-I}\norm{\tilde{\varsigma}_{j}}_{H_{J}}^{2}\\
 & \le\norm A_{\MLO[m][{\H[J]}]}\sum_{k_{1},\ldots,k_{m}\in\Z^{d}}\prod_{j=1}^{m}\left(1+|k_{j}|^{2}\right)^{-I+J}.
\end{align*}
This completes the proof of the lemma.
\end{proof}
We will further focus on the case $m=2$. Take $a\in\Cf^{J,J}(\T^{d},\T^{d})$
for some even $J\in\N_{0}$ and define the multilinear operator $K_{a}$
with kernel $a$ by
\begin{align*}
K_{a}(f,g):=\inn{f\otimes g,a} & =\sum_{k,l\in\Z^{d}}\inn{\varsigma_{(k,l)},a}\inn{f,\varsigma_{k}}\inn{g,\varsigma_{l}}\\
 & =\sum_{k,l\in\Z^{d}}\inn{a,\varsigma_{(-k,-l)}}\inn{f,\varsigma_{k}}\inn{g,\varsigma_{l}},\quad f,g\in\Cf(\Td),
\end{align*}
where $\varsigma_{(k,l)}(x,y)=\varsigma_{k}\otimes\varsigma_{l}(x,y)=\varsigma_{k}(x)\varsigma_{l}(y)$,
$x,y\in\Td$. Then the operator $K_{a}$ can be uniquely extended
to a multilinear operator on $\H[-J]$, denoted also by $K_{a}$.
Moreover, it is a Hilbert-Schmidt operator satisfying
\begin{equation}
\norm{K_{a}}_{\cL_{2}^{HS}(H_{-J})}\le C\normC[J,J][a].\label{eq:norm_of_kernel_multilinear_operator}
\end{equation}
Indeed, this directly follows from the following computation 
\begin{align*}
\norm{K_{a}}_{\cL_{2}^{HS}}^{2} & =\sum_{k,l\in\Z^{d}}(1+|k|^{2})^{J}(1+|l|^{2})^{J}\left|\inn{a,\varsigma_{(k,l)}}\right|^{2}\\
 & =\sum_{k,l\in\Z^{d}}\left|\binn{a,(1+\Delta)^{J/2}\otimes(1+\Delta)^{J/2}\varsigma_{k}\otimes\varsigma_{l}}\right|^{2}\\
 & =\sum_{k,l\in\Z^{d}}\left|\binn{(1+\Delta)^{J/2}\otimes(1+\Delta)^{J/2}a,\varsigma_{k}\otimes\varsigma_{l}}\right|^{2}\\
 & =\bnorm{(1+\Delta)^{J/2}\otimes(1+\Delta)^{J/2}a}^{2}\le C\norm a_{\Cf^{J,J}}^{2},
\end{align*}
where we have used the integration-by-parts passing from the second
to the third line. 

Since we usually work with the Fourier basis $\{\varsigma_{k},k\in\Z^{d}\}$
instead of $\{\tilde{\varsigma}_{k},k\in\Z^{d}\}$, we will extend
$A\in\cL_{m}(\H[J])$ linearly with respect to each component to the
set of complex valued square integrable function, following e.g. the
definition of the kernel operator $K_{a}$. In this case, a simple
computation shows that
\[
\pr_{n}^{\otimes m}A:=\sum_{k\in\left(\Znd\right)^{m}}A[\varsigma_{-(\times k)}]\varsigma_{k}=\sum_{k\in\left(\Znd\right)^{m}}A[\tilde{\varsigma}_{\times k}]\tilde{\varsigma}_{k}
\]
for all $n\in\N$, where $\varsigma_{\times k}=(\varsigma_{k_{j}})_{j\in[m]}$,
$\tilde{\varsigma}_{\times k}=\left(\tilde{\varsigma}_{k_{j}}\right)_{j\in[m]}$,
$\varsigma_{k}(x)=\prod_{j=1}^{m}\varsigma_{k_{j}}(x_{j})$ and $\tilde{\varsigma}_{k}(x)=\prod_{j=1}^{m}\tilde{\varsigma}_{k_{j}}(x_{j})$
for $k=(k_{j})_{j\in[m]}$ and $x=(x_{j})_{j\in[m]}$. Thus, for each
$f=(f_{j})_{j\in[m]}\in\left(\H[J]\right)^{m}$ we have
\begin{align}
\inn{f^{\otimes},\pr_{n}^{\otimes m}A} & =\sum_{k\in\left(\Znd\right)^{m}}A[\varsigma_{\times k}]\binn{f,\varsigma_{k}}=A[\pr_{n}f]\to A[f],\quad n\to\infty,\label{eq:expansion_of_multilinear_operator_in_fourier_series}
\end{align}
where $\pr_{n}f:=\left(\pr_{n}f_{j}\right)_{j\in[m]}$ and $f^{\otimes}(x)=\prod_{j=1}^{m}f_{j}(x_{j})$,
$x=(x_{j})_{j\in[m]}$, according to the continuity of $A$ and the
convergence $\pr_{n}f_{j}\to f_{j}$ in $\H[J]$. Thus, we can consider
$\pr_{n}^{\otimes m}$ as an analog of the operator $\pr_{n}$. 

The following statement is an analog of Lemma \ref{lem:basic_properties_of_pr}
(iv).
\begin{lem}
\label{lem:orthogonal_projection_for_multilinear_oper}For each $J,I\in\R$,
$J<I$, a multi-linear operator $A\in\MLOHS[2][{\H[J]}]$ and $n\ge1$
the kernel operator $K_{\pr_{n}^{\otimes2}A}$ belongs to $\MLOHS[2][{\H[J]}]$
and
\[
\bnorm{A-K_{\pr_{n}^{\otimes2}A}}_{\MLOHS[2][{\H[I]}]}^{2}\le\frac{1}{n^{I-J}}\bnorm{A-K_{\pr_{n}^{\otimes2}A}}_{\MLOHS[2][{\H[J]}]}^{2}.
\]
\end{lem}

\begin{proof}
The fact that $K_{\pr_{n}^{\otimes2}A}\in\MLOHS[2][{\H[J]}]$ follows
from the definitions of the kernel operator and (\ref{eq:norm_of_kernel_multilinear_operator}).
We next estimate
\begin{align*}
\bnorm{A-K_{\pr_{n}^{\otimes2}A}}_{\MLOHS[2][{\H[I]}]}^{2} & =\sum_{k,l\in\Z^{d}}(1+|k|^{2})^{-I}(1+|l|^{2})^{-I}\left|A[\varsigma_{k},\varsigma_{l}]-\inn{\varsigma_{(k,l)},\pr_{n}^{\otimes2}A}\right|^{2}\\
 & =\sum_{(k,l)\not\in(\Z_{n}^{d})^{2}}(1+|k|^{2})^{-I}(1+|l|^{2})^{-I}|A[\varsigma_{k},\varsigma_{l}]|^{2}\\
 & \le\frac{1}{n^{2(I-J)}}\sum_{(k,l)\not\in(\Z_{n}^{d})^{2}}(1+|k|^{2})^{-J}(1+|l|^{2})^{-J}|A[\varsigma_{k},\varsigma_{l}]|^{2}\\
 & =\frac{1}{n^{2(I-J)}}\bnorm{A-K_{\pr_{n}^{\otimes2}A}}_{\MLOHS[2][{\H[J]}]}^{2}.
\end{align*}
This completes the proof of the lemma.
\end{proof}
We next define a bounded linear operator $\Tr:\MLOHS[2][{\H[-J]}]\to\H[I]$
for some $J$ and $I$ such that $\Tr K_{a}=a(x,x)$ for a kernel
$a$. Note that the $\delta_{x}$-function belongs to $H_{-J}$ for
$J>\frac{d}{2}$ due to the inequality
\begin{align*}
\norm{\delta_{x}}_{H_{-J}}^{2} & =\sum_{k\in\Z^{d}}(1+|k|^{2})^{-J}\left|\inn{\delta_{x},\varsigma_{k}}\right|^{2}=\sum_{k\in\Z^{d}}(1+|k|^{2})^{-J}|\varsigma_{k}(x)|^{2}\\
 & =\sum_{k\in\Z^{d}}(1+|k|^{2})^{-J}<\infty.
\end{align*}
 
\begin{lem}
\label{lem:tr_operator}Let $J>\frac{d}{2}$ and $I<J-\frac{d}{2}$.
Then for each $A\in\MLOHS[2][{\H[-J]}]$ the function $\Tr A$, defined
by 
\begin{align*}
\Tr A(x) & :=A[\delta_{x},\delta_{x}]
\end{align*}
is continuous, belongs to $\H[I]$ and 
\begin{equation}
\Tr A(x)=\sum_{l\in\Z^{d}}\left(\sum_{k\in\Z^{d}}A[\varsigma_{k-l},\varsigma_{-k}]\right)\varsigma_{l}(x)\label{eq:expansion_of_Tr}
\end{equation}
for all $x\in\Td$. Moreover, $\Tr:\MLOHS[2][{\H[-J]}]\to\H[I]$ is
a bounded linear operator satisfying $\Tr K_{a}(x)=a(x,x)$, $x\in\Td$,
for each $a\in\Cf^{m,m}(\T^{d},\T^{d})$, where $m\ge J$ is an even
number.
\end{lem}

\begin{proof}
The continuity of $\Tr A$ as a map from $\Td$ to $\R$ follows from
the continuity of $A:\left(\H[-J]\right)^{2}\to\R$ and $\delta_{\cdot}:\T^{d}\to\H[-J]$.
By the definition of $\Tr$ and (\ref{eq:expansion_of_multilinear_operator_in_fourier_series}),
we have 
\begin{align*}
\Tr A(x) & =\lim_{n\to\infty}\inn{\delta_{x}\otimes\delta_{x},\pr_{n}^{\otimes2}A}=\sum_{k,l\in\Z^{d}}A[\varsigma_{k},\varsigma_{l}]\binn{\delta_{x},\varsigma_{k}}\inn{\delta_{x},\varsigma_{l}}\\
 & =\sum_{k,l\in\Z^{d}}A[\varsigma_{k},\varsigma_{l}]\varsigma_{-k}(x)\varsigma_{-l}(x).
\end{align*}
The series above converges absolutely because
\begin{align*}
\sum_{k,l\in\Z^{d}}\left|A(\varsigma_{k},\varsigma_{l})\varsigma_{-k}(x)\varsigma_{-l}(x)\right| & =\sum_{k,l\in\Z^{d}}\frac{1}{(1+|k|^{2})^{J/2}(1+|l|^{2})^{J/2}}\\
 & \qquad\qquad\qquad\cdot(1+|k|^{2})^{J/2}(1+|l|^{2})^{J/2}\left|A[\varsigma_{k},\varsigma_{l}]\right|\\
 & \le\left(\sum_{k,l\in\Z^{d}}\frac{1}{(1+|k|^{2})^{J}(1+|l|^{2})^{J}}\right)^{\frac{1}{2}}\\
 & \qquad\qquad\qquad\cdot\left(\sum_{k,l\in\Z^{d}}(1+|k|^{2})^{J}(1+|l|^{2})^{J}\left|A[\varsigma_{k},\varsigma_{l}]\right|^{2}\right)^{\frac{1}{2}}\\
 & \le C_{J}\norm A_{\cL_{2}^{HS}(H_{-J})},
\end{align*}
where we have used the Cauchy-Schwarz inequality in the second step.
Thus, we may interchange the summands in the series, to get the expression
(\ref{eq:expansion_of_Tr}). 

We next show that $\Tr A\in\H[I]$ for $I<J-\frac{d}{2}$. Similarly
to the estimate above, we conclude
\begin{align*}
\left|\sum_{k\in\Z^{d}}A[\varsigma_{k-l},\varsigma_{-k}]\right|^{2} & \le\norm A_{\cL_{2}^{HS}}^{2}\sum_{k\in\Z^{d}}\frac{1}{(1+|k-l|^{2})^{J}(1+|k|^{2})^{J}}
\end{align*}
since $J>\frac{d}{2}$. Thus, 
\begin{align*}
\norm{\Tr A}_{H_{I}}^{2} & =\sum_{l\in\Z^{d}}(1+|l|^{2})^{I}\left|\sum_{k\in\Z^{d}}A[\varsigma_{k-l},\varsigma_{-k}]\right|^{2}\\
 & \le\norm A_{\cL_{2}^{HS}}^{2}\sum_{l\in\Z^{d}}(1+|l|^{2})^{I}\sum_{k\in\Z^{d}}\frac{1}{(1+|k-l|^{2})^{J}(1+|k|^{2})^{J}}\\
 & \le\norm A_{\cL_{2}^{HS}}^{2}\sum_{k,l\in\Z^{d}}\frac{(1+|l+k|^{2})^{I}}{(1+|l|^{2})^{J}(1+|k|^{2})^{J}}<\infty
\end{align*}
due to $J-I>\frac{d}{2}$.

We note that $K_{a}\in\MLOHS[2][{\H[-J]}]$, by (\ref{eq:norm_of_kernel_multilinear_operator}),
and trivially $\Tr K_{a}(x)=K_{a}[\delta_{x},\delta_{x}]=a(x,x)$
for all $x\in\Td$. This completes the proof of the lemma.
\end{proof}

Define the mixed derivative of a multilinear operator from $\MLO[2][{\H[J]}]$
by 
\[
\partial_{j}^{\otimes2}A[f,g]=-A[\partial_{j}f,\partial_{j}g],\quad f,g\in\H[J+1].
\]
The following statement easily follows from the definition of $\partial_{j}$
on $\H[J]$.
\begin{lem}
\label{lem:norm_of_derivative_of_multilinear_operator} For each $J\in\R$
and $A\in\MLOHS[2][{\H[J]}]$ the multilinear operator $\partial_{j}^{\otimes2}A$
is well-defined and belongs to $\MLOHS[2][{\H[J+1]}]$. Moreover,
\[
\bnorm{\partial_{j}^{\otimes2}A}_{\MLOHS[2][{\H[J+1]}]}\le\norm A_{\MLOHS[2][{\H[J]}]}.
\]
\end{lem}

\begin{rem}
(i) The statement of Lemma \ref{lem:norm_of_derivative_of_multilinear_operator}
remains true, if we replace $\cL_{2}^{HS}$ by $\cL_{2}$ and $\norm{\cdot}_{\cL_{2}^{HS}}$
by $\norm{\cdot}_{\cL_{2}}$.

(ii) According integration-by-parts formula, we have the equality
$\partial_{j}^{\otimes2}K_{a}=K_{\partial_{j}^{\otimes2}a}$ for each
$j\in[d]$. 
\end{rem}

With some abuse of notation we will also set 
\[
\Tr f(x)=f(x,x),\quad x\in\Tnd,
\]
if $f\in(\Ln)^{2}$.
\begin{prop}
\label{prop:expansion_of_diffusion_term_for_ssep} Let $J-1-\frac{d}{2}>I\ge0$
and $j\in[d]$. Then for every $A\in\MLOHS[2][{\H[-J]}]$ and $n\in\N$
the estimate
\begin{equation}
\bnorm{\ex_{n}\left[\Tr\partial_{n,j}^{\otimes2}\pr_{n}^{\otimes2}A\right]}_{H_{I}}\le C_{J,I}\norm A_{\MLOHS[2][{\H[-J]}]}\label{eq:estimate_of_ex_for_diff_of_Tr}
\end{equation}
holds. Moreover, if $J>d+2$ then for each $A\in\MLOHS[2][{\H[-J]}]$
and $n\in\N$
\[
\max_{x\in\T_{n}^{d}}\left|\partial_{n,j}^{\otimes2}\pr_{n}^{\otimes2}A(x,x)-\pr_{n}\Tr\left(\partial_{j}^{\otimes2}A\right)(x)\right|\le\frac{C_{J}}{n}\norm A_{\MLOHS[2][{\H[-J]}]}.
\]
\end{prop}

\begin{proof}
We first prove the estimate (\ref{eq:estimate_of_ex_for_diff_of_Tr}).
Setting 
\[
R_{n}(x):=\Tr\partial_{n,j}^{\otimes2}\pr_{n}^{\otimes2}A(x)=\partial_{n,j}^{\otimes2}\pr_{n}^{\otimes2}A(x,x),\quad x\in\Tnd,
\]
 and using (\ref{eq:connection_between_ex_and_varsigma}), we get
\begin{align*}
\norm{\ex_{n}R_{n}}_{H_{I}}^{2} & =\sum_{k\in\Z^{d}}\left(1+|k|^{2}\right)^{I}\left|\inn{\ex_{n}R_{n},\varsigma_{k}}\right|^{2}=\sum_{k\in\Znd}\left(1+|k|^{2}\right)^{I}\left|\inn{R_{n},\varsigma_{k}}_{n}\right|^{2}\\
 & =\sum_{k\in\Znd}\left(1+|k|^{2}\right)^{I}\left|\binn{\Tr\partial_{n,j}^{\otimes2}\pr_{n}^{\otimes2}A,\varsigma_{k}}_{n}\right|^{2}.
\end{align*}
By (\ref{eq:eigenvalues_for_discrete_operators}), we can write for
$x\in\Tnd$
\begin{align*}
\Tr\partial_{n,j}^{\otimes2}\pr_{n}^{\otimes2}A(x) & =\sum_{l,\tilde{l}\in\Znd}A[\varsigma_{-l},\varsigma_{-\tilde{l}}]\left(\partial_{n,j}\varsigma_{l}(x)\partial_{n,j}\varsigma_{\tilde{l}}(x)\right)\\
 & =\sum_{l,\tilde{l}\in\Znd}\mu_{l,j}\mu_{\tilde{l},j}A[\varsigma_{-l},\varsigma_{-\tilde{l}}]\varsigma_{l+\tilde{l}}(x)
\end{align*}
and thus, using the periodicity of $\varsigma_{l+\tilde{l}}$ on $\Tnd$,
we estimate
\begin{align}
 & \left|\binn{\Tr\partial_{n,j}^{\otimes2}\pr_{n}^{\otimes2}A,\varsigma_{k}}_{n}\right|^{2}=\left|\sum_{l,\tilde{l}\in\Znd}\mu_{l,j}\mu_{\tilde{l},j}A[\varsigma_{l},\varsigma_{\tilde{l}}]\inn{\varsigma_{l+\tilde{l}},\varsigma_{k}}_{n}\right|^{2}\nonumber \\
 & \qquad\le\left|\sum_{l,\tilde{l}\in\Znd}\mu_{l,j}\mu_{\tilde{l},j}A[\varsigma_{l},\varsigma_{\tilde{l}}]\I_{\{l+\tilde{l}=k\mod (2n+1)\}}\right|^{2}\nonumber \\
 & \qquad\le\norm A_{\MLOHS[2][{\H[-J]}]}^{2}\sum_{l,\tilde{l}\in\Znd}\frac{|l|^{2}|\tilde{l}|^{2}}{(1+|l|^{2})^{J}(1+|\tilde{l}|^{2})^{J}}\I_{\{l+\tilde{l}=k\mod (2n+1)\}}\label{eq:estimate_of_discrete_tr_of_pr}\\
 & \qquad\le\norm A_{\MLOHS[2][{\H[-J]}]}^{2}\sum_{l,\tilde{l}\in\Znd}\frac{1}{(1+|l|^{2})^{J-1}(1+|\tilde{l}|^{2})^{J-1}}\I_{\{l+\tilde{l}=k\mod (2n+1)\}}.\nonumber 
\end{align}
Combining the estimates together and using the fact that $|k|\I_{\{l+\tilde{l}=k\mod (2n+1)\}}\le|l+\tilde{l}|\I_{\{l+\tilde{l}=k\mod (2n+1)\}}$,
we obtain
\begin{align*}
\norm{\ex_{n}R_{n}}_{H_{I}}^{2} & \le\norm A_{\MLOHS[2][{\H[-J]}]}^{2}\sum_{k\in\Znd}\left(1+|k|^{2}\right)^{I}\\
 & \qquad\qquad\qquad\cdot\sum_{l,\tilde{l}\in\Znd}\frac{1}{(1+|l|^{2})^{J-1}(1+|\tilde{l}|^{2})^{J-1}}\I_{\{l+\tilde{l}=k\mod (2n+1)\}}\\
 & \le\norm A_{\MLOHS[2][{\H[-J]}]}^{2}\sum_{k\in\Znd}\sum_{l,\tilde{l}\in\Znd}\frac{\left(1+|l+\tilde{l}|^{2}\right)^{I}}{(1+|l|^{2})^{J-1}(1+|\tilde{l}|^{2})^{J-1}}\I_{\{l+\tilde{l}=k\mod (2n+1)\}}\\
 & =3^{d}\norm A_{\MLOHS[2][{\H[-J]}]}^{2}\sum_{l,\tilde{l}\in\Znd}\frac{\left(1+|l+\tilde{l}|^{2}\right)^{I}}{(1+|l|^{2})^{J-1}(1+|\tilde{l}|^{2})^{J-1}}<C_{J,I}\norm A_{\MLOHS[2][{\H[-J]}]}^{2},
\end{align*}
since $J-1-I>\frac{d}{2}$. 

To get the second part of the statement, we first use (\ref{eq:eigenvalues_for_discrete_operators})
and the Cauchy-Schwarz inequality to estimate for each $x,y\in\T_{n}^{d}$
\begin{align*}
 & \left|\partial_{n,j}^{\otimes2}\pr_{n}^{\otimes2}A(x,y)-\partial_{j}^{\otimes2}\pr_{n}^{\otimes2}A(x,y)\right|^{2}\\
 & \qquad=\left|\sum_{k,l\in\Znd}A[\varsigma_{-k},\varsigma_{-l}]\partial_{n,j}\varsigma_{k}(x)\partial_{n,j}\varsigma_{l}(y)-\sum_{k,l\in\Znd}A[\varsigma_{-k},\varsigma_{-l}]\partial_{j}\varsigma_{k}(x)\partial_{j}\varsigma_{l}(y)\right|^{2}\\
 & \qquad\le\left[\sum_{k,l\in\Z_{n}^{d}}|A[\varsigma_{-k},\varsigma_{-l}]|\left|\mu_{k,j}^{n}\mu_{l,j}^{n}+k_{j}l_{j}\right|\right]^{2}\\
 & \qquad\le\sum_{k,l\in\Znd}|A[\varsigma_{k},\varsigma_{l}]|^{2}(1+|k|^{2})^{J}(1+|l|^{2})^{J}\\
 & \qquad\qquad\qquad\cdot\sum_{k,l\in\Znd}(1+|k|^{2})^{-J}(1+|l|^{2})^{-J}\left|\mu_{k,j}^{n}\mu_{l,j}^{n}+k_{j}l_{j}\right|^{2}\\
 & \qquad\le2\norm A_{\MLOHS[2][{\H[-J]}]}^{2}\sum_{k,l\in\Znd}(1+|k|^{2})^{-J}(1+|l|^{2})^{-J}\\
 & \qquad\qquad\qquad\cdot\left[|\mu_{l,j}^{n}|^{2}\left|\mu_{k,j}^{n}-\i k_{j}\right|^{2}+|k_{j}|^{2}|\mu_{l,j}^{n}-\i l_{j}|^{2}\right].
\end{align*}
By Lemma \ref{lem:estimate_of_discrete_eigenvalues} and Taylor's
expansion 
\[
\mu_{k,j}^{n}=\i k_{j}-\frac{k_{j}^{2}}{n}\theta_{n,j}
\]
for some $\theta_{j,n}\in\C$ such that $|\theta_{j,n}|\le1$, we
can continue the estimate as follow
\begin{align*}
 & \frac{2}{n^{2}}\norm A_{\MLOHS[2][{\H[-J]}]}^{2}\sum_{k,l\in\Znd}(1+|k|^{2})^{-J}(1+|l|^{2})^{-J}\left[l_{j}^{2}k_{j}^{4}+k_{j}^{2}l_{j}^{4}\right]\\
 & \qquad\le\frac{4}{n^{2}}\norm A_{\MLOHS[2][{\H[-J]}]}^{2}\sum_{k,l\in\Znd}(1+|k|^{2})^{-J+2}(1+|l|^{2})^{-J+2}\\
 & \qquad\le\frac{C_{J}}{n^{2}}\norm A_{\MLOHS[2][{\H[-J]}]}^{2}
\end{align*}
since $J>\frac{d}{2}+2.$ Thus,
\begin{equation}
\max_{x\in\Tnd}\left|\partial_{n,j}^{\otimes2}\pr_{n}^{\otimes2}A(x,x)-\partial_{j}^{\otimes2}\pr_{n}^{\otimes2}A(x,x)^{2}\right|\le\frac{C_{J}}{n}\norm A_{\MLOHS[2][{\H[-J]}]}^{2}.\label{eq:first_estimate_for_Tr}
\end{equation}

We next compute the norm
\begin{align}
 & \bnorm{K_{\partial_{j}^{\otimes2}\pr_{n}^{\otimes2}A}-\partial_{j}^{\otimes2}A}_{\MLOHS[2][{\H[-J+2]}]}^{2}\nonumber \\
 & \qquad=\sum_{k,l\in\Z^{d}}(1+|k|^{2})^{J-2}(1+|l|^{2})^{J-2}\left|K_{\partial_{j}^{\otimes2}\pr_{n}^{\otimes2}A}[\varsigma_{k},\varsigma_{l}]-\partial_{j}^{\otimes2}A[\varsigma_{k},\varsigma_{l}]\right|^{2}\label{eq:estimate_of_difference_K_and_A}\\
 & \qquad=\sum_{k,l\in\Z^{d}}(1+|k|^{2})^{J-2}(1+|l|^{2})^{J-2}\left|\binn{\varsigma_{k}\otimes\varsigma_{l},\partial_{j}^{\otimes2}\pr_{n}^{\otimes2}A}-A[\partial_{j}\varsigma_{k},\partial_{j}\varsigma_{l}]\right|^{2}.\nonumber 
\end{align}
Using the definition of $\pr_{n}^{\otimes2}$ and (\ref{eq:eigenvalues_for_discrete_operators}),
we continue the equality as follows
\begin{align*}
 & \sum_{k,l\in\Z^{d}}(1+|k|^{2})^{J-2}(1+|l|^{2})^{J-2}\left|\binn{\partial_{j}\varsigma_{k}\otimes\partial_{j}\varsigma_{l},\pr_{n}^{\otimes2}A}-A[\partial_{j}\varsigma_{k},\partial_{j}\varsigma_{l}]\right|^{2}\\
 & \qquad=\sum_{k,l\not\in\Znd}(1+|k|^{2})^{J-2}(1+|l|^{2})^{J-2}k_{j}^{2}l_{j}^{2}\left|A[\varsigma_{k},\varsigma_{l}]\right|^{2}\\
 & \qquad\le\sum_{k,l\not\in\Znd}(1+|k|^{2})^{J-1}(1+|l|^{2})^{J-1}\left|A[\varsigma_{k},\varsigma_{l}]\right|^{2}\le\frac{1}{n^{2}}\norm A_{\MLOHS[2][{\H[-J]}]}^{2}.
\end{align*}

Now, taking $\tilde{I}\in\left(\frac{d}{2},J-2-d/2\right)$ and applying
Lemmas \ref{lem:basic_properties_of_pr}, \ref{lem:tr_operator},
\ref{lem:norm_of_derivative_of_multilinear_operator} and (\ref{eq:estimate_of_difference_K_and_A}),
we obtain 
\begin{align}
 & \bnorm{\Tr K_{\partial_{j}^{\otimes2}\pr_{n}^{\otimes2}A}-\pr_{n}\Tr\partial_{j}^{\otimes2}A}_{\Cf}\le\bnorm{\Tr K_{\partial_{j}^{\otimes2}\pr_{n}^{\otimes2}A}-\Tr\partial_{j}^{\otimes2}A}_{\Cf}\nonumber \\
 & \qquad+\bnorm{\Tr\partial_{j}^{\otimes2}A-\pr_{n}\Tr\partial_{j}^{\otimes2}A}_{\Cf}\nonumber \\
 & \qquad\le C_{J}\bnorm{\Tr K_{\partial_{j}^{\otimes2}\pr_{n}^{\otimes2}A}-\Tr\partial_{j}^{\otimes2}A}_{H_{\tilde{I}}}+\bnorm{\Tr\partial_{j}^{\otimes2}A-\pr_{n}\Tr\partial_{j}^{\otimes2}A}_{\H[\tilde{I}]}\nonumber \\
 & \qquad\le C_{J,\tilde{I}}\bnorm{K_{\partial_{j}^{\otimes2}\pr_{n}^{\otimes2}A}-\partial_{j}^{\otimes2}A}_{\MLOHS[2][{\H[-J+2]}]}+\frac{C_{\tilde{I}}}{n}\bnorm{\Tr\partial_{j}^{\otimes2}A}_{\H[\tilde{I}+1]}\label{eq:second_estimate_for_Tr}\\
 & \qquad\le\frac{C_{J,\tilde{I}}}{n}\norm A_{\MLOHS[2][{\H[-J]}]}+\frac{C_{\tilde{I}}}{n}\norm{\partial_{j}^{\otimes2}A}_{\MLOHS[2][{\H[-J+1]}]}\nonumber \\
 & \qquad\le\frac{C_{J,\tilde{I}}}{n}\norm A_{\MLOHS[2][{\H[-J]}]}.\nonumber 
\end{align}
Thus, 
\begin{align*}
\max_{x\in\T_{n}^{d}}\Big|\partial_{n,j}^{\otimes2}\pr_{n}^{\otimes2}A(x,x) & -\pr_{n}\Tr\partial_{j}^{\otimes2}A(x)\Big|\le\max_{x\in\T_{n}^{d}}\left|\partial_{n,j}^{\otimes2}\pr_{n}^{\otimes2}A(x,x)-\partial_{j}^{\otimes2}\pr_{n}^{\otimes2}A(x,x)\right|\\
 & +\norm{\Tr K_{\partial_{j}^{\otimes2}\pr_{n}^{\otimes2}A}-\pr_{n}\Tr\partial_{j}^{\otimes2}A}_{\Cf}\le\frac{C}{n}\norm A_{\MLOHS[2][{\H[-J]}]},
\end{align*}
due to the triangle inequality the estimates (\ref{eq:first_estimate_for_Tr}),
(\ref{eq:second_estimate_for_Tr}) and the fact that $\Tr K_{\partial_{j}^{\otimes2}\pr_{n}^{\otimes2}A}=\partial_{j}^{\otimes2}\pr_{n}^{\otimes2}A(x,x)$,
$x\in\T^{d}$ (see Lemma \ref{lem:tr_operator}). This completes the
proof of the proposition.
\end{proof}
For $A\in\MLOHS[2][{\H[-J]}]$ and $B\in\MLOHS[2][{\H[J]}]$ we define
\[
A:B:=\sum_{k,l\in\Z^{d}}A[\tilde{\varsigma}_{k},\tilde{\varsigma}_{l}]B[\tilde{\varsigma}_{k},\tilde{\varsigma}_{l}],
\]
The series above absolutely converges and
\begin{align}
|A:B|^{2} & \le\left(\sum_{k,l\in\Z^{d}}\left|A[\tilde{\varsigma}_{k},\tilde{\varsigma}_{l}]B[\tilde{\varsigma}_{k},\tilde{\varsigma}_{l}]\right|\right)^{2}\nonumber \\
 & =\sum_{k,l\in\Z^{d}}(1+|k|^{2})^{-J}(1+|l|^{2})^{-J}\left|A[\tilde{\varsigma}_{k},\tilde{\varsigma}_{l}]\right|^{2}\label{eq:norm_of_:}\\
 & \qquad\qquad\qquad\cdot\sum_{k,l\in\Z^{d}}(1+|k|^{2})^{J}(1+|l|^{2})^{J}\left|B[\tilde{\varsigma}_{k},\tilde{\varsigma}_{l}]\right|^{2}\nonumber \\
 & =\norm A_{\MLOHS[2][{\H[-J]}]}^{2}\norm B_{\MLOHS[2][{\H[J]}]}^{2},\nonumber 
\end{align}
where we Lemma (\ref{eq:SPDE_for_OU_process-1}) and Hölder's inequality. 

\subsection{Differentiable functions on $\protect\H[J]$\protect\label{subsec:Differentiable-functions-on-H}}

In this section, we will investigate some differential properties
of functions defined on Sobolev spaces.

Recall that for $J<I$ we have $\H[I]\subset\H[J].$ Abusing notation,
the restrictions of a function $F:\H[J]^{k}\to\R$ to $\H[I]^{k}$
will be denoted also by $F$ for $k\in\N$. Let $(E,\norm{\cdot}_{E})$
denote a Banach space. The following statement directly follows from
the inequality $\norm{\cdot}_{\H[J]}\le\norm{\cdot}_{\H[I]}$. 
\begin{lem}
\label{lem:restriction_of_multilinear_operator}For each $J,I\in\R$,
$J<I,$ and $k\in\N$ the space $\MLO[k][{\H[J];E}]$ is continuously
embedded into $\MLO[k][{\H[I];E}]$ and 
\[
\norm A_{\cL_{k}(H_{I};E)}\le\norm A_{\cL_{k}(H_{J};E)}
\]
for all $A\in\MLO[k][{\H[J];E}]$.
\end{lem}

Recall that a function $F:H_{J}\to E$ is differentiable\footnote{see \citep[Definition on p. 25]{Cartan:1971}}
at $f\in H_{J}$ if it is continuous at $f$ and there exists a bounded
linear map $DF(f)$ from $\H[J]$ to $E$ such that 
\[
F(f+h)=F(f)+DF(f)[h]+o(\norm h_{\H[J]})
\]
as $h\to0$. Following \citep[Section 5]{Cartan:1971}, we defined
the $k$-th derivative $D^{k}F(f)$ of a map $F:\H[J]\to\R$ at $f\in\H[J]$
as an element of $\MLO[k][{\H[J]}]$ that is identified with the derivative
of $D^{k-1}F:\H[J]\to\MLO[k-1][{\H[J]}]$ at $f$. Considering a differentiable
function defined on a Sobolev space $\H[J]$, we often consider its
restriction to a smaller Sobolev space $\H[I]$ with $J<I$. The following
statement guarantees the preservation of the differentiability. To
point out that $D^{k}F$ is the derivative of $F$ with respect to
the topology of the space $\H[J]$ in the next statements, we will
write $D_{J}^{k}F$ instead.
\begin{lem}
\label{lem:derivative_in_different_sobolev_spaces}Let $J,I\in\R$,
$J<I$, and $F:\H[J]\to E$ be a differentiable function at $f\in\H[I]$
in the space $\H[J]$. Then $F$ is differentiable at $f$ in the
topology of the space $\H[I]$, $D_{I}F(f)$ coincide with the restriction
of $D_{J}F(f)$ to $\H[I]$ and
\[
\norm{D_{I}F(f)}_{\MLO[1][{\H[I];E}]}\le\norm{D_{J}F(f)}_{\MLO[1][{\H[J];E}]}.
\]
\end{lem}

\begin{proof}
The continuity of $F$ at $f$ in the space $\H[I]$ trivially follows
from the continuous embedding of $\H[I]$ into $\H[J]$. According
to Lemma \ref{lem:restriction_of_multilinear_operator}, $D_{J}F(f)\in\MLO[k][{\H[I];E}]$
and 
\[
\norm{D_{J}F(f)}_{\cL_{k}(H_{I};E)}\le\norm{D_{J}F(f)}_{\cL_{k}(H_{J};E)}.
\]
 We have only to show that $D_{I}F(f)[h]=D_{J}F(f)[h]$, $h\in\H[I]$.
Using the differentiability of $F$ in $\H[J]$ at $f$ and the fact
that $\normH[J][\cdot]\le\normH[I][\cdot]$, we get 
\begin{align*}
F(f+h) & =F(f)+D_{J}(f)[h]+o(\normH[J][h])\\
 & =F(f)+D_{J}(f)[h]+o(\normH[I][h]).
\end{align*}
This completes the proof of the lemma.
\end{proof}
The following corollary is the direct consequence of Lemma \ref{lem:derivative_in_different_sobolev_spaces}.
\begin{cor}
For each $J,I\in\R$, $J<I,$ and $m\in\N_{0}$ the space $\Cf^{m}(\H[J])$
is a subset of $\Cf^{m}(\H[I])$. Moreover, for each $F\in\Cf^{m}(\H[J])$,
$k\in[m]$, $f\in\H[I]$ the derivatives $D_{J}^{k}F(f)$ and $D_{I}^{k}F(f)$
coincide on $\H[I]^{k}$ and 
\[
\norm{D_{I}^{k}F(f)}_{\cL_{k}(H_{I})}\le\norm{D_{J}^{k}F(f)}_{\cL_{k}(H_{J})}.
\]
\end{cor}

Our further goal will be to investigate the differentiability of $F\circ\ex_{n}:\Ln\to\R$
for $F\in\Cf^{m}(\H[J])$. Using the fact that $\ex_{n}:\Ln\to\H[J]$
is a continuous linear operator, it is continuously differentiable
with
\[
D\ex_{n}(f)[h]=\ex_{n}h
\]
for each $f,h\in\Ln$.
\begin{lem}
\label{lem:differentiability_of_F_ex_n}Let $F\in\Cf^{1}(\H)$ for
some $m\in\N$ and $J\in\R$. Then the function $F\circ\ex_{n}$ belongs
to $\Cf^{1}(\Ln)$ and 
\[
D(F\circ\ex_{n})=\pr_{n}DF\circ\ex_{n}.
\]
\end{lem}

\begin{proof}
The differentiability of $F\circ\ex_{n}$ follows from \citep[Theorem 2.2.1]{Cartan:1971}
and the differentiability of $F:\H[J]\to\R$ and $\ex_{n}:\Ln\to\H[J]$.
We will only compute the derivative of $D(F\circ\ex_{n})$. Taking
$f,h\in\Ln$ and using the chain rule and (\ref{eq:connection_betwwen_ex_and_pr}),
we compute
\begin{align*}
D(F\circ\ex_{n})(f)[h] & =DF(\ex_{n}f)\left[D\ex_{n}(f)[h]\right]=DF(\ex_{n}f)[\ex_{n}h]\\
 & =\inn{DF(\ex_{n}f),\ex_{n}h}=\inn{\pr_{n}(DF)(\ex_{n}f),h}_{n}.
\end{align*}
This completes the proof of the statement.
\end{proof}

\section{Some additional facts and proofs\protect\label{subsec:Some-additional-facts}}

We recall that $\Her_{n}$ denotes the Hilbert space of symmetric
matrices $A=(A_{k,l})_{k,l\in\Znd}$ with real-valued entries equipped
with the inner product 
\[
A:B=\sum_{k,l\in\Znd}A_{k,l}B_{k,l}.
\]
An open subset of positively defined matrices from $\Her_{n}$ is
denoted by $\Her_{n}^{+}.$
\begin{lem}
\label{lem:integration-by-parts_for_gauss}Let $A=(A_{k,l})_{k,l\in\Znd}\in\Her_{n}^{+}$
and $B\in\Her_{n}$. Then for a standard Gaussian vector $\zeta$
in $\R^{\Znd}$ and $f\in\Cf_{l}^{2}(\R^{\Znd})$ the integration-by-parts
formula 
\[
\E\left[Df(A\zeta)\cdot(B\zeta)\right]=\E\left[D^{2}f(A\zeta):\left(BA\right)\right]
\]
holds.
\end{lem}

\begin{proof}
Setting $R:=BA^{-1}$, $\eta:=A\zeta$, and using the integration-by-parts
formula, we get
\begin{align*}
\E\left[Df\left(A\zeta\right)\cdot\left(B\zeta\right)\right] & =\E\left[Df\left(A\zeta\right)\cdot\left(BA^{-1}A\zeta\right)\right]\\
 & =\sum_{k,l\in\Znd}\E\left[\frac{\partial f}{\partial x_{k}}(\eta)R_{k,l}\eta_{l}\right]\\
 & =\sum_{k,l,\tilde{k}\in\Znd}\E\left[\frac{\partial^{2}f}{\partial x_{k}\partial x_{\tilde{k}}}(\eta)R_{k,l}\cov(\eta_{l},\eta_{\tilde{k}})\right]\\
 & =\sum_{k,l,\tilde{k}\in Z_{n}^{d}}\E\left[\frac{\partial^{2}f}{\partial x_{k}\partial x_{\tilde{k}}}(\eta)R_{k,l}(A^{2})_{l,\tilde{k}}\right]\\
 & =\sum_{k,\tilde{k}\in Z_{n}^{d}}\E\left[\frac{\partial^{2}f}{\partial x_{k}\partial x_{\tilde{k}}}(\eta)(RA^{2})_{k,\tilde{k}}\right]\\
 & =\E\left[D^{2}f(A\zeta):(BA)\right].
\end{align*}
This completes the proof of the lemma. 
\end{proof}
For completeness of the presentation, we next provide the estimate
of the term 
\[
I:=\frac{1}{(2n+1)^{d}}\sum_{x\not=y\in\Tnd}\E\left[\left(\eta_{t}^{n}(x)-\rho_{t}^{n}(x)\right)\left(\eta_{t}^{n}(y)-\rho_{t}^{n}(y)\right)\right]\varphi_{n}(x)\varphi_{n}(y)
\]
in the proof of Lemma \citep[p. 32]{Ravishankar:1992}, following
the proof of the main theorem in \citep[p. 32]{Ravishankar:1992}. 
\begin{proof}[Proof. (Estimate of off-diagonal sum in the proof of Lemma \ref{lem:estimate_of_inner_product_of_eta_with_varphi}) ]
 Set for $x,y\in\Tnd$
\begin{align*}
V(t,x,y) & :=\frac{1}{(2n+1)^{d}}\E(\eta_{t}^{n}(x)-\rho_{t}^{n}(x))(\eta_{t}^{n}(y)-\rho_{t}^{n}(y))\\
 & =\E\eta_{t}^{n}(x)\eta_{t}^{n}(y)-\rho_{t}^{n}(x)\rho_{t}^{n}(y),
\end{align*}
where the letter equality follows from the definition of $\rho_{t}^{n}$.
Applying the generator $2\pi^{2}\Delta_{n}\otimes\cG_{n}^{EP}$ of
the Markov process $\left(\rho_{t}^{n},\eta_{t}^{n}\right)$, $t\ge0$,
to the function $G(x,y;\eta,\rho)=\eta(x)\eta(y)-\rho(x)\rho(y)$
for fixed $x,y\in\Tnd$, $x\not=y$, we get 
\[
\left(2\pi^{2}\Delta_{n}\otimes\cG_{n}^{EP}\right)G(x,y;\cdot,\cdot)(\eta,\rho)=\cG_{n}^{EP}\left[\eta(x)\eta(y)\right]-2\pi^{2}\Delta_{n}\left[\rho(x)\rho(y)\right].
\]
Now we separately rewrite
\begin{align*}
\cG_{n}^{EP}\left[\eta(x)\eta(y)\right] & =\frac{(2n+1)^{2}}{2}\sum_{j=1}^{d}\sum_{z\in\T_{n}^{d}}\left(\eta^{z\leftrightarrow z+e_{j}}(x)\eta^{z\leftrightarrow z+e_{j}}(y)-\eta(x)\eta(y)\right)\\
 & =\frac{(2n+1)^{2}}{2}\sum_{e\in E_{1}}\left(\eta(x+e)\eta(y)-\eta(x)\eta(y)\right)\left(1-\I_{\left\{ x+e=y\right\} }\right)\\
 & +\frac{(2n+1)^{2}}{2}\sum_{e\in E_{1}}\left(\eta(x)\eta(y-e)-\eta(x)\eta(y)\right)\left(1-\I_{\left\{ x+e=y\right\} }\right),
\end{align*}
where the summation is taken over $E_{1}:=\left\{ \pm e_{j},j\in[d]\right\} $.
We also note that 
\begin{align*}
2\pi^{2}\Delta_{n}\left[\rho(x)\rho(y)\right] & =\frac{(2n+1)^{2}}{2}\sum_{j=1}^{d}\left(\rho(x+e_{j})\rho(y)+\rho(x-e_{j})\rho(y)-2\rho(x)\rho(y)\right)\\
 & +\frac{(2n+1)^{2}}{2}\sum_{j=1}^{d}\left(\rho(x)\rho(y+e_{j})+\rho(x)\rho(y-e_{j})-2\rho(x)\rho(y)\right)\\
 & =\frac{(2n+1)^{2}}{2}\sum_{e\in E_{1}}\left(\rho(x+e)\rho(y)-\rho(x)\rho(y)\right)\\
 & +\frac{(2n+1)^{2}}{2}\sum_{e\in E_{1}}\left(\rho(x)\rho(y-e)-\rho(x)\rho(y)\right).
\end{align*}
Hence 

\begin{align*}
 & \left(2\pi^{2}\Delta_{n}\otimes\cG_{n}^{EP}\right)G(x,y;\eta,\rho)\\
 & \qquad=\frac{(2n+1)^{2}}{2}\sum_{e\in E_{1}}\left(\eta(x+e)\eta(y)-\eta(x)\eta(y)\right)\left(1-\I_{\left\{ x+e=y\right\} }\right)\\
 & \qquad+\frac{(2n+1)^{2}}{2}\sum_{e\in E_{1}}\left(\eta(x)\eta(y-e)-\eta(x)\eta(y)\right)\left(1-\I_{\left\{ x+e=y\right\} }\right)\\
 & \qquad-\frac{(2n+1)^{2}}{2}\sum_{e\in E_{1}}\left(\rho(x+e)\rho(y)-\rho(x)\rho(y)\right)\\
 & \qquad-\frac{(2n+1)^{2}}{2}\sum_{e\in E_{1}}\left(\rho(x)\rho(y-e)-\rho(x)\rho(y)\right)\\
 & \qquad=\frac{(2n+1)^{2}}{2}\sum_{e\in E_{1}}\left(G(x+e,y;\eta,\rho)-G(x,y;\eta,\rho)\right)\left(1-\I_{\left\{ x+e=y\right\} }\right)\\
 & \qquad+\frac{(2n+1)^{2}}{2}\sum_{e\in E_{1}}\left(G(x,y-e;\eta,\rho)-G(x,y;\eta,\rho)\right)\left(1-\I_{\left\{ x+e=y\right\} }\right)\\
 & \qquad-\frac{(2n+1)^{2}}{2}\sum_{e\in E_{1}}\left(\rho(x)-\rho(y)\right)^{2}\I_{\left\{ x+e=y\right\} }.
\end{align*}
This implies that the function
\[
V(t,x,y)=\E\left[\eta_{t}^{n}(x)\eta_{t}^{n}(y)\right]-\rho_{t}^{n}(x)\rho_{t}^{n}(y)
\]
is a solution to the following differential equation
\[
\frac{d}{dt}V(t,x,y)=\cL V(t,x,y)-\frac{(2n+1)^{2}}{2}\sum_{e\in E_{1}}\left(\rho_{t}^{n}(x)-\rho_{t}^{n}(y)\right)^{2}\I_{\left\{ x+e=y\right\} },
\]
where
\begin{align*}
\cL V(x,y) & =\frac{(2n+1)^{2}}{2}\sum_{e\in E_{1}}\left(V(x+e,y)-V(x,y)\right)\left(1-\I_{\left\{ x+e=y\right\} }\right)\\
 & +\frac{(2n+1)^{2}}{2}\sum_{e\in E_{1}}\left(V(x,y-e)-V(x,y)\right)\left(1-\I_{\left\{ x+e=y\right\} }\right).
\end{align*}

Note that $\cL$ is the generator of the process $\{X_{t},Y_{t}\}$,
$t\ge0$, on $\T_{n}^{d}\times\T_{n}^{d}$ that evolves as an exclusion
process with two particles. Let $P_{t}(x,y;u,v)$, $u,v\in\Tnd$,
be its semigroup. Then
\begin{align*}
V(t,x,y) & =P_{t}V(0,x,y)\\
 & -\frac{(2n+1)^{2}}{2}\int_{0}^{t}\sum_{u,v\in\T_{n}^{d}}\sum_{e\in E_{1}}P_{t-s}(x,y;u,v)\left(\rho_{s}^{n}(u)-\rho_{s}^{n}(v)\right)^{2}\I_{\left\{ u+e=v\right\} }ds\\
 & =P_{t}V(0,x,y)\\
 & -\frac{(2n+1)^{2}}{2}\int_{0}^{t}\sum_{e\in E_{1}}\sum_{u\in\T_{n}^{d}}P_{t-s}(x,y;u,u+e)\left(\rho_{s}^{n}(u)-\rho_{s}^{n}(u+e)\right)^{2}ds.
\end{align*}
Due to the independents of $\eta_{0}^{n}(x)$ and $\eta_{0}^{n}(y)$
we conclude $V(0,x,y)=0.$ Therefore, $P_{t}V(0,x,y)=0.$ Thus,
\[
V(t,x,y)=-\frac{(2n+1)^{2}}{2}\int_{0}^{t}\sum_{e\in E_{1}}\sum_{u\in\T_{n}^{d}}P_{t-s}(x,y;u,u+e)\left(\rho_{s}^{n}(u)-\rho_{s}^{n}(u+e)\right)^{2}ds.
\]

Consequently, we can estimate
\begin{align*}
|I| & \le\frac{(2n+1)^{2}}{2}\frac{1}{(2n+1)^{d}}\sum_{e\in E_{1}}\sum_{u\in\T_{n}^{d}}\sum_{x\not=y}|\varphi_{n}(x)||\varphi_{n}(y)|\\
 & \qquad\qquad\qquad\cdot\int_{0}^{t}P_{t-s}(x,y;u,u+e)\left(\rho_{s}^{n}(u)-\rho_{s}^{n}(u+e)\right)^{2}ds\\
 & \le\frac{2\pi^{2}}{(2n+1)^{d}}\sup_{s\in[0,t]}\max_{u\in\Tnd}\left|\nabla_{n}\rho_{s}^{n}(u)\right|^{2}\\
 & \qquad\qquad\qquad\cdot\sum_{e\in E_{1}}\sum_{u\in\T_{n}^{d}}\sum_{x\not=y}|\varphi_{n}(x)||\varphi_{n}(y)|\int_{0}^{t}P_{t-s}(x,y;u,u+e)ds.
\end{align*}
Using the duality of the SSEP, we get for each $e\in E_{1}$
\begin{align*}
 & \frac{1}{(2n+1)^{d}}\sum_{u\in\T_{n}^{d}}\sum_{x\not=y}|\varphi_{n}(x)||\varphi_{n}(y)|\int_{0}^{t}P_{t-s}(x,y;u,u+e)ds\\
 & \qquad=\frac{1}{(2n+1)^{d}}\sum_{u\in\T_{n}^{d}}\sum_{x\not=y}|\varphi_{n}(x)||\varphi_{n}(y)|\int_{0}^{t}P_{t-s}(u,u+e;x,y)ds\\
 & \qquad\le\frac{\norm{\varphi_{n}}_{n,\Cf}}{(2n+1)^{d}}\sum_{u\in\T_{n}^{d}}\sum_{x\in\Tnd}|\varphi_{n}(x)|\int_{0}^{t}P_{t-s}(u,u+e;x,\Tnd)ds.
\end{align*}
Since 
\[
P_{t-s}(u,u+e;x,\Tnd)=P_{t-s}^{0}(u;x)+P_{t-s}^{0}(u+e;x),
\]
where $P_{t}^{0}$ is the transition kernel for a single particle
executing a random walk in $\Tnd$, we get for $e\in E_{1}$
\begin{align*}
 & \frac{1}{(2n+1)^{d}}\sum_{u\in\T_{n}^{d}}\sum_{x\in\Tnd}|\varphi_{n}(x)|\int_{0}^{t}P_{t-s}(u,u+e;x,\Tnd)ds\\
 & \qquad=\frac{1}{(2n+1)^{d}}\sum_{u\in\T_{n}^{d}}\sum_{x\in\Tnd}|\varphi_{n}(x)|\int_{0}^{t}\left[P_{t-s}^{0}(u;x)+P_{t-s}^{0}(u+e;x)\right]ds\\
 & \qquad=\frac{1}{(2n+1)^{d}}\sum_{u\in\T_{n}^{d}}\sum_{x\in\Tnd}|\varphi_{n}(x)|\int_{0}^{t}\left[P_{t-s}^{0}(x;u)+P_{t-s}^{0}(x;u+e)\right]ds\\
 & \qquad=\frac{t}{(2n+1)^{d}}\sum_{x\in\Tnd}|\varphi_{n}(x)|\le t\norm{\varphi_{n}}_{n,\Cf}.
\end{align*}
Combining the estimates above, we conclude
\[
|I|\le2\pi^{2}\sup_{s\in[0,t]}\max_{u\in\Tnd}\left|\nabla_{n}\rho_{s}^{n}(u)\right|^{2}\norm{\pr_{n}\varphi}_{n,\Cf}^{2}t.
\]
This completes the estimate.
\end{proof}

\end{document}